\documentclass[12pt]{amsart}
\usepackage{a4wide,enumerate,xcolor}
\usepackage{amsmath}
\usepackage{scrextend} % for labeling environment
\allowdisplaybreaks

\let\pa\partial
\let\na\nabla
\let\eps\varepsilon
\newcommand{\N}{{\mathbb N}}
\newcommand{\R}{{\mathbb R}}
\newcommand{\diver}{\operatorname{div}}

\newcommand{\dd}{\mathrm{d}}
\newcommand{\dom}{\mathcal{O}}
\newcommand{\E}{{\mathbb E}}
\newcommand{\F}{{\mathbb F}}
\newcommand{\Prob}{{\mathbb P}}
\newcommand{\DD}{\mathrm{D}}
\renewcommand{\L}{{\mathcal L}}
\newcommand{\simplex}{\mathcal{D}}

\newcommand{\blue}{\textcolor{blue}}

\newtheorem{theorem}{Theorem}
\newtheorem{lemma}[theorem]{Lemma}

\newtheorem{proposition}[theorem]{Proposition}
\newtheorem{remark}[theorem]{Remark}
\newtheorem{corollary}[theorem]{Corollary}
\newtheorem{definition}{Definition}

\begin{document}
\title[Stochastic Shigesada--Kawasaki--Teramoto model]{Global martingale solutions for 
stochastic Shigesada--Kawasaki--Teramoto population models} 

\author[M. Braukhoff]{Marcel Braukhoff}
\address{Mathematisches Institut, Heinrich-Heine-Universit\"at D\"usseldorf, 
Universit\"atsstr. 1, 40225 D\"usseldorf, Germany}
\email{marcel.braukhoff@hhu.de} 

\author[F. Huber]{Florian Huber}
\address{Institute for Analysis and Scientific Computing, Vienna University of  
	Technology, Wiedner Hauptstra\ss e 8--10, 1040 Wien, Austria;
	Department of Statistics and Operations Research, University of Vienna, 
	Kolingasse 14--16, 1090 Vienna, Austria}
\email{fhuber@univie.ac.at}

\author[A. J\"ungel]{Ansgar J\"ungel}
\address{Institute for Analysis and Scientific Computing, Vienna University of  
	Technology, Wiedner Hauptstra\ss e 8--10, 1040 Wien, Austria}
\email{juengel@tuwien.ac.at} 

\date{\today}

\thanks{The authors thank Prof.\ Alexandra Neam\c{t}u 
(Konstanz, Germany) for her constructive help in the preparation of this manuscript.
The second and third authors acknowledge partial support from   
the Austrian Science Fund (FWF), grants I3401, P33010, W1245, and F65.
The second author has been supported by the FWF, grant Y1235 of the START program.
This work has received funding from the European 
Research Council (ERC) under the European Union's Horizon 2020 research and 
innovation programme, ERC Advanced Grant no.~101018153.} 

\begin{abstract}
The existence of global nonnegative martingale solutions to cross-diffusion systems 
of Shigesada--Kawasaki--Teramoto type with multiplicative noise is proven. 
The model describes the stochastic segregation dynamics of an arbitrary number 
of population species in a bounded domain with no-flux boundary conditions. 
The diffusion matrix is generally 
neither symmetric nor positive semidefinite, which excludes standard methods
for evolution equations.
Instead, the existence proof is based on the entropy structure of the model,
a novel regularization of the entropy variable, higher-order moment estimates,
and fractional time regularity. The regularization technique is generic and is
applied to the population system with self-diffusion in any space dimension
and without self-diffusion in two space dimensions.  
\end{abstract}

% \paragraph{Keywords:}  
 
% \paragraph{AMS classification:}  

\maketitle

%%%%%%%%%%%%%%%%%%%%%%%%%%%%%%%%%%%%%%%%%%%%%%%%%%%%%%%%%%%%%%%%%%%%%%%%%%%%%%%

\section{Introduction}

Shigesada, Kawasaki, and Teramoto (SKT) suggested in their seminal paper \cite{SKT79}
a deterministic cross-diffusion system for two competing species, 
which is able to describe the segregation of the populations. A random
influence of the environment or the lack of knowledge of certain biological parameters
motivate the introduction of noise terms, leading to the stochastic
system for $n$ species with the population density $u_i$ of the $i$th species: 
\begin{equation}\label{1.eq}
  \dd u_i - \diver\bigg(\sum_{j=1}^n A_{ij}(u)\na u_j\bigg) \dd t
	= \sum_{j=1}^n \sigma_{ij}(u)\dd W_j(t)\quad\mbox{in }\dom,\ t>0,\ i=1,\ldots,n,
\end{equation}
with initial and no-flux boundary conditions
\begin{equation}\label{1.bic}
  u_i(0)=u_i^0\quad\mbox{in }\dom, \quad
	\sum_{j=1}^n A_{ij}(u)\na u_j\cdot\nu = 0\quad\mbox{on }\pa\dom,\ t>0,\
	i=1,\ldots,n,
\end{equation}
and diffusion coefficients
\begin{equation}\label{1.SKT}
  A_{ij}(u) = \delta_{ij}\bigg(a_{i0} + \sum_{k=1}^n a_{ik}u_k\bigg) + a_{ij}u_i,
	\quad i,j=1,\ldots,n,
\end{equation}
where $\dom\subset\R^d$ ($d\ge 1$) is a bounded domain, $\nu$ is the
exterior unit normal vector to $\pa\dom$, $(W_1,\ldots,W_n)$ is an
$n$-dimensional cylindrical Wiener process, and $a_{ij}\ge 0$ for $i=1,\ldots,n$,
$j=0,\ldots,n$ are parameters.
The stochastic framework is detailed in Section \ref{sec.main}.

The deterministic analog of \eqref{1.eq}--\eqref{1.SKT} generalizes the 
two-species model of \cite{SKT79} to an arbitrary number of species. 
The deterministic model can be derived rigorously from nonlocal population systems 
\cite{DiMo21,Mou20}, stochastic interacting particle systems \cite{CDHJ21},
and finite-state jump Markov models \cite{BMM21,DDD19}.
The original system in \cite{SKT79} also contains a deterministic environmental
potential and Lotka--Volterra terms, which are neglected here for simplicity.

We call $a_{i0}$ the diffusion coefficients, $a_{ii}$
the self-diffusion coefficients, and $a_{ij}$ for $i\neq j$ the cross-diffusion
coefficients. 
We say that system \eqref{1.eq}-\eqref{1.SKT} is {\em with self-diffusion} if
$a_{i0}\ge 0$, $a_{ii} > 0$ for all $i=1,\ldots,n$, and
{\em without self-diffusion} if $a_{i0}>0$, $a_{ii}=0$ for all $i=1,\ldots,n$.

The aim of this work is to prove the existence of global nonnegative
martingale solutions to system \eqref{1.eq}--\eqref{1.SKT} allowing for large
cross-diffusion coefficients. The existence of a local
pathwise mild solution to \eqref{1.eq}--\eqref{1.SKT} with $n=2$ was shown
in \cite[Theorem 4.3]{KuNe20} under the assumption that the diffusion matrix
is positive definite. Global martingale solutions to a SKT model
with quadratic instead of linear coefficients $A_{ij}(u)$ were found in
\cite{DJZ19}. Besides detailed balance, this result needs a moderate smallness
condition on the cross-diffusion coefficients. 
We prove the existence of global martingale solutions to the SKT model
for general coefficients satisfying detailed balance.
This result seems to be new.

There are two major difficulties in the analysis of system \eqref{1.eq}.
The first difficulty is the fact that the diffusion matrix
associated to \eqref{1.eq} is generally neither symmetric nor positive semidefinite. 
In particular, standard semigroup theory is not applicable. These issues
have been overcome in \cite{ChJu04,ChJu06} in the deterministic case
by revealing a formal gradient-flow or entropy structure. 
The task is to extend this idea to the stochastic setting.

In the deterministic case, usually an implicit Euler time discretization is used
\cite{Jue15}. In the stochastic case, we need an explicit Euler
scheme because of the stochastic It\^o integral, but this excludes entropy estimates. 
An alternative is the Galerkin scheme, which reduces the infinite-dimensional stochastic
system to a finite-dimensional one; see, e.g., the proof of 
\cite[Theorem 4.2.4]{LiRo15}. This is possible only if energy-type ($L^2$)
estimates are available, i.e.\ if $u_i$ can be used as a test function.
In the present case, however, only entropy estimates
are available with the test function $\log u_i$,
which is not an element of the Galerkin space.

In the following, we describe our strategy to overcome these difficulties.
We say that system \eqref{1.eq} has an entropy structure if there exists a 
function $h:[0,\infty)^n\to[0,\infty)$, called an entropy density, 
such that the deterministic analog of \eqref{1.eq} 
can be written in terms of the entropy variables
(or chemical potentials) $w_i=\pa h/\pa u_i$ as
\begin{equation}\label{1.w}
  \pa_t u_i(w) - \diver\bigg(\sum_{j=1}^n B_{ij}(w)\na w_j\bigg) = 0, \quad
	i=1\ldots,n,
\end{equation}
where $w=(w_1,\ldots,w_n)$, $u_i$ is interpreted as a function of $w$, 
and $B(w)=A(u(w))h''(u(w))^{-1}$ 
with $B=(B_{ij})$ is positive semidefinite. 
For the deterministic analog of \eqref{1.eq}, it was 
shown in \cite{CDJ18} that the entropy density is given by
\begin{equation}\label{1.h}
  h(u) = \sum_{i=1}^n\pi_i \big(u_i(\log u_i-1)+1\big), \quad u\in[0,\infty)^n,
\end{equation}
where the numbers $\pi_i>0$ are assumed to satisfy 
$\pi_ia_{ij}=\pi_ja_{ji}$ for all $i,j=1,\ldots,n$.
This condition is the detailed-balance condition for the Markov chain associated
to $(a_{ij})$, and $(\pi_1,\ldots,\pi_n)$ is the corresponding reversible stationary
measure \cite{CDJ18}. 
Using $w_i=\pi_i\log u_i$ in \eqref{1.w} as a test function and 
summing over $i=1,\ldots,n$, a formal computation shows that
\begin{equation}\label{1.ei}
  \frac{\dd}{\dd t}\int_\dom h(u)\dd x
	+ 2\int_\dom\sum_{i=1}^n\pi_i\bigg(2a_{i0}|\na\sqrt{u_i}|^2
	+ 2a_{ii}|\na u_i|^2 + \sum_{j\neq i}a_{ij}|\na\sqrt{u_iu_j}|^2\bigg)\dd x = 0.
\end{equation}
A similar expression holds in the stochastic setting; see \eqref{5.ei}.
It provides $L^2$ estimates for $\na\sqrt{u_i}$ if $a_{i0}>0$ and for
$\na u_i$ if $a_{ii}>0$. Moreover, having proved the existence of a solution $w$
to an approximate version of \eqref{1.eq} leads to the positivity
of $u_i(w)=\exp(w_i/\pi_i)$ (and nonnegativity after passing to the 
de-regularization limit).

To define the approximate scheme, our idea is to ``regularize'' the entropy
variable $w$. Indeed, instead of the algebraic mapping $w\mapsto u(w)$,
we introduce the mapping $Q_\eps(w)=u(w) + \eps L^*Lw$, where 
$L:D(L)\to H$ with domain $D(L)\subset H$ is a 
suitable operator and $L^*$ its dual; see Section \ref{sec.op} for details.
The operator $L$ is chosen in such a way that all elements of $D(L)$ are 
bounded functions, implying that $u(w)$ is well defined.
Introducing the regularization operator $R_\eps:D(L)'\to D(L)$ as the inverse
of $Q_\eps:D(L)\to D(L)'$, the approximate scheme to \eqref{1.eq} is defined,
written in compact form, as
\begin{equation}\label{1.approx}
  \dd v(t) = \diver\big(B(R_\eps(v))\na R_\eps(v)\big)\dd t
	+ \sigma\big(u(R_\eps(v))\big)\dd W(t), \quad t>0.
\end{equation}

The existence of a local weak solution $v^\eps$ to \eqref{1.approx} with suitable
initial and boundary conditions is proved by applying the abstract result of
\cite[Theorem 4.2.4]{LiRo15}; see Theorem \ref{thm.approx}. The entropy
inequality for $w^\eps:=R_\eps(v^\eps)$ and $u^\eps:=u(w^\eps)$,
\begin{align*}
  \E&\sup_{0<t<T\wedge\tau_R}\int_\dom h(u^\eps(t))\dd x
	+ \frac{\eps}{2}\E\sup_{0<t<T\wedge\tau_R}\|Lw^\eps(t))\|_{L^2(\dom)}^2 \\
	&{}+ \E\sup_{0<t<T\wedge\tau_R}\int_0^t\int_\dom
	\na w^\eps(s):B(w^\eps(s))\na w^\eps(s)\dd x\dd s \le C(u^0,T),
\end{align*}
up to some stopping time $\tau_R>0$ allows us to extend the local solution to a 
global one (Proposition \ref{prop.ent}). 

For the de-regularization limit $\eps\to 0$, we need suitable uniform bounds.
The entropy inequality provides gradient bounds for $u_i^\eps$ in the case
with self-diffusion and for $(u_i^\eps)^{1/2}$ in the case without self-diffusion.
Based on these estimates, we use the Gagliardo--Nirenberg inequality to
prove uniform bounds for $u_i^\eps$ in $L^q(0,T;L^q(\dom))$ with $q\ge 2$. 
Such an estimate is crucial to define, for instance, the product $u_i^\eps u_j^\eps$. 
Furthermore, we show a uniform estimate for $u_i^\eps$ in the 
Sobolev--Slobodeckij space $W^{\alpha,p}(0,T;D(L)')$ for some $\alpha<1/2$ and
$p>2$ such that $\alpha p>1$. 
These estimates are needed to prove the tightness of the laws of
$(u^\eps)$ in some sub-Polish space and to conclude strong convergence in $L^2$
thanks to the Skorokhod--Jakubowski theorem. 

For the uniform estimates, we need to distinguish the cases with
and without self-diffusion. In the former case, we
obtain an $L^2(0,T;H^1(\dom))$ estimate for $u_i^\eps$, such that the
product $u_i^\eps\na u_j^\eps$ is integrable, and we can pass to the limit
in the coefficients $A_{ij}(u_i^\eps)$. Without self-diffusion,
we can only conclude that $(u_i^\eps)$ is bounded in $L^2(0,T;W^{1,1}(\dom))$,
and products like $u_i^\eps\na u_j^\eps$ may be not integrable. To overcome this
issue, we use the fact that 
\begin{equation}\label{1.form}
  \diver\bigg(\sum_{j=1}^n A_{ij}(u^\eps)\na u_j^\eps\bigg) 
	= \Delta\bigg(u_i^\eps\bigg(a_{i0} + \sum_{j=1}^n a_{ij} u_j^\eps\bigg)\bigg)
\end{equation}
and write \eqref{1.eq} in a ``very weak'' formulation by applying the Laplace
operator to the test function. Since the bound in $L^2(0,T;W^{1,1}(\dom))$
implies a bound in $L^2(0,T;L^2(\dom))$ bound in two space dimensions, products
like $u_i^\eps u_j^\eps$ are integrable. In the deterministic case, we can
exploit the $L^2$ bound for $\na (u_i^\eps u_j^\eps)^{1/2}$ to find a bound for
$u_i^\eps u_j^\eps$ in $L^1(0,T;L^1(\dom))$ in any space dimension, 
but the limit involves an identification
that we could not extend to the martingale solution concept.

On an informal level, we may state our main result as follows. We refer to 
Section \ref{sec.main} for the precise formulation.

\begin{theorem}[Informal statement]
Let $a_{ij}\ge 0$ satisfy the detailed-balance condition,
let the stochastic diffusion $\sigma_{ij}$ be Lipschitz continuous on the space
of Hilbert--Schmidt operators, and let a certain interaction condition between 
the entropy and stochastic diffusion hold (see Assumption (A5) below).
Then there exists a global nonnegative martingale solution to 
\eqref{1.eq}--\eqref{1.SKT} in the case with self-diffusion in any space dimension
and in the case without self-diffusion in at most two space dimensions.
\end{theorem}

We discuss examples for $\sigma_{ij}(u)$ in Section \ref{sec.noise}.
Here, we only remark that an admissible diffusion term is
\begin{equation}\label{1.sigma}
  \sigma_{ij}(u)=\delta_{ij}u_i^\alpha\sum_{k=1}^\infty a_k(e_k,\cdot)_{U},
	\quad i.j=1,\ldots,n,
\end{equation}
where $1/2\le\alpha\le 1$, $\delta_{ij}$ is the Kronecker symbol,
$a_k\ge 0$ decays sufficiently fast, 
$(e_k)$ is a basis of the Hilbert space $U$ with inner product
$(\cdot,\cdot)_{U}$.

We end this section by giving a brief overview of the state of the art for
the deterministic SKT model. 
First existence results for the two-species model
were proven under restrictive conditions
on the parameters, for instance in one space dimension \cite{Kim84},
for the triangular system with $a_{21}=0$ \cite{LNW98}, or for small
cross-diffusion parameters, since in the latter situation the diffusion matrix
becomes positive definite \cite{Deu87}. Amann \cite{Ama89} proved that a priori 
estimates in the $W^{1,p}(\dom)$ norm with $p>d$ are sufficient to conclude
the global existence of solutions to quasilinear parabolic systems, and
he applied this result to the triangular SKT system. The first global existence
proof without any restriction on the parameters $a_{ij}$ (except nonnegativity)
was achieved in \cite{GGJ03} in one space dimension. This result was generalized
to several space dimensions in \cite{ChJu04,ChJu06} and to the whole space
problem in \cite{Dre08}. SKT-type systems with nonlinear coefficients $A_{ij}(u)$, 
but still for two species, were analyzed in \cite{DLM14,DLMT15}.
Global existence results for SKT-type models with an arbitrary number of species
and under a detailed-balance condition were first proved in \cite{CDJ18} and
later generalized in \cite{LeMo17}. 

This paper is organized as follows. We present our notation and the main
results in Section \ref{sec.main}. The operators needed to define the
approximative scheme are introduced in Section \ref{sec.op}. In Section 
\ref{sec.approx}, the existence of solutions to a general approximative scheme 
is proved and the corresponding entropy inequality is derived.
Theorems \ref{thm.skt} and \ref{thm.skt2} are shown in Sections \ref{sec.skt} and
\ref{sec.skt2}, respectively. Section \ref{sec.noise} is concerned with examples
for $\sigma_{ij}(u)$ satisfying our assumptions.
Finally, the proofs of some auxiliary lemmas are presented in 
Appendix \ref{sec.proofs}, and Appendix \ref{sec.aux} states a tightness criterion
that (slightly) extends \cite[Corollary 2.6]{BrMo14} \blue{to the Banach space setting}.

%%%%%%%%%%%%%%%%%%%%%%%%%%%%%%%%%%%%%%%%%%%%%%%%%%%%%%%%%%%%%%%%%%%%%%%%%%%%%%%

\section{Notation and main result}\label{sec.main}

\subsection{Notation and stochastic framework}

Let $\dom\subset\R^d$ ($d\ge 1$) be a bounded domain. The Lebesgue and Sobolev
spaces are denoted by $L^p(\dom)$ and $W^{k,p}(\dom)$, respectively, where
$p\in[1,\infty]$, $k\in\N$, and $H^k(\dom)=W^{k,2}(\dom)$. 
For notational simplicity, we generally do not distinguish between $W^{k,p}(\dom)$
and $W^{k,p}(\dom;\R^n)$.
We set $H_N^m(\dom) = \{v\in H^m(\dom):\na v\cdot\nu=0$ on $\pa\dom\}$ for $m\ge 2$.
If $u=(u_1,\ldots,u_n)\in X$
is some vector-valued function in the normed space $X$, we write
$\|u\|_X^2=\sum_{i=1}^n\|u_i\|_X^2$. The inner product of a Hilbert space $H$
is denoted by $(\cdot,\cdot)_H$, and $\langle\cdot,\cdot\rangle_{V',V}$ is the dual
product between the Banach space $V$ and its dual $V'$. If $F:U\to V$
is a Fr\'echet differentiable function between Banach spaces $U$ and $V$, 
we write $\DD F[v]:U\to V$ for its Fr\'echet derivative, for any $v\in U$.

Given two quadratic matrices $A=(A_{ij})$, $B=(B_{ij})\in\R^{n\times n}$, 
$A:B=\sum_{i,j=1}^n A_{ij}B_{ij}$ is the Frobenius matrix product,
$\|A\|_F=(A:A)^{1/2}$ the Frobenius norm of $A$, and $\operatorname{tr}A
=\sum_{i=1}^n A_{ii}$ the trace of $A$. The constants $C>0$
in this paper are generic and their values change from line to line.

Let $(\Omega,\mathcal{F},\Prob)$ be a probability space endowed with a complete
right-continuous filtration $\F=(\mathcal{F}_t)_{t\ge 0}$ and let 
$H$ be a Hilbert space.
Then $L^0(\Omega;H)$ consists of all measurable functions from $\Omega$ to $H$, and 
$L^2(\Omega;H)$ consists of all $H$-valued random variables $v$ such that
$\E\|v\|_H^2=\int_\Omega\|v(\omega)\|_H^2\Prob(\dd \omega)<\infty$. 
Let $U$ be a separable Hilbert space 
and $(e_k)_{k\in\N}$ be an orthonormal basis of $U$. 
The space of Hilbert--Schmidt operators from $U$ to $L^2(\dom)$ is defined by
$$
  \L_2(U;L^2(\dom)) = \bigg\{F:U\to L^2(\dom) \mbox{ linear, continuous}:
	\sum_{k=1}^\infty\|Fe_k\|_{L^2(\dom)}^2 < \infty\bigg\},
$$
and it is endowed with the norm $\|F\|_{\L_2(U;L^2(\dom))} 
= (\sum_{k=1}^\infty\|Fe_k\|_{L^2(\dom)}^2)^{1/2}$. 

Let $W=(W_1,\ldots,W_n)$ be an $n$-dimensional $U$-cylindrical Wiener process,
taking values in the separable Hilbert space $U_0\supset U$ 
and adapted to the filtration $\F$. We can write
$W_j=\sum_{k=1}^\infty e_k W_j^k$, where 
$(W_j^k)$ is a sequence of independent standard
one-dimensional Brownian motions \cite[Section 4.1.2]{DaZa14}. 
Then $W_j(\omega)\in C^0([0,\infty);U_0)$ for a.e.\ $\omega$ 
\cite[Section 2.5.1]{LiRo15}.

\subsection{Assumptions}

We impose the following assumptions:

\begin{itemize}
\item[(A1)]  Domain: $\dom\subset\R^d$ ($d\ge 1$) is a bounded domain
with Lipschitz boundary. Let $T>0$ and set $Q_T=\dom\times(0,T)$. 

\item[(A2)] Initial datum: $u^0=(u_1^0,\ldots,u_n^0)\in 
L^\infty(\Omega;L^2(\dom;\R^n))$ is a $\mathcal{F}_0$-measurable 
random variable satisfying $u^0(x)\ge 0$ for a.e.\ $x\in\dom$
$\Prob$-a.s.

\item[(A3)] Diffusion matrix: $a_{ij}\ge 0$ for $i=1,\ldots,n$, $j=0,\ldots,n$ and
there exist $\pi_1,\ldots,\pi_n>0$ such that
$\pi_i a_{ij}=\pi_j a_{ji}$ for all $i,j=1,\ldots,n$ (detailed-balance condition).

\item[(A4)] Multiplicative noise: $\sigma=(\sigma_{ij})$ is an $n\times n$ matrix, 
where $\sigma_{ij}:L^2(\dom;\R^n)\to
\L_2(U;L^2(\dom))$ is $\mathcal{B}(L^2(\dom;\R^n))/$ $\mathcal{B}
(\L_2(U;L^2(\dom)))$-measurable and $\F$-adapted. 
Furthermore, there exists $C_\sigma>0$ such that
for all $u$, $v\in L^2(\dom;\R^n)$,
\begin{align*}
  \|\sigma(u)-\sigma(v)\|_{\L_2(U;L^2(\dom))} 
	&\le C_\sigma\|u-v\|_{L^2(\dom)}, \\
  \|\sigma(v)\|_{\L_2(U;L^2(\dom))} &\le C_\sigma(1+\|v\|_{L^2(\dom)}).
\end{align*}

\item[(A5)] Interaction between entropy and noise: There exists $C_h>0$ such that
for all $u\in L^\infty(\dom\times(0,T))$,
\begin{align*}
  \bigg\{\int_0^t\sum_{k=1}^\infty\sum_{i,j=1}^n
	\bigg(\int_\dom \frac{\pa h}{\pa u_i}(u(s))
	\sigma_{ij}(u(s))e_k\dd x\bigg)^2\dd s\bigg\}^{1/2}
	&\le C_h\bigg(1 + \int_0^t\int_\dom h(u(s))\dd x\dd s\bigg), \\
	\int_0^t\sum_{k=1}^\infty\int_\dom\operatorname{tr}\big[
	(\sigma(u)e_k)^T h''(u)\sigma(u)e_k\big](s)\dd x\dd s 
	&\le C_h\bigg(1 + \int_0^t\int_\dom h(u(s))\dd x\dd s\bigg),
\end{align*}
where $h$ is the entropy density defined in \eqref{1.h}.
\end{itemize}

\begin{remark}[Discussion of the assumptions]\rm
\begin{itemize}
\item[(A1)] The Lipschitz regularity of the boundary $\pa\dom$ is needed to apply
the Sobolev and Gagliardo--Nirenberg inequalities.

\item[(A2)] The regularity condition on $u^0$ can be
weakened to $u^0\in L^p(\Omega;L^2(\dom;\R^n))$ for sufficiently large $p\ge 2$
(only depending on the space dimension); it is used to derive the higher-order moment
estimates.

\item[(A3)] The detailed-balance condition is also needed in the deterministic
case to reveal the entropy structure of the system; see \cite{CDJ18}.

\item[(A4)] The Lipschitz continuity of the stochastic diffusion $\sigma(u)$
is a standard condition for stochastic PDEs; see, e.g., \cite{PrRo07}.

\item[(A5)] This is the most restrictive assumption. It compensates for 
the singularity of $(\pa h/\pa u_i)(u)=\pi_i\log u_i$ at $u_i=0$.
We show in Lemma \ref{lem.sigma} that
$$
  \sigma_{ij}(u)(\cdot) = \frac{u_i\delta_{ij}}{1+u_i^{1/2+\eta}} 
	\sum_{k=1}^\infty a_k(e_k,\cdot)_U
$$
satisfies Assumption (A5), where $\eta>0$ and $(a_k)\in\ell^2(\R)$. 
Taking into account the gradient estimate from the
entropy inequality (see \eqref{1.ei}), we can allow for more general stochastic 
diffusion terms like \eqref{1.sigma}; see Lemma \ref{lem.sigma2}. 
\end{itemize}
\end{remark}

\blue{\begin{remark}[Reaction terms]\rm
It is possible to include additional nonlinear, continuous reaction terms 
$f:\R^{n}\to \R^{n}$ satisfying
$$
  \int_0^t \sum_{i=1}^{n} \int_\dom f_{i}(u)\frac{\partial h}{\partial u_{i}}
	\dd x\dd s \le C_f\bigg(1 + \int_0^t\int_\dom h(u(s))\dd x\dd s\bigg).
$$
A prominent choice are the so-called Lotka-Volterra source terms
$$
  f_{i}(u)=\bigg(b_{i0}-\sum_{j=1}^{n}b_{ij}u_{j}\bigg)u_{i},\quad i=1,2,
$$
where $b_{ij}\geq 0$ for $i=1,\dots,n$, $j=0,1,\dots,n$. Considering the entropy
density $h$ given by \eqref{1.h}, it is easy to see that this reaction term even 
improves the integrability of the solution, due to bounds for terms of the 
form $u_{i}^{2}\log(u_{i})$, $i=1,\dots,n$.
\end{remark}}

%%%%%%%%%%%%%%%%%%%%%%

\subsection{Main results}

Let $T>0$, $m\in\N$ with $m>d/2+1$, and $D(L)=H_N^m(\dom)$.

\begin{definition}[Martingale solution]\label{def.skt}
A {\em martingale solution} to \eqref{1.eq}--\eqref{1.SKT} is the triple
$(\widetilde U,\widetilde W,\widetilde u)$ such that $\widetilde U
=(\widetilde \Omega,\widetilde{\mathcal F},\widetilde\Prob,\widetilde\F)$
is a stochastic basis with filtration 
$\widetilde \F=(\widetilde{\mathcal F}_t)_{t\ge 0}$, 
$\widetilde W$ is an $n$-dimensional cylindrical Wiener process,
and $\widetilde u=(\widetilde u_1,\ldots,\widetilde u_n)$ 
is a continuous $D(L)'$-valued $\widetilde{\F}$-adapted process such that
$\widetilde u_i\ge 0$ a.e.\ in $\dom\times(0,T)$ $\widetilde\Prob$-a.s., 
\begin{equation}\label{2.regul}
  \widetilde u_i\in L^0(\widetilde \Omega;C^0([0,T];D(L)'))\cap
	L^0(\widetilde \Omega;L^2(0,T;H^1(\dom))),
\end{equation}
the law of $\widetilde u_i(0)$ is the same as for $u_i^0$,
and for all $\phi\in D(L)$, $t\in(0,T)$, $i=1,\ldots,n$, $\widetilde \Prob$-a.s.,
\begin{align}\label{2.martin}
  \langle\widetilde u_i(t),\phi\rangle_{D(L)',D(L)} 
	&= \langle\widetilde u_i(0),\phi\rangle_{D(L)',D(L)}
	- \sum_{j=1}^n\int_0^t\int_\dom A_{ij}(\widetilde u(s))
	\na\widetilde u_j(s)\cdot\na\phi\dd x\dd s \\
  &\phantom{xx}{}+ \sum_{j=1}^n\int_\dom\bigg(\int_0^t\sigma_{ij}(\widetilde u(s))
	\dd\widetilde W_j(s)\bigg)\phi\dd x. \nonumber
\end{align}
\end{definition}

Our main results read as follows.

\begin{theorem}[Existence for the SKT model with self-diffusion]\label{thm.skt}\sloppy
Let Assumptions (A1)-- (A5) be satisfied and let $a_{ii}>0$ for $i=1,\ldots,n$.
Then \eqref{1.eq}--\eqref{1.SKT}
has a global nonnegative martingale solution in the sense of Definition \ref{def.skt}.
\end{theorem}

\begin{theorem}[Existence for the SKT model without self-diffusion]\label{thm.skt2}
Let Assumptions (A1)--(A5) be satisfied, let $d\le 2$, 
and let $a_{0i}>0$ for $i=1,\ldots,n$. We strengthen Assumption (A4) slightly
by assuming that for all $v\in L^2(\dom;\R^n)$,
$$
  \|\sigma(v)\|_{\mathcal{L}_2(U;L^2(\dom))}\le C_\sigma(1+\|v\|_{L^2(\dom)}^\gamma),
$$
where $\gamma<1$ if $d=2$ and $\gamma=1$ if $d=1$. Then \eqref{1.eq}--\eqref{1.SKT}
has a global nonnegative martingale solution in the sense of Definition
\ref{def.skt} with the exception that \eqref{2.regul} and \eqref{2.martin}
are replaced by
$$
  \widetilde u_i\in L^0(\widetilde \Omega;C^0([0,T];D(L)'))\cap
	L^0(\widetilde \Omega;L^2(0,T;W^{1,1}(\dom)))
$$
and, for all $\phi\in D(L)\cap W^{2,\infty}(\dom)$,
\begin{align*}
  \langle\widetilde u_i(t),\phi\rangle_{D(L)',D(L)} 
	&= \langle\widetilde u_i(0),\phi\rangle_{D(L)',D(L)}
	- \int_0^t\int_\dom \widetilde u_i(s)
	\bigg(a_{i0}+\sum_{j=1}^n a_{ij}\widetilde u_j(s)\bigg)\Delta\phi\dd x\dd s \\
  &\phantom{xx}{}+ \sum_{j=1}^n\int_\dom\bigg(\int_0^t\sigma_{ij}(\widetilde u(s))
	\dd\widetilde W_j(s)\bigg)\phi\dd x.
\end{align*}
\end{theorem}

The weak formulation for the SKT system without self-diffusion is weaker than
that one with self-diffusion, since we have only the gradient regularity
$\na\widetilde{u}_i\in L^1(\dom)$, and $A_{ij}(\widetilde{u})$ may be nonintegrable.
However, system \eqref{1.eq} can be written in Laplacian form
according to \eqref{1.form}, which allows for the ``very weak'' formulation stated
in Theorem \ref{thm.skt2}. The condition on $\gamma$ if $d=2$ is needed to
prove the fractional time regularity for the approximative solutions.

\begin{remark}[Nonnegativity of the solution]\rm
The a.s.\ nonnegativity of the population densities is a consequence of the
entropy structure, since the approximate densities $u_i^\eps$ satisfy
$u_i^\eps = u_i(R_\eps(v^\eps)) = \exp(R_\eps(v^\eps)/\pi_i)>0$ a.e.\ in $Q_T$.
This may be surprising since we do not assume that the noise vanishes at zero,
i.e.\ $\sigma_{ij}(u)=0$ if $u_i=0$.
This condition is replaced by the weaker integrability condition for 
$\sigma_{ij}(u)\log u_i$ in Assumption (A5). A similar, but pointwise
condition was imposed in the deterministic case; see Hypothesis (H3) in
\cite[Section 4.4]{Jue16}. The examples in Section \ref{sec.noise} satisfy
$\sigma_{ij}(u)=0$ if $u_i=0$.
\qed
\end{remark}

%%%%%%%%%%%%%%%%%%%%%%%%%%%%%%%%%%%%%%%%%%%%%%%%%%%%%%%%%%%%%%%%%%%%%%%%%%%%%%%

\section{Operator setup}\label{sec.op}

In this section, we introduce the operators needed to define the approximate
scheme. 

\subsection{Definition of the connection operator $L$}

We define an operator $L$ that ``connects'' two Hilbert spaces $V$ and $H$ satisfying
$V\subset H$. This abstract operator allows us to define a regularization operator
that ``lifts'' the dual space $V'$ to $V$. 

\begin{proposition}[Operator $L$]\label{prop.L}
Let $V$ and $H$ be separable Hilbert spaces 
such that the embedding $V\hookrightarrow H$ 
is continuous and dense. Then there exists a bounded, self-adjoint, positive operator
$L:D(L)\to H$ with domain $D(L)=V$.
Moreover, it holds for $L$ and its dual operator $L^*:H\to V'$ (we identify
$H$ and its dual $H'$) that, for some $0<c<1$,
\begin{equation}\label{3.L}
  c\|v\|_V \le \|L(v)\|_H = \|v\|_V, \quad \|L^*(w)\|_{V'}\le \|w\|_H, 
	\quad v\in V,\ w\in H.
\end{equation}
\end{proposition}

We abuse slightly the notation by denoting both dual and adjoint operators by $A^*$.
The proof is similar to \cite[Theorem 1.12]{KPS82}. For the convenience of the
reader, we present the full proof.

\begin{proof}
We first construct some auxiliary operator by means of the Riesz representation
theorem. Let $w\in H$. The mapping $V\to\R$, $v\mapsto(v,w)_H$, is linear 
and bounded. Hence, there exists a unique
element $\widetilde{w}\in V$ such that $(v,\widetilde{w})_{V}=(v,w)_H$
for all $v\in V$. This defines the linear operator $G:H\to V$,
$G(w):=\widetilde{w}$, such that
$$
  (v,w)_H = (v,G(w))_{V}\quad\mbox{for all }v\in V,\ w\in H.
$$
The operator $G$ is bounded and symmetric, since $\|G(w)\|_V=\|\widetilde{w}\|_V
= \|w\|_H$ and
\begin{equation}\label{3.G}
  (G(w),v)_H = (G(w),G(v))_V = (w,G(v))_H\quad\mbox{for all }v,w\in H.
\end{equation}
This means that $G$ is self-adjoint as an operator on $H$. Choosing $v=w\in H$
in \eqref{3.G} gives $(G(v),v)_H=\|G(v)\|_V^2\ge 0$, i.e., 
$G$ is positive. We claim that $G$
is also one-to-one. Indeed, let $G(w)=0$ for some $w\in H$. Then
$0=(v,G(w))_V=(v,w)_H$ for all $v\in V$ and, by the density of the embedding 
$V\hookrightarrow H$, for all $v\in H$.
This implies that $w=0$ and shows the claim.

The properties on $G$ allow us to define $\Lambda:=G^{-1}:D(\Lambda)\to H$, 
where $D(\Lambda)=\operatorname{ran}(G)\subset V$ and $D(\Lambda)$ denotes the
domain of $\Lambda$. By definition, this operator satisfies
$$
  (v,\Lambda(w))_H = (v,w)_V\quad\mbox{for all }v\in V,\ w\in D(\Lambda).
$$
Hence, for all $v,w\in D(\Lambda)$, we have $(v,\Lambda(w))_H=(v,w)_V=(\Lambda(v),w)_H$,
i.e., $\Lambda$ is symmetric. 
Since $G=G^*$, we have $D(\Lambda^*)=\operatorname{ran}(G^*)
=\operatorname{ran}(G)=D(\Lambda)$ and consequently, $\Lambda$ is self-adjoint.
Moreover, $\Lambda$ is densely defined (since $V\hookrightarrow H$ is dense).
As a densely defined, self-adjoint operator, it is also closed.
Finally, $\Lambda$ is one-to-one and positive:
\begin{equation*}
  C\|\Lambda(v)\|_H\|v\|_V 
	\ge \|\Lambda(v)\|_H\|v\|_H \ge (\Lambda(v),v)_H = (v,v)_V = \|v\|_V^2\ge 0
\end{equation*}
for all $v\in D(\Lambda)$ and some $C>0$
and consequently, $\|\Lambda(v)\|_H\ge C^{-1}\|v\|_V$.

Therefore, we can define the square root of $\Lambda$, 
$\Lambda^{1/2}:D(\Lambda^{1/2})\to H$, which is densely defined and closed.
Its domain can be obtained by closing $D(\Lambda)$ with respect to
\begin{equation}\label{3.Lambda}
  \|\Lambda^{1/2}(v)\|_H = (\Lambda^{1/2}(v),\Lambda^{1/2}(v))_H^{1/2}
	= (\Lambda(v),v)_H^{1/2} = (v,v)_V^{1/2} = \|v\|_V
\end{equation}
for $v\in D(\Lambda^{1/2})$.
In particular, the graph norm $\|\cdot\|_H + \|\Lambda^{1/2}(\cdot)\|_H$ is equivalent
to the norm in $V$. We claim that $D(\Lambda^{1/2})=V$. To prove this, let $w\in V$
be orthogonal to $D(\Lambda^{1/2})$. 
Then $(w,v)_V=0$ for all $v\in D(\Lambda^{1/2})$ and, since
$D(\Lambda)\subset D(\Lambda^{1/2})$, in particular for all $v\in D(\Lambda)$. 
It follows that $0=(w,v)_V=(w,\Lambda(v))_H$ for $v\in D(\Lambda)$. 
Since $\Lambda$ is the inverse of $G:H\to V$, we have
$\operatorname{ran}(\Lambda)=H$, and it holds that $(w,\xi)_H=0$ for all $\xi\in H$,
implying that $w=0$. This shows the claim. 

Finally, we define $L:=\Lambda^{1/2}:D(L)=V\to H$, which is a positive and self-adjoint
operator. Estimate \eqref{3.Lambda} shows that
$\|L(v)\|_H = \|v\|_V$ for $v\in V$. 
We deduce from the equivalence between the norm in $V$ and the 
graph norm of $L$ that, for some $C>0$ and all $v\in V$,
$$
  \|v\|_V \le C(\|L(v)\|_H+\|v\|_H) = C(\|L(v)\|_V + \|L^{-1}L(v)\|_H)
	\le C(1+\|L^{-1}\|)\|L(v)\|_H,
$$
which proves the lower bound in \eqref{3.L}.
The dual operator $L^*:H\to V'$ is bounded too, since it holds for
all $w\in H$ that
$$%\begin{align*}
  \|L^*(w)\|_{V'} %&= \sup_{\|v\|_V=1}|\langle L^*(w),v\rangle_{V',V}|
	= \sup_{\|v\|_V=1}|(w,L(v))_H| 
	%&\le \sup_{\|v\|_V=1}\|w\|_H\|L(v)\|_H 
	\le \sup_{\|v\|_V=1}\|w\|_H\|v\|_V = \|w\|_H.
$$%\end{align*}
This ends the proof.
\end{proof}

We apply Proposition \ref{prop.L} to $V=H^m_N(\dom)$ and $H=L^2(\dom)$, recalling
that $H^m_N(\dom)=\{v\in H^m(\dom):\na v\cdot\nu=0$ on $\pa\dom\}$ and $m>d/2+1$.
Then, by Sobolev's embedding, $D(L)\hookrightarrow W^{1,\infty}(\dom)$.
Observe the following two properties that are used later:
\begin{align}\label{3.LL}
  & \|L^*L(v)\|_{V'}\le \|v\|_V, \quad \|L^*(w)\|_{V'}\le \|w\|_H
	\quad\mbox{for all }v\in V,\ w\in H. %\\
	%& \langle L^*L(v),w\rangle_{V',V} = (L(v),L(w))_H \quad\mbox{for all }
	%v,w\in V. \label{3.LLL}
\end{align}

The following lemma is used in the proof of
Proposition \ref{prop.ent} to apply It\^o's lemma.

\begin{lemma}[Operator $L^{-1}$]\label{lem.L1}
Let $L^{-1}:\operatorname{ran}(L)\to D(L)$ be the inverse of $L$ and let
$D(L^{-1}):=\overline{D(\Lambda)}$ be the closure of $D(\Lambda)$ 
with respect to $\|L^{-1}(\cdot)\|_H$.
Then $D(L)'$ is isometric to $D(L^{-1})$.
In particular, it holds that
$(L^{-1}(v),L^{-1}(w))_{H} = (v,w)_{D(L)'}$ for all $v$, $w\in D(L)'$.
\end{lemma}

\begin{proof}
The proof is essentially contained in \cite[p.~136ff]{KPS82} and we only sketch it.
Let $F\in D(L^{-1})'$. Then $|F(v)|\le C\|L^{-1}(v)\|_H$
for all $v\in D(\Lambda)$ and, as a consequence, $|F(Lu)|\le C\|u\|_H$ for 
$u=L^{-1}(v)\in D(L)$. The density of $L^{-1}(D(\Lambda))$ in $H$ guarantees the
unique representation $F(Lu)=(u,w)_H$ for some $w\in H$, and we can represent
$F$ in the form $F(v)=(L^{-1}v,w)_H=(v,L^{-1}w)_H$, where $L^{-1}w\in D(L)$.
This shows that every element of $D(L^{-1})'$ 
can be identified with an element of $D(L)$.

Conversely, if $w\in D(L)$, we consider
functionals of the type $v\mapsto(v,w)_H$ for $v\in D(\Lambda)$, which are
bounded in $\|L^{-1}(\cdot)\|_H$. These functionals can be extended by continuity
to functionals $F$ belonging to $D(L^{-1})'$. 
The proof in \cite[p.~137]{KPS82}
shows that $\|F\|_{D(L^{-1})'}=\|w\|_{D(L)}$. 
We conclude that $D(L^{-1})'$
is isometric to $D(L)$. Since Hilbert spaces are reflexive,
$D(L^{-1})$ is isometric to $D(L)'$.
\end{proof}

%\begin{remark}[Inverse operator $L^{-1}$]\label{rem.inverse}\rm
%The inverse operator $L^{-1}:H\to V\subset H$ exists and is linear and bounded.
%Since $D(L)'\subset H'=H$ is dense, we deduce from the continuous linear 
%extension theorem that there exists a unique bounded extension $L^{-1}:D(L)'\to H$.
%Following the arguments of \cite[p.~136ff]{KPS82}, we can identify $D(L^{-1})'$ in an 
%isomorphic way with $D(L)$. By reflexibility, $D(L)'$ is isomorphic to $D(L^{-1})$, 
%endowed with the norm $\|L^{-1}(\cdot)\|_H$. In particular, it holds that
%$(L^{-1}(v),L^{-1}(w))_{H} = (v,w)_{D(L)'}$ for all $v$, $w\in D(L)'$. 
%Moreover, the arguments of \cite[p.~136ff]{KPS82} and the self-adjointness
%of $L$ lead to
%\begin{align*}
%  \langle v&,w\rangle_{D(L)',D(L)} = (L^{-1}(v),L(w))_{H}
%	= (LL^{-1}L^{-1}(v),L(w))_H \\
%	&= (L^{-1}(v),(LL^{-1})^*L(w))_H
%	= (L^{-1}(v),L^{-1}L^*L(w))_H = (v,L^*L(w))_{D(L)'}
%\end{align*}
%for $v\in D(L)'$, $w\in D(L)$. Thus, we can associate an element $w\in D(L)$ with
%$L^*L(w)\in D(L)'$. This property will be needed in the verification
%of the It\^o lemma below.
%\qed\end{remark}

\begin{lemma}[Operator $u$]\label{lem.u}
The mapping $u:=(h')^{-1}$ from $D(L)$ to $L^\infty(\dom)$ is Fr\'echet 
differentiable and, as a mapping from $D(L)$ to $D(L)'$, monotone.
\end{lemma}

\begin{proof}
Let $w\in D(L)\hookrightarrow L^\infty(\dom)$ (here we use $m>d/2$). 
Then $u(w)=(x\mapsto u(w(x)))
\in L^\infty(\dom)$,
showing that $u:D(L)\to L^\infty(\dom)=(L^1(\dom))'\hookrightarrow D(L)'$ 
is well defined.
It follows from the mean-value theorem that for all $w$, $\xi\in D(L)$,
$$
  \|u(w+\xi)-u(w)-u'(w)\xi\|_{L^\infty(\dom)}
	\le C\|\xi\|_{D(L)}^2\bigg\|\int_0^1 (1-s)u''(w+s\xi)\dd s\bigg\|_{L^\infty(\dom)}.
$$
Since $u''$ maps bounded sets to bounded sets, the integral is bounded. Thus,
$u:D(L)\to L^\infty(\dom)$ is Fr\'echet differentiable.
For the monotonicity, we use the convexity of $h$ and hence the monotonicity of $h'$:
\begin{align*}
  \langle u(v)-u(w),v-w\rangle_{D(L)',D(L)}
	&= (u(v)-u(w),v-w)_{L^2(\dom)} \\
	&= (u(v)-u(w),h'(u(v))-h'(u(w)))_{L^2(\dom)} \ge 0
\end{align*}
for all $v$, $w\in D(L)$. This proves the lemma.
\end{proof}

%%%%%%%%%%%%%%%%

\subsection{Definition of the regularization operator $R_\eps$}

First, we define another operator, \blue{denoted by $Q_{\eps}$}, that maps $D(L)$ to $D(L)'$. Its inverse
is the desired regularization operator.

\begin{lemma}[Operator $Q_\eps$]\label{lem.Qeps}
Let $\eps>0$ and define $Q_\eps:D(L)\to D(L)'$ by $Q_\eps(w)=u(w)+\eps L^*Lw$, where
$w\in D(L)$. Then $Q_\eps$ is Fr\'echet differentiable, strongly monotone,
coercive, and invertible. 
Its Fr\'echet derivative $\DD Q_\eps[w](\xi)=u'(w)\xi + \eps L^*L\xi$
for $w$, $\xi\in D(L)$ is continuous, strongly monotone, coercive, and
invertible.
\end{lemma}

\begin{proof}
The mapping $Q_\eps$ is well defined since $w\in D(L)\hookrightarrow L^\infty(\dom)$
implies that $u(w)\in L^\infty(\dom)$ and hence, 
$\|u(w)\|_{D(L)'}\le C\|u(w)\|_{L^1(\dom)}$ is finite. We show that
$Q_\eps$ is strongly monotone. For this, let $v$, $w\in D(L)$ and compute
\begin{align}\label{3.aux}
  \langle Q_\eps&(v)-Q_\eps(w),v-w\rangle_{D(L)',D(L)} \\
	&= (u(v)-u(w),v-w)_H + \eps\langle L^*L(v-w),v-w\rangle_{D(L)',D(L)} \nonumber \\
	&\ge \eps\langle L^*L(v-w),v-w\rangle_{D(L)',D(L)}
	= \eps\|L(v-w)\|_H^2 \ge \eps c\|v-w\|_{D(L)}^2 \nonumber
\end{align}
where we used the monotonicity of $w\mapsto u(w)$
and the lower bound in \eqref{3.L}.
The coercivity of $Q_\eps$ is a consequence of the strong monotonicity:
\begin{align*}
  \langle Q_\eps(v),v\rangle_{D(L)',D(L)}
	&= \langle Q_\eps(v)-Q_\eps(0),v-0\rangle_{D(L)',D(L)}
	+ \langle Q_\eps(0),v\rangle_{D(L)',D(L)} \\
	&\ge \eps c\|v\|_{D(L)}^2 + (u(0),v)_H 
	\ge \eps c\|v\|_{D(L)}^2 - C|u(0)|\,\|v\|_{D(L)}
\end{align*}
for $v\in D(L)$. Based on these properties, the invertibility of $Q_\eps$ 
now follows from Browder's theorem \cite[Theorem 6.1.21]{DrMi13}.

Next, we show the properties for $\DD Q_\eps$. The operator 
$\DD Q_\eps[w]:D(L)\to D(L)'$ is well defined for all $w\in D(L)$, since
$$
  \|u'(w)\xi\|_{D(L)'}\le C\|u'(w)\xi\|_{L^2(\dom)}\le C\|u'(w)\|_{L^2(\dom)}
  \|\xi\|_{L^\infty(\dom)}\le C\|u'(w)\|_{L^2(\dom)}\|\xi\|_{D(L)}
$$ 
for all $\xi\in D(L)\hookrightarrow L^\infty(\dom)$.
The strong monotonicity of $\DD Q_\eps[w]$ for $w\in D(L)$ follows from
the positive semidefiniteness of $u'(w)=(h'')^{-1}(u(w))$ and the 
lower bound in \eqref{3.L}:
\begin{align*}
  \langle \DD & Q_\eps[w](\xi)-\DD Q_\eps[w](\eta),\xi-\eta\rangle_{D(L)',D(L)} \\
	&= (u'(w)(\xi-\eta),\xi-\eta)_H 
	+ \eps\langle L^*L(\xi-\eta),\xi-\eta\rangle_{D(L)',D(L)} \\
	&\ge \eps\|L(\xi-\eta)\|_H^2 \ge \eps c\|\xi-\eta\|_{D(L)}^2
\end{align*}
for $\xi$, $\eta\in D(L)$.
The choice $\eta=0$ yields immediately the coercivity of $\DD Q_\eps[w]$.
The invertibility of $\DD Q_\eps[w]$ follows again from Browder's theorem. 
\end{proof}

Lemma \ref{lem.Qeps} shows that the inverse of $Q_\eps$ exists. We set
$R_\eps:=Q_\eps^{-1}:D(L)'\to D(L)$, which is the desired regularization
operator. It has the following properties.

\begin{lemma}[Operator $R_\eps$]\label{lem.Reps}
The operator $R_\eps:D(L)'\to D(L)$ is Fr\'echet differentiable and strictly monotone.
In particular, it is Lipschitz continuous with Lipschitz constant
$C/\eps$, where $C>0$ does not depend on $\eps$. The Fr\'echet derivative \blue{is also Lipschitz continuous with the same constant and satisfies}
$$
  \DD R_\eps[v] = (\DD Q_\eps[R_\eps(v)])^{-1}
	= (u'(R_\eps(v))+\eps L^*L)^{-1} \quad\mbox{for }v\in D(L)',
$$
and it is Lipschitz continuous with constant $C/\eps$, satisfying
$\|\DD R_\eps[v](\xi)\|_{D(L)}\le \eps^{-1}C\|\xi\|_{D(L)'}$
for $v$, $\xi\in D(L)'$.
\end{lemma}

\begin{proof}
We show first the Lipschitz continuity of $R_\eps$. 
Let $v_1$, $v_2\in D(L)'$. Then there exist
$w_1$, $w_2\in D(L)$ such that $v_1=Q_\eps(w_1)$, $v_2=Q_\eps(w_2)$. Hence,
using \eqref{3.L} and \eqref{3.aux},
\begin{align*}
  \|R_\eps(v_1)-R_\eps(v_2)\|_{D(L)}^2
	&= \|w_1-w_2\|_{D(L)}^2 \le C\|L(w_1-w_2)\|_H^2 \\
	&\le \eps^{-1}C\langle Q_\eps(w_1)-Q_\eps(w_2),w_1-w_2\rangle_{D(L)',D(L)} \\
	&\le \eps^{-1}C\|Q_\eps(w_1)-Q_\eps(w_2)\|_{D(L)'}\|w_1-w_2\|_{D(L)} \\
	&= \eps^{-1}C\|v_1-v_2\|_{D(L)'}\|R_\eps(v_1)-R_\eps(v_2)\|_{D(L)},
\end{align*}
proving that $R_\eps$ is Lipschitz continuous with Lipschitz constant $C/\eps$.
The Fr\'echet differentiability is a consequence of the inverse function theorem
and $\DD R_\eps[v] = (\DD Q_\eps[R_\eps(v)])^{-1}$ for $v\in D(L)'$.

We verify the strict monotonicity of $R_\eps$. Let $v$, $w\in D(L)'$
with $v\neq w$. Because of the strong monotonicity of $Q_\eps$, we have
\begin{align*}
  \langle v-w,R_\eps(v)-R_\eps(w)\rangle_{D(L)',D(L)}
	&= \langle Q_\eps(R_\eps(v))-Q_\eps(R_\eps(w)),R_\eps(v)-R_\eps(w)
	\rangle_{D(L)',D(L)} \\
	&\ge \eps^{-1} c\|R_\eps(v)-R_\eps(w)\|_{D(L)}^2 > 0,
\end{align*}
and the right-hand side vanishes only if $v=w$, since $R_\eps$ is one-to-one.

Next, we show that $\DD R_\eps[v]$ is Lipschitz continuous. Let
$w_1$, $w_2\in D(L)$. By Lemma \ref{lem.Qeps}, $\DD Q_\eps[w]$ is strongly
monotone. Thus, for any $w\in D(L)$,
\begin{align*}
  \eps c\|w_1-w_2\|_{D(L)}^2 &\le \langle\DD Q_\eps[w](w_1)-\DD Q_\eps[w](w_2),
	w_1-w_2\rangle_{D(L)',D(L)} \\
	&\le \|\DD Q_\eps[w](w_1)-\DD Q_\eps[w](w_2)\|_{D(L)'}\|w_1-w_2\|_{D(L)}.
\end{align*}
Let $v_1=\DD Q_\eps[w](w_1)$ and $v_2=\DD Q_\eps[w](w_2)$. We infer that
\begin{align*}
  \|(\DD & Q_\eps[w])^{-1}(v_1)-(\DD Q_\eps[w])^{-1}(v_2)\|_{D(L)}
	= \|w_1-w_2\|_{D(L)} \\
	&\le \eps^{-1}C\|\DD Q_\eps[w](w_1)-\DD Q_\eps[w](w_2)\|_{D(L)'}
	= \eps^{-1}C\|v_1-v_2\|_{D(L)'},
\end{align*}
showing the Lipschitz continuity of $(\DD Q_\eps[w])^{-1}$ and
$\DD R_\eps[v] = (\DD Q_\eps[R_\eps(v)])^{-1}$. Finally, choosing $w=R_\eps[v]$
and $v_2=0$, $\|\DD R_\eps[v](v_1)\|_{D(L)}\le \eps^{-1}C\|v_1\|_{D(L)'}$.
\end{proof}

%%%%%%%%%%%%%%%%%%%%%%%%%%%%%%%%%%%%%%%%%%%%%%%%%%%%%%%%%%%%%%%%%%%%%%%%%%%%%%%

\section{Existence of approximate solutions}\label{sec.approx}

In the previous section, we have introduced the regularization operator
$R_\eps:D(L)'\to D(L)$. The entropy variable $w$ is replaced by the regularized
variable $R_\eps(v)$ for $v\in D(L)'$. 
Setting $v = u(R_\eps(v)) + \eps L^*LR_\eps(v)$, we
consider the regularized problem
\begin{align}\label{4.approx}
  & \dd v = \diver\big(B(R_\eps(v))\na R_\eps(v)\big)\dd t 
	+ \sigma\big(u(R_\eps(v))\big)\dd W(t) \quad\mbox{in }\dom,\ t\in[0,T\wedge\tau), \\
	& v(0) = u^0\quad\mbox{in }\dom, \quad 
	\na R_\eps(v)\cdot\nu = 0 \quad\mbox{on }
	\pa\dom,\ t>0, \label{4.bic}
\end{align}
recalling that $B(w)=A(u(w))h''(u(w))^{-1}$ for $w\in\R^n$.

We clarify the notion of solution to problem \eqref{4.approx}--\eqref{4.bic}.
Let $T>0$, let $\tau$ be an $\F$-adapted stopping time,
and let $v$ be a continuous, $D(L)'$-valued, $\F$-adapted process. We call $\blue{(v,\tau)}$
a {\em local \blue{strong} solution} to \eqref{4.approx} if
$$
  v(\omega,\cdot,\cdot)\in L^2([0,T\wedge\tau(\omega));D(L)')\cap
	C^0([0,T\wedge\tau(\omega));D(L)')
$$
for a.e.\ $\omega\in\Omega$ and for all $t\in[0,T\wedge\tau)$,
\begin{align}
  & v(t) = v(0) + \int_0^t\diver\big(B(R_\eps(v(s)))\na R_\eps(v(s))\big)\dd s
	+ \int_0^t\sigma\big(u(R_\eps(v(s))\big)\dd W(s), \label{4.defsol} \\
	& \na R_\eps(v)\cdot\nu = 0\quad\mbox{on }\pa\dom\quad
	\Prob\mbox{-a.s.} \label{4.bc}
\end{align}
It can be verified that $R_\eps$ is strongly measurable and, if $v$ is 
progressively measurable, also progressively measurable. Furthermore, if
$w$ is progressively measurable then so does $u(w)$, and if $v\in C^0([0,T];D(L)')$,
we have $R_\eps(v)\in C^0([0,T];D(L))$ and $u(R_\eps(v))\in L^\infty(Q_T)$. 
Finally, if $v\in L^0(\Omega;L^p(0,T;D(L)'))$ for $1\le p\le\infty$, then
$\diver(B(u(R_\eps(v)))\na R_\eps(v))\in L^0(\Omega;L^p(0,T;D(L)')))$.
Therefore, the integrals in \eqref{4.defsol} are well defined.
The local \blue{strong} solution is called a {\em global \blue{strong} solution} 
if $\Prob(\tau=\infty)=1$.
Given $t>0$ and a process $v\in L^2(\Omega;C^0([0,t];D(L)'))$, we introduce the 
stopping time
$$
  \tau_R := \inf\{s\in[0,t]:\|v(s)\|_{D(L)'}>R\}\quad\mbox{for }R>0.
$$ 
\blue{The stopping time $\tau_{R}$ is $\Prob$-a.s.\ positive}. Indeed, by Chebychev's inequality, 
it holds for $\delta>0$ that
$$
  \Prob(\tau_R>\delta) \ge \Prob\Big(\sup_{0<t<\delta}\|v(t\wedge\tau_R)\|_{D(L)'}
	\le R\Big) \ge 1 - \frac{1}{R^2}\E\sup_{0<t<\delta}
	\|v(t\wedge\tau_R)\|_{D(L)'}^2.
$$
Then, inserting \eqref{4.defsol} and using the properties of the operators
introduced in Section \ref{sec.op}, we can show that
$\Prob(\tau_R>\delta)\ge 1-C(\delta)$, where $C(\delta)\to 0$ as $\delta\to 0$,
which proves the claim.

We impose the following general assumptions.

\begin{labeling}{(A44)} 

\item[(H1)] Entropy density: Let $\simplex\subset\R^n$ be a domain and let
$h\in C^2(\simplex;[0,\infty))$ be such that $h':\simplex\to\R^n$ and
$h''(u)\in\R^{n\times n}$ for $u\in\simplex$ are invertible and
there exists $C>0$ such that $|u|\le C(1+h(u))$ for all $u\in\simplex$.

\item[(H2)] Initial datum: $u^0=(u_1^0,\ldots,u_n^0)\in 
L^\infty(\Omega;L^2(\dom;\R^n))$ is 
$\mathcal{F}_0$-measurable satisfying $u^0(x)\in\simplex$ for a.e.\ $x\in\dom$
$\Prob$-a.s.

\item[(H3)] Diffusion matrix: $A=(A_{ij})\in C^1(\overline\dom;\R^{n\times n})$
grows at most linearly and the matrix $h''(u)A(u)$ is positive semidefinite for
all $u\in\simplex$.
\end{labeling}

\begin{remark}[Discussion of the assumptions]\rm 
Hypothesis (H1) and the positive semidefiniteness condition of 
\blue{$h''(u)A(u)$} in (H3) are necessary for the entropy structure
of the general cross-diffusion system. The entropy density \eqref{1.h} 
with $\simplex=(0,\infty)^n$ satisfies 
Hypothesis (H1), and the diffusion matrix \eqref{1.SKT} fulfills (H3).
The differentiability of $A$ is needed to apply \cite[Prop.~4.1.4]{LiRo15}
(stating that the assumptions of the abstract existence Theorem 4.2.2 are
satisfied)
and can be weakened to continuity, weak monotonicity, and coercivity conditions. 
The growth condition for $A$ is technical;
it guarantees that the integral formulation associated to \eqref{1.eq} is well defined.
Hypothesis (H2) guarantees that $h(u^0)$ is well defined.
\qed\end{remark}

We consider general approximate stochastic cross-diffusion systems, since
the existence result for \eqref{4.approx} may be
useful also for other stochastic cross-diffusion systems.

\begin{theorem}[Existence of approximate solutions]\label{thm.approx}\
Let Assumptions (A1)--(A2), (A4)--(A5), (H1)--(H3) be satisfied and 
let $\eps>0$, $R>0$.
Then problem \eqref{4.approx}--\eqref{4.bic} has a unique local solution 
$\blue{(v^\eps,\tau_R)}$.
\end{theorem}

\begin{proof}
We want to apply Theorem 4.2.4 and Proposition 4.1.4 of \cite{LiRo15}.
To this end, we need to verify that the operator
$M:D(L)'\to D(L)'$, $M(v) := \diver(B(R_\eps(v))\na R_\eps(v))$, is Fr\'echet
differentiable and has at most linear growth, $\DD M[v]-cI$ is negative semidefinite
for all $v\in D(L)'$ and some $c>0$, and $\sigma$ is Lipschitz continuous.
 
By the regularity of the matrix $A$ and the entropy density $h$,
the operator $D(L)\to D(L)'$, $w\mapsto\diver(B(w)\na w)$, is Fr\'echet
differentiable. Then the Fr\'echet differentiability of $R_\eps$ 
(see Lemma \ref{lem.Reps}) and the chain rule imply that the operator
$M$ is also Fr\'echet differentiable with derivative
$$
  \DD M[v](\xi) = \diver\big(\DD B[R_\eps(v)](\DD R_\eps[v](\xi))
	\na R_\eps(v)\big) + \diver\big(B(R_\eps(v))\na\DD R_\eps[v](\xi)\big),
$$
where $v$, $\xi\in D(L)'$. We claim that this derivative is
locally bounded, i.e.\ if $\|v\|_{D(L)'}\le K$ then 
$\|\DD M[v](\xi)\|_{D(L)'}\le C(K)\|\xi\|_{D(L)'}$.
For this, we deduce from the Lipschitz continuity of $R_\eps$
(Lemma \ref{lem.Reps}) and the property $u(R_\eps(v))\in L^\infty(\dom)$ 
for $v\in D(L)'$ that
$$
  \|B(R_\eps(v))\|_{L^\infty(\dom)} + \|\DD B[R_\eps(v)]\|_{L^\infty(\dom)}
	\le C(1+\|R_\eps(v)\|_{D(L)}) \le C(\eps)(1+\|v\|_{D(L)'}),
$$
where $\DD B[R_\eps(v)]$ is interpreted as a matrix.
Recalling from Lemma \ref{lem.Reps} that
$$
  \|\DD R_\eps[v](\xi)\|_{D(L)}\le C(\eps)\|\xi\|_{D(L)'}
	\quad\mbox{for all }\xi\in D(L)',
$$
we obtain for $\|v\|_{D(L)'}\le K$ and $\xi\in D(L)'$:
\begin{align*}
  \|\DD M[v](\xi)\|_{D(L)'}
	&\le C\big\|\DD B[R_\eps(v)](\DD R_\eps[v](\xi))\na R_\eps(v)
	+ B(R_\eps(v))\na\DD R_\eps[v](\xi)\big\|_{L^1(\dom)} \\
	&\le C\|\DD B[R_\eps(v)](\DD R_\eps[v](\xi))\|_{L^\infty(\dom)}
	\|\na R_\eps(v)\|_{L^1(\dom)} \\
	&\phantom{xx}{}
	+ C\|B(R_\eps(v))\|_{L^\infty(\dom)}\|\na\DD R_\eps[v](\xi)\|_{L^1(\dom)} \\
	&\le C\|\DD B[R_\eps(v)]\|_{L^\infty(\dom)}\|\DD R_\eps[v](\xi)\|_{D(L)}
	\|R_\eps(v)\|_{D(L)} \\
	&\phantom{xx}{}+ C\|B(R_\eps(v))\|_{L^\infty(\dom)}
	\|\DD R_\eps[v](\xi)\|_{D(L)} \\
	&\le C(\eps)(1+\|v\|_{D(L)'})\|\xi\|_{D(L)'}\le C(\eps,K)\|\xi\|_{D(L)'}.
\end{align*}
This proves the claim.
Thus, if $\|v\|_{D(L)'}\le K$, there exists $c>0$ such that
$$
  (\xi,\DD M[v](\xi)- c\xi)_{D(L)'}\le 0\quad\mbox{for }\xi\in D(L)'.
$$
Moreover, by Lemma \ref{lem.Reps} again,
\begin{align*}
  \|M(v)\|_{D(L)'} &\le C\|B(R_\eps(v))\na R_\eps(v)\|_{L^1(\dom)}
	\le C\|\na R_\eps(v)\|_{L^1(\dom)} \\
	&\le C\|R_\eps(v)\|_{D(L)} \le \eps^{-1}C(1+\|v\|_{D(L)'}).
\end{align*}
It follows from Assumption (A4) and Lemma \ref{lem.u} that for 
$v$, $\bar v\in D(L)'$ with  $\|v\|_{D(L)'}\le K$ and $\|\bar v\|_{D(L)'}\le K$, 
\begin{align*}
  \|\sigma&(u(R_\eps(v)))-\sigma(u(R_\eps(\bar v)))\|_{\mathcal{L}_2(U;D(L)')}
  \le C\|\sigma(u(R_\eps(v)))-\sigma(u(R_\eps(\bar v)))
	\|_{\mathcal{L}_2(U;L^2(\dom))} \\
	&\le C(K)\|u(R_\eps(v)))-u(R_\eps(\bar v))\|_{L^2(\dom)} \\
	&\le C(K)\|R_\eps(v)-R_\eps(\bar v)\|_{D(L)}
	\le C(\eps,K)\|v-\bar v\|_{D(L)'},
\end{align*}
where $C(K)$ also depends on the $L^\infty(\dom)$ norms of $u'(R_\eps(v))$
and $u'(R_\eps(\bar v))$. 

These estimates show that the assumptions of \blue{\cite[Theorem 4.2.4]{LiRo15}}
are satisfied in the ball $\{v\in D(L)':\|v\|_{D(L)'}\le K\}$. 
An inspection of the proof of that theorem, which is based on the Galerkin
method and It\^o's lemma, shows that {\em local} bounds are sufficient to
conclude the existence of a {\em local} solution $v$ up to the stopping time $\tau_R$. 
The boundary conditions follow from $R_\eps(v)\in D(L)=H^m_N(\dom)$ and the
definition of the space $H^m_N(\dom)$.
\end{proof}

For the entropy estimate we need two technical lemmas whose proofs are deferred to 
Appendix \ref{sec.proofs}.

\begin{lemma}\label{lem.ab}
Let $w\in D(L)$, $a=(a_{ij})\in L^1(\dom;\R^{n\times n})$, and 
$b=(b_{ij})\in D(L)^{n\times n}$ satisfying $\DD R_\eps[w](a) = b$.
Then
$$
	\int_\dom a:b\dd x 
	\le \int_\dom\operatorname{tr}[a^T u'(w)^{-1}a]\dd x.
%	= \int_\dom\big\|\sqrt{u'(w)}{\,}^{-1}a\|_F^2\dd x,
$$
%where $\|\cdot\|_F$ denotes the Frobenius norm.
\end{lemma}

\begin{lemma}\label{lem.v0}
Let $v^0\in L^p(\Omega;L^1(\dom))$ for some $p\ge 1$ satisfies
$\E\int_\dom h(v^0)\dd x\le C$. Then 
$$
  \int_\dom h(u(R_\eps(v^0)))\dd x + \frac{\eps}{2}\|LR_\eps(v^0)\|_{L^2(\dom)}^2
	\le \int_\dom h(v^0)\dd x.
$$
\end{lemma}

We turn to the entropy estimate.

\begin{proposition}[Entropy inequality]\label{prop.ent}
Let $\blue{(v^\eps,\tau_R)}$ be a local solution to \eqref{4.approx}--\eqref{4.bic} and set
$v^R(t)=v^\eps(\omega,t\wedge\tau_R(\omega))$ for $\omega\in\Omega$, 
$t\in(0,\tau_R(\omega))$.
Then there exists a constant $C(u^0,T)>0$, depending on $u^0$ and $T$ but 
not on $\eps$ and $R$, such that
\begin{align*}
  \E&\sup_{0<t<T\wedge\tau_R}\int_\dom h(u^\eps(t))\dd x
	+ \frac{\eps}{2}\E\sup_{0<t<T\wedge\tau_R}\|Lw^\eps(t)\|_{L^2(\dom)}^2 \\
	&{}+ \E\sup_{0<t<T\wedge\tau_R}\int_0^t\int_\dom
	\na w^\eps(s):B(w^\eps(s))\na w^\eps(s)\dd x\dd s \le C(u^0,T),
\end{align*}
where $u^\eps:=u(R_\eps(v^R))$ and $w^\eps:=R_\eps(v^R)$.
\end{proposition}

\begin{proof}
The result follows from It\^o's lemma using a regularized entropy.
More precisely, we want to apply the It\^o lemma in the version of 
\cite[Theorem 3.1]{Kry13}.
To this end, we verify the assumptions of that theorem. Basically, we need 
a twice differentiable function ${\mathcal H}$ on a Hilbert space $H$, 
whose derivatives
satisfy some local growth conditions on $H$ and $V$, where $V$ is another
Hilbert space such that the embedding
$V\hookrightarrow H$ is dense and continuous. We choose $V=H=D(L)'$
and the regularized entropy
\begin{equation}\label{3.ent}
  {\mathcal H}(v) := \int_\dom h(u(R_\eps(v)))\dd x 
	+ \frac{\eps}{2}\|LR_\eps(v)\|_{L^2(\dom)}^2, \quad v\in D(L)'.
\end{equation}
Recall that $R_\eps(v) = h'(u(R_\eps(v)))$ for $v\in D(L)'$, since
$u=u(w)$ is the inverse of $h'$. 
Then, in view of the regularity assumptions for $h$ and Lemma \ref{lem.Reps},
${\mathcal H}$ is Fr\'echet differentiable with derivative
\begin{align*}
  \DD {\mathcal H}[v](\xi) &= \int_\dom\big(h'(u(R_\eps(v)))u'(R_\eps(v))
	\DD R_\eps[v](\xi)
	+ \eps L\DD R_\eps[v](\xi)\cdot LR_\eps(v)\big)\dd x \\
	&= \big\langle (u'(R_\eps(v))+\eps L^*L)\DD R_\eps[v](\xi),R_\eps(v) 
	\big\rangle_{D(L)',D(L)} \\
	&= \big\langle \DD Q_\eps[R_\eps(v)]\DD R_\eps[v](\xi),R_\eps(v)
	\big\rangle_{D(L)',D(L)}
	= \langle \xi,R_\eps(v)\rangle_{D(L)',D(L)},
\end{align*}
where $v$, $\xi\in D(L)'$.
In other words, $\DD {\mathcal H}[v]$ can be identified with 
$R_\eps(v)\in D(L)$.
In a similar way, we can prove that $\DD {\mathcal H}[v]$ 
is Fr\'echet differentiable with
$$
  \DD^2 {\mathcal H}[v](\xi,\eta) = \langle \xi,\DD R_\eps[v](\eta)\rangle_{D(L)',D(L)}
	\quad\mbox{for } v,\,\xi,\,\eta\in D(L)'.
$$
We have, thanks to the Lipschitz continuity of $R_\eps$
and $\DD R_\eps[v]$ (see Lemma \ref{lem.Reps})
for all $v$, $\xi\in D(L)'$ with $\|v\|_{D(L)'}\le K$ for some $K>0$,
\begin{align*}
  |\DD {\mathcal H}[v](\xi)| &\le \|R_\eps(v)\|_{D(L)}\|\xi\|_{D(L)'}
	\le C(\eps)(1+\|v\|_{D(L)'})\|\xi\|_{D(L)'} \le C(\eps,K)\|\xi\|_{D(L)'}, \\
  |\DD^2 {\mathcal H}[v](\xi,\xi)| &\le \|\DD R_\eps[v](\xi)\|_{D(L)}\|\xi\|_{D(L)'}
	\le C(\eps)\|\xi\|_{D(L)'}^2.
\end{align*}

Finally, for any $\eta\in D(L)'$, we need an estimate for the mapping $D(L)'\to\R$, 
$v\mapsto \DD {\mathcal H}[v](\eta)$. 
We have identified $\DD {\mathcal H}[v]$ with $R_\eps(v)\in D(L)$, but we need
an identification in $D(L)'$. As in Lemma \ref{lem.L1}, the operator $L$ can be 
constructed in such a way that the Riesz representative in $D(L)'$ of a functional
acting on $D(L)'$ can be expressed via the application of $L^*L$ to an element of
$D(L)$. Indeed, for $F\in D(L)$ and $\xi\in D(L)'$, we infer from Lemma \ref{lem.L1}
that
\begin{align*}
  \langle\xi,F\rangle_{D(L)',D(L)} &= (L^{-1}\xi,LF\rangle_{D(L)',D(L)}
	= ((LL^{-1})L^{-1}\xi,LF)_{L^2(\dom)} \\
	&= (L^{-1}\xi,L^{-1}L^*LF)_{L^2(\dom)} = (L^*LF,\xi)_{D(L)'}.
\end{align*}
Hence, we can associate $\DD {\mathcal H}[v]$ with $L^*LR_\eps(v) \in D(L)'$.
Then, by the first estimate in \eqref{3.LL} and the Lipschitz continuity of 
$R_\eps$,
\begin{align*}
  \|L^*LR_\eps(v)\|_{D(L)'} 
	&\le C\|R_\eps(v)\|_{D(L)} 
	\le C\|R_\eps(v)-R_\eps(0)\|_{D(L)} + C\|R_\eps(0)\|_{D(L)} \\
	&\le C(\eps)(1+\|v\|_{D(L)'})\quad\mbox{for all }v\in D(L)',
\end{align*}
giving the desired estimate for $\DD{\mathcal H}[v]$ in $D(L)'$.
Thus, the assumptions of the It\^o lemma, as stated in \cite{Kry13}, are satisfied.

To simplify the notation, we set $u^\eps := u(R_\eps(v^R))$ and 
$w^\eps := R_\eps(v^R)$ in the following.
By It\^o's lemma, using $\DD {\mathcal H}[v^R]=h'(u^\eps)$, 
$\DD^2 {\mathcal H}[v^R]=\DD R_\eps(v^R)$, we have
\begin{align}\label{4.H0}
  {\mathcal H}(v^R(t))  
	&= {\mathcal H}(v(0)) + \int_0^t\big\langle
	\diver\big(B(w^\eps)\na h'(u^\eps(s))\big),w^\eps(s)
	\big\rangle_{D(L)',D(L)}\dd s \\
	&\phantom{xx}{}+ \sum_{k=1}^\infty\sum_{i,j=1}^n\int_0^t\int_\dom
	\frac{\pa h}{\pa u_i}(u^\eps(s))\sigma_{ij}(u^\eps(s))e_k\dd x\dd W_j^k(s)
  \nonumber \\
	&\phantom{xx}{}
	+ \frac12\sum_{k=1}^\infty\int_0^t\int_\dom\DD R_\eps[v^R(s)]
	\big(\sigma(u^\eps(s))e_k\big):\big(\sigma(u^\eps(s))e_k\big)
	\dd x\dd s. \nonumber
\end{align}
Lemma \ref{lem.v0} shows that the first term on the right-hand side can be
estimated from above by $\int_\dom h(u^0)\dd x$.
Using $w^\eps = R_\eps(v^R)=h'(u^\eps)$ and integrating by parts,
the second term on the right-hand side can be written as
\begin{align*}
  \int_0^t\big\langle &
	\diver\big(B(w^\eps)\na h'(u^\eps(s))\big),w^\eps(s)
	\big\rangle_{D(L)',D(L)}\dd s \\
	&= -\int_0^t\int_\dom \na w^\eps(s):B(w^\eps)\na w^\eps(s)\dd x\dd s \le 0.
\end{align*}
The boundary integral vanishes because of the choice of the space 
$D(L)=H^m_N(\dom)$. The last inequality
follows from Assumption (A3), which implies that 
$B(w^\eps) = A(u(w^\eps))h''(u(w^\eps))^{-1}$ is positive semidefinite..
We reformulate the last term in \eqref{4.H0} by applying
Lemma \ref{lem.ab} with $a=\sigma(u^\eps)e_k$ and
$b=\DD R_\eps[v](\sigma(u^\eps)e_k)$:
\begin{align*}
  \int_\dom & \DD R_\eps[v^R]
	\big(\sigma(u^\eps)e_k\big):\big(\sigma(u^\eps)e_k\big)\dd x \\
	&\le \int_\dom\operatorname{tr}\big[(\sigma(u^\eps)e_k)^T u'(w^\eps)^{-1}
	\sigma(u^\eps)e_k\big]\dd x.
\end{align*}
Taking the supremum in \eqref{4.H0} over $(0,T_R)$, where $T_R\le T\wedge\tau_R$, 
and the expectation yields
\begin{align}\label{4.Hest}
  \E&\sup_{0<t<T_R}\int_\dom h(u^\eps(t))\dd x 
	+ \frac{\eps}{2}\E\sup_{0<t<T_R}\|Lw^\eps\|_{L^2(\dom)}^2 \\
	&\phantom{xx}{}
	+ \E\sup_{0<t<T_R}\int_0^t\int_\dom \na w^\eps(s):B(w^\eps)\na w^\eps(s)
  \dd x\dd s - \E\int_\dom h(u^0)\dd x\nonumber \\
	&\le\E\sup_{0<t<T_R}\sum_{k=1}^\infty\sum_{i,j=1}^n\int_0^t\int_\dom
	\frac{\pa h}{\pa u_i}(u^\eps(s))\sigma_{ij}(u^\eps(s))e_k\dd x\dd W_j^k(s) 
	\nonumber \\
	&\phantom{xx}{}+ \frac12\E\sup_{0<t<T_R}\sum_{k=1}^\infty\int_0^t\int_\dom
	\operatorname{tr}\big[(\sigma(u^\eps(s))e_k)^T u'(w^\eps(s))^{-1}
	\sigma(u^\eps(s))e_k\big]\dd x\dd s \nonumber \\
	&=: I_1 + I_2. \nonumber
\end{align}
We apply the Burkholder--Davis--Gundy inequality \cite[Theorem 6.1.2]{LiRo15}
to $I_1$ and use Assumption (A5):
\begin{align*}
  I_1 &\le C\E\sup_{0<t<T_R}\bigg\{\int_0^t\sum_{k=1}^\infty\sum_{i,j=1}^n
	\bigg(\int_\dom\frac{\pa h}{\pa u_i}(u^\eps(s))\sigma_{ij}(u^\eps(s))e_k
	\dd x\bigg)^2\dd s\bigg\}^{1/2} \\
	&\le C\E\sup_{0<t<T_R}\bigg(1+\int_0^t\int_\dom h(u^\eps(s))\dd x\dd s\bigg).
\end{align*}
Also the remaining integral $I_2$ can be bounded from above by Assumption (A5):
$$
  I_2 \le C\E\sup_{0<t<T_R}\bigg(1+\int_0^t\int_\dom h(u^\eps(s))\dd x\dd s\bigg).
$$
Therefore, \eqref{4.Hest} becomes
\begin{align}\label{4.Hest2}
  \E&\sup_{0<t<T_R}\int_\dom h(u^\eps(t))\dd x 
	+ \frac{\eps}{2}\E\sup_{0<t<T_R}\|Lw^\eps\|_{L^2(\dom)}^2 \\
	&\phantom{xx}{}
	+ \E\sup_{0<t<T_R}\int_0^t\int_\dom \na w^\eps(s):B(w^\eps)\na w^\eps(s)
  \dd x\dd s - \E\int_\dom h(u^0)\dd x\nonumber \\
	&\le C\E\sup_{0<t<T_R}\bigg(1+\int_0^t\int_\dom h(u^\eps(s))\dd x\dd s\bigg) 
	\nonumber \\
	&\le C + C\E\int_0^{T_R}
	\int_\dom\sup_{0<s<t}h(u^\eps(s))\dd x\dd t. \nonumber
\end{align}
We apply Gronwall's lemma to the function $F(t)=\sup_{0<s<t}\int_\dom h(u^\eps(s))\dd x$
to find that
$$
  \E\sup_{0<t<T_R}\int_\dom h(u^\eps(t))\dd x \le C(u^0,T).
$$
Using this bound in \eqref{4.Hest2} then finishes the proof.
\end{proof}

The entropy inequality allows us to extend the local solution to a global one.

\begin{proposition}
Let $\blue{(v^\eps,\tau_R)}$ be a local solution to \eqref{4.defsol}--\eqref{4.bc},
constructed in Theorem \ref{thm.approx}. Then $v^\eps$ can be extended to 
a global solution to \eqref{4.defsol}--\eqref{4.bc}.
\end{proposition}

\begin{proof}
With the notation $u^\eps = u(R_\eps(v^\eps))$ and $w^\eps=R_\eps(v^\eps)$, 
we observe that $v^\eps=Q_\eps(R_\eps(v^\eps))$ 
$= u(R_\eps(v^\eps))+\eps L^*LR_\eps(v^\eps) = u^\eps + \eps L^*Lw^\eps$.
Thus, we have for $T_R\le T\wedge\tau_R$,
\begin{align*}
  \E&\sup_{0<t<T_R}\|v^\eps(t)\|_{D(L)'}
	\le \E\sup_{0<t<T_R}\|u^\eps\|_{D(L)'} 
	+ \eps\E\sup_{0<t<T_R}\|L^*Lw^\eps(t)\|_{D(L)'} \\
	&\le C\E\sup_{0<t<T_R}\|u^\eps\|_{L^1(\dom)}
	+ \eps \E\sup_{0<t<T_R}\|L^*Lw^\eps(t)\|_{D(L)'}.
\end{align*}
We know from Hypothesis (H1) that $|u^\eps|\le C(1+h(u^\eps))$. Therefore,
taking into account the entropy inequality and the second inequality
in \eqref{3.LL},
$$
  \E\sup_{0<t<T_R}\|v(t)\|_{D(L)'} 
	\le C\E\sup_{0<t<T_R}\|h(u^\eps(t))\|_{L^1(\dom)} + \eps C\sup_{0<t<T_R}
	\|Lw^\eps(t)\|_{L^2(\dom)} \le C(u^0,T).
$$
This allows us to perform the limit $R\to\infty$ and to conclude that we have
indeed a solution $v^\eps$ in $(0,T)$ for any $T>0$.
\end{proof}

%%%%%%%%%%%%%%%%%%%%%%%%%%%%%%%%%%%%%%%%%%%%%%%%%%%%%%%%%%%%%%%%%%%%%%%%%%%%%%

\section{Proof of Theorem \ref{thm.skt}}\label{sec.skt}

We prove the global existence of martingale solutions to the SKT model
with self-diffusion.

\subsection{Uniform estimates}

Let $v^\eps$ be a global solution to \eqref{4.defsol}--\eqref{4.bc}
and set $u^\eps=u(R_\eps(v^\eps))$. We assume that $A(u)$ is given by
\eqref{1.SKT} and that $a_{ii}>0$ for $i=1,\ldots,n$.
We start with some uniform estimates, which are a consequence
of the entropy inequality in Proposition \ref{prop.ent}.

\begin{lemma}[Uniform estimates]\label{lem.est}
There exists a constant $C(u^0,T)>0$ such that for all $\eps>0$ and
$i,j=1,\ldots,n$ with $i\neq j$,
\begin{align}
  \E\|u_i^\eps\|_{L^\infty(0,T;L^1(\dom))} &\le C(u^0,T), \label{5.L1} \\
	a_{i0}^{1/2}\E\|(u_i^\eps)^{1/2}\|_{L^2(0,T;H^1(\dom))}
	+ a_{ii}^{1/2}\E\|u_i^\eps\|_{L^2(0,T;H^1(\dom))} &\le C(u^0,T), \label{5.H1} \\
	a_{ij}^{1/2}\E\|\na(u_i^\eps u_j^\eps)^{1/2}\|_{L^2(0,T;L^2(\dom))} &\le C(u^0,T).
	\nonumber %\label{5.nabla}
\end{align}
Moreover, we have the estimate
\begin{equation}\label{5.v}
  \eps\E\|LR_\eps(v^\eps)\|_{L^\infty(0,T;L^2(\dom))}^2
	+ \E\|v^\eps\|_{L^\infty(0,T;D(L)')}^2 \le C(u^0,T).
\end{equation}
\end{lemma}

\begin{proof}
Let $v^\eps$ be a global solution to \eqref{4.defsol}--\eqref{4.bc}.
We observe that $R_\eps(v^\eps) = h'(u(R_\eps(v^\eps))) = h'(u^\eps)$
implies that $\na R_\eps(v^\eps) = h''(u^\eps)\na u^\eps$.
It is shown in \cite[Lemma 4]{CDJ18}
that for all $z\in\R^n$ and $u\in(0,\infty)^n$,
$$
  z^T h''(u)A(u)z \ge \sum_{i=1}^n\pi_i\bigg(a_{0i}\frac{z_i^2}{u_i} 
	+ 2a_{ii}z_i^2\bigg) + \frac12\sum_{i,j=1,\,i\neq j}^n
	\pi_i a_{ij}\bigg(\sqrt{\frac{u_j}{u_i}}z_i + \sqrt{\frac{u_i}{u_j}}z_j\bigg)^2.
$$
Using $B(R_\eps(v^\eps))=A(u^\eps)h''(u^\eps)^{-1}$ 
and the previous inequality with $z=\na u^\eps$, we find that
\begin{align}\label{5.RBR}
  \na R_\eps(v^\eps)&:B(R_\eps(v^\eps))\na R_\eps(v)
	= \na u^\eps:h''(u^\eps)\big(A(u^\eps)h''(u^\eps)^{-1}\big)h''(u^\eps)\na u^\eps \\
	&= \na u^\eps:h''(u^\eps)A(u^\eps)\na u^\eps \nonumber \\
	&\ge \sum_{i=1}^n\pi_i\big(4a_{0i}|\na(u^\eps)^{1/2}|^2 + 2a_{ii}|\na u^\eps|^2\big)
	+ 2\sum_{i\neq j}\pi_i a_{ij}|\na (u^\eps_i u^\eps_j)^{1/2}|^2. \nonumber
\end{align}
Therefore, the entropy inequality in Proposition \ref{prop.ent} becomes
\begin{align}\label{5.ei}
  \E\sup_{0<t<T}\int_\dom h(u^\eps(t))\dd x 
	&+ \E\sup_{0<t<T}\frac{\eps}{2}\|LR(v^\eps(t))\|_{L^2(\dom)}^2 \\
	&+ \E\int_0^T\int_\dom\sum_{i=1}^n\pi_i\big(4a_{0i}|\na(u^\eps)^{1/2}|^2 
	+ 2a_{ii}|\na u^\eps|^2\big)\dd x\dd s \nonumber \\
	&+ 2\E\int_0^T\int_\dom\sum_{i\neq j}\pi_i 
	a_{ij}|\na (u^\eps_i u^\eps_j)^{1/2}|^2\dd x\dd s
	\le C(u^0,T). \nonumber
\end{align}
This is the stochastic analog of the entropy inequality \eqref{1.ei}.
By Hypothesis (H1), we have $|u|\le C(1+h(u))$ and consequently,
$$
  \E\sup_{0<t<T}\|u^\eps(t)\|_{L^1(\dom)}
	\le C\E\sup_{0<t<T}\int_\dom h(u^\eps(t))\dd x + C \le C(u^0,T),
$$
which proves \eqref{5.L1}. Estimate \eqref{5.H1} then follows from the
Poincar\'e--Wirtinger inequality. 

It remains to show estimate \eqref{5.v}. We deduce from the second inequality in
\eqref{3.LL} that
\begin{align*}
  \|v^\eps(t)\|_{D(L)'} &= \|Q_\eps(R_\eps(v^\eps(t)))\|_{D(L)'}
	= \|u(R_\eps(v^\eps(t))) + \eps L^*LR_\eps(v^\eps(t))\|_{D(L)'} \\
	&\le C\|u(R_\eps(v^\eps(t)))\|_{L^1(\dom)} 
	+ \eps \|L^*LR_\eps(v^\eps(t))\|_{D(L)'} \\
	&\le  C\|u^\eps(t)\|_{L^1(\dom)} 
	+ \eps C\|LR_\eps(v^\eps(t))\|_{L^2(\dom)}.
\end{align*}
This shows that
$$
  \E\sup_{0<t<T}\|v^\eps(t)\|_{D(L)'}
	\le C\E\sup_{0<t<T}\|u^\eps\|_{L^1(\dom)} 
	+ \eps C\E\sup_{0<t<T}\|LR_\eps(v^\eps(t))\|_{L^2(\dom)} \le C(u^0,T),
$$
ending the proof.
\end{proof}

We also need higher-order moment estimates.

\begin{lemma}[Higher-order moments I]\label{lem.hom}
Let $p\ge 2$.
There exists a constant $C(p,u^0,T)$, which is independent of $\eps$, such that
\begin{align}
  \E\|u^\eps\|_{L^\infty(0,T;L^1(\dom))}^p &\le C(p,u^0,T), \label{5.L1p} \\
	a_{i0}^{p/2}\E\|(u_i^\eps)^{1/2}\|_{L^2(0,T;H^1(\dom))}^p
	+ a_{ii}^{p/2}\E\|u_i^\eps\|_{L^2(0,T;H^1(\dom))}^p &\le C(p,u^0,T), \label{5.H1p} \\
	a_{ij}^{p/2}\E\|\na(u_i^\eps u_j^\eps)^{1/2}\|_{L^2(0,T;L^2(\dom))}^p &\le C(p,u^0,T).
	\label{5.nablap}
\end{align}
Moreover, we have
\begin{equation}
  \E\bigg(\eps\sup_{0<t<T}\|LR_\eps(v^\eps(t))\|_{L^2(\dom)}^2\bigg)^p
	+ \E\bigg(\sup_{0<t<T}\|v^\eps(t)\|_{D(L)'}\bigg)^p \le C(p,u^0,T). \label{5.vp}
\end{equation}
\end{lemma}

\begin{proof}
Proceeding as in the proof of Proposition \ref{prop.ent}
and taking into account identity \eqref{4.H0} and inequality \eqref{5.RBR}, we obtain
\begin{align*}
  {\mathcal H}(v^\eps(&t)) 
	+ \int_0^T\int_\dom\sum_{i=1}^n\pi_i\big(4a_{i0}|\na(u^\eps)^{1/2}|^2 
	+ 2a_{ii}|\na u^\eps|^2\big)\dd x\dd s \\
	&\phantom{xx}{}+ 2\E\int_0^T\int_\dom\sum_{i\neq j}\pi_i a_{ij}
	|\na (u^\eps_i u^\eps_j)^{1/2}|^2\dd x\dd s \\
	&\le {\mathcal H}(v^\eps(0)) +  \sum_{k=1}^\infty\sum_{i,j=1}^n\int_0^t\int_\dom
	\pi_i\log u^\eps_i(s)\sigma_{ij}(u^\eps(s))e_k\dd x\dd W_j^k(s) \\
	&\phantom{xx}{}
	+ \frac12\sum_{k=1}^\infty\int_0^t\int_\dom\operatorname{tr}
	\big[(\sigma(u^\eps(s))e_k)^T h''(u^\eps(s))\sigma(u^\eps(s))e_k\big]\dd x\dd s,
\end{align*}
recalling Definition \ref{3.ent} of ${\mathcal H}(v^\eps)$. 
We raise this inequality to the $p$th power, take the expectation, apply
the Burkholder--Davis--Gundy inequality (for the second term on the right-hand side),
and use Assumption (A5) to find that
\begin{align}\label{5.est3}
  \E&\bigg(\sup_{0<t<T}\int_\dom h(u^\eps(t))\dd x
	+ \eps\sup_{0<t<T}\|LR_\eps(v^\eps(t))\|_{L^2(\dom)}^2\bigg)^p \\
	&\phantom{xx}{}+ C\E\bigg(\int_0^T\int_\dom\sum_{i=1}^n\pi_i a_{i0}
	|\na(u_i^\eps(s))^{1/2}|^2\dd x\dd s\bigg)^p \nonumber \\
	&\phantom{xx}{}+ C\E\bigg(\int_0^T\int_\dom\sum_{i=1}^n\pi_i a_{ii}
	|\na u_i^\eps(s)|^2\dd x\dd s\bigg)^p \nonumber \\
	&\phantom{xx}{}+ C\E\bigg(\int_0^T\int_\dom\sum_{i\neq j}\pi_i a_{ij}
	|\na (u^\eps_i u^\eps_j)^{1/2}|^2\dd x\dd s\bigg)^p \nonumber \\
	&\le C(p,u^0) + C\E\bigg(\int_0^T\sum_{k=1}^\infty\sum_{i,j=1}^n
	\bigg(\int_\dom\log u^\eps_i(s)\sigma_{ij}(u^\eps(s))e_k\dd x\bigg)^2
	\dd s\bigg)^{p/2} \nonumber \\
	&\phantom{xx}{}+ C\E\bigg(\int_0^T\sum_{k=1}^\infty\int_\dom\operatorname{tr}
	\big[(\sigma(u^\eps(s))e_k)^T h''(u^\eps(s))(\sigma(u^\eps(s))e_k)\big]\dd x\dd s
	\bigg)^p \nonumber \\
	&\le C(p,u^0) + C\E\bigg(\int_0^T\int_\dom h(u^\eps(s))\dd x\dd s\bigg)^p.
	\nonumber 
\end{align}
We neglect the expression $\eps\|LR_\eps(v^\eps(t))\|_{L^2(\dom)}^2$ and
apply Gronwall's lemma. Then, taking into account the fact that the entropy
dominates the $L^1(\dom)$ norm, thanks to Hypothesis (H1), 
and applying the Poincar\'e--Wirtinger inequality,
we obtain estimates \eqref{5.L1p}--\eqref{5.nablap}. Going back to \eqref{5.est3},
we infer that
\begin{align*}
  \E\bigg(\eps\sup_{0<t<T}\|LR_\eps(v^\eps(t))\|_{L^2(\dom)}^2\bigg)^p
	&\le C(p,u^0) + C(p,T)\E\int_0^T\bigg(\int_\dom h(u^\eps(s))\dd x\bigg)^p \dd s \\
	&\le C(p,u^0,T).
\end{align*}
Combining the previous estimates and arguing as in the proof of Lemma \ref{lem.est},
we have
\begin{align*}
  \E&\bigg(\sup_{0<t<T}\|v^\eps(t)\|_{D(L)'}\bigg)^p
	= \E\bigg(\sup_{0<t<T}\|u^\eps(t)+\eps L^*LR_\eps(v^\eps(t))\|_{D(L)'}\bigg)^p \\
	&\le C\E\bigg(\sup_{0<t<T}\|u^\eps(t)\|_{L^1(\dom)}\bigg)^p
	+ C\E\bigg(\eps^2\sup_{0<t<T}\|LR_\eps(v^\eps(t))\|_{L^2(\dom)}^2\bigg)^{p/2}
	\le C(p,u^0,T).
\end{align*}
This ends the proof.
\end{proof}

Using the Gagliardo--Nirenberg inequality, we can derive further estimates. 
We recall that $Q_T=\dom\times(0,T)$.

\begin{lemma}[Higher-order moments II]\label{lem.hom2}
Let $p\ge 2$. There exists a constant
$C(p,u^0,T)$ $>0$, which is independent of $\eps$, such that
\begin{align}
  \E\|u_i^\eps\|_{L^{2+2/d}(Q_T)}^p &\le C(p,u^0,T), \label{5.22d} \\
	\E\|u_i^\eps\|_{L^{2+4/d}(0,T;L^2(\dom))}^p &\le C(p,u^0,T). \label{5.24d}
\end{align}
\end{lemma}

\begin{proof}
We apply the Gagliardo--Nirenberg inequality:
\begin{align*}
  \E&\bigg(\int_0^T\|u_i^\eps\|_{L^r(\dom)}^s\dd t\bigg)^{p/s}
	\le C\E\bigg(\int_0^T\|u_i^\eps\|_{H^1(\dom)}^{\theta s}
	\|u_i^\eps\|_{L^1(\dom)}^{(1-\theta)s}\dd t\bigg)^{p/s} \\
	&\le C\E\bigg(\|u_i^\eps\|_{L^\infty(0,T;L^1(\dom))}^{(1-\theta)s}\int_0^T
	\|u_i^\eps\|_{H^1(\dom)}^{2}\dd t\bigg)^{p/s} \\
	&\le C\big(\E\|u_i^\eps\|_{L^\infty(0,T;L^1(\dom))}^{2(1-\theta)p}\big)^{1/2}
	\big(\E\|u_i^\eps\|_{L^2(0,T;H^1(\dom))}^{4p/s}\big)^{1/2} \le C,
\end{align*}
where $r>1$ and $\theta\in(0,1]$ are related by $1/r=1-\theta(d+2)/(2d)$ 
and $s=2/\theta\ge 2$. The right-hand side is
bounded in view of estimates \eqref{5.L1p} and \eqref{5.H1p}.
Estimate \eqref{5.22d} follows after choosing $r=s$, implying that $r=2+2/d$, and
\eqref{5.24d} follows from the choice $s=2+4/d$, implying that
$r=2$. 
\end{proof}

Next, we show some bounds for the fractional time derivative of $u^\eps$.
This result is used to establish the tightness of the laws of $(u^\eps)$ in a
sub-Polish space. Alternatively, the tightness property can be proved by
verifying the Aldous condition; see, e.g., \cite{DJZ19}.
We recall the definition of the Sobolev--Slobodeckij spaces. Let $X$ be a vector
space and let $p\ge 1$, $\alpha\in(0,1)$. 
Then $W^{\alpha,p}(0,T;X)$ is the set of
all functions $v\in L^p(0,T;X)$ for which 
\begin{align*}
  \|v\|_{W^{\alpha,p}(0,T;X)}^p 
	&= \|v\|_{L^p(0,T;X)}^p + |v|_{W^{\alpha,p}(0,T;X)}^p \\
	&= \int_0^T\|v\|_X^p \dd t
	+ \int_0^T\int_0^T\frac{\|v(t)-v(s)\|_X^p}{|t-s|^{1+\alpha p}}\dd t\dd s < \infty.
\end{align*}
With this norm, $W^{\alpha,p}(0,T;X)$ becomes a Banach space.
We need the following technical lemma, which is proved in Appendix \ref{sec.proofs}.

\begin{lemma}\label{lem.g}
Let $g\in L^1(0,T)$ and $\delta<2$, $\delta\neq 1$. Then
\begin{equation}\label{5.int}
  \int_0^T\int_0^T|t-s|^{-\delta}\int_{s\wedge t}^{t\lor s}g(r)\dd r\dd t\dd s 
	< \infty.
\end{equation}
\end{lemma}

We obtain the following uniform bounds for $u^\eps$ and $v^\eps$ in 
Sobolev--Slobodeckij spaces.

\begin{lemma}[Fractional time regularity]\label{lem.frac}
Let $\alpha<1/2$. 
There exists a constant $C(u^0,T)>0$ such that, for $p:=(2d+4)/d >2$,
\begin{align}
  \E\|u^\eps\|_{W^{\alpha,p}(0,T;D(L)')}^p &\le C(u^0,T), \nonumber\\ %\label{5.timeu}\\
	\eps^p\E\|L^*LR_\eps(v^\eps)\|_{W^{\alpha,p}(0,T;D(L)')}^p 
	+ \E\|v^\eps\|_{W^{\alpha,p}(0,T;D(L)')}^p &\le C(u^0,T). \label{5.LLR}
\end{align}
\end{lemma}

Since $p>2$, we can choose $\alpha<1/2$ such that $\alpha p>1$.
Then the continuous embedding $W^{\alpha,p}(0,T)\hookrightarrow C^{0,\beta}([0,T])$
for $\beta=\alpha-1/p>0$ implies that 
\begin{equation}\label{5.C0}
  \E\|u^\eps\|_{C^{0,\beta}([0,T];D(L)')}^p\le C(u^0,T).
\end{equation}

\begin{proof}
First, we derive the $W^{\alpha,p}$ estimate for $v^\eps$ and then we conclude
the estimate for $u^\eps$ from the definition 
$v^\eps = u^\eps + \eps L^*LR_\eps(v^\eps)$
and Lemma \ref{lem.hom2}.
Equation \eqref{4.approx} reads in terms of $u^\eps$ as
$$
  \dd v^\eps_i = \diver\bigg(\sum_{j=1}^n A_{ij}(u^\eps)\na u^\eps_j\bigg)\dd t
	+ \sum_{j=1}^n\sigma_{ij}(u^\eps)\dd W_j, \quad i=1,\ldots,n.
$$
We know from \eqref{5.vp} that $\E\|v^\eps\|_{L^\infty(0,T;D(L)')}^p$
is bounded. Thus, to prove the bound for the second term in \eqref{5.LLR},
it remains to estimate the following seminorm:
\begin{align*}
  \E|v_i^\eps&|_{W^{\alpha,p}(0,T;D(L)'}^p
	= \E\int_0^T\int_0^T\frac{\|v_i^\eps(t)-v_i^\eps(s)\|_{D(L)'}^p}{
	|t-s|^{1+\alpha p}}\dd t\dd s \\
  &\le \E\int_0^T\int_0^T|t-s|^{-1-\alpha p}\bigg\|\int_{s\wedge t}^{t\lor s}
	\diver\sum_{j=1}^n A_{ij}(u^\eps(r))\na u_j^\eps(r)\dd r\bigg\|_{D(L)'}^p
	\dd t\dd s \\
	&\phantom{xx}{}+ \E\int_0^T\int_0^T|t-s|^{-1-\alpha p}
	\bigg\|\int_{s\wedge t}^{t\lor s}
	\sum_{j=1}^n\sigma_{ij}(u^\eps(r))\dd W_j(r)\bigg\|_{D(L)'}^p\dd t\dd s \\
	&=: J_1 + J_2.
\end{align*}

We need some preparations before we can estimate $J_1$. We observe that
\begin{align*}
  \bigg\|\sum_{j=1}^n A_{ij}(u^\eps)\na u_j^\eps\bigg\|_{L^1(\dom)}
	&= \bigg\|\bigg(a_{i0}+2\sum_{j=1}^n a_{ij}u_j^\eps\bigg)\na u_i^\eps
	+ \sum_{j\neq i}a_{ij}u_i^\eps\na u_j^\eps\bigg\|_{L^1(\dom)} \\
	&\le C\|\na u_i^\eps\|_{L^1(\dom)} + C\|u^\eps\|_{L^2(\dom)}
	\|\na u^\eps\|_{L^2(\dom)}.
\end{align*}
It follows from the embedding $L^1(\dom)\hookrightarrow D(L)'$ that
\begin{align*}
  J_1 &\le \E\int_0^T\int_0^T|t-s|^{-1-\alpha p}\bigg(\int_{s\wedge t}^{t\lor s}
	\bigg\|\diver\sum_{j=1}^n A_{ij}(u^\eps(r))\na u_j^\eps(r)\bigg\|_{D(L)'}\dd r
	\bigg)^p\dd t\dd s \\
	&\le C\E\int_0^T\int_0^T|t-s|^{-1-\alpha p}\bigg(\int_{s\wedge t}^{t\lor s}
	\bigg\|\sum_{j=1}^nA_{ij}(u^\eps(r))\na u_j^\eps(r)\bigg\|_{L^1(\dom)}\dd r
	\bigg)^p\dd t\dd s \\
	&\le C\E\int_0^T\int_0^T|t-s|^{-1-\alpha p}\bigg(\int_{s\wedge t}^{t\lor s}
	\|\na u^\eps(r)\|_{L^2(\dom)}\dd r\bigg)^p \dd t\dd s \\
	&\phantom{xx}{}+ C\E\int_0^T\int_0^T|t-s|^{-1-\alpha p}
	\bigg(\int_{s\wedge t}^{t\lor s}
	\|u^\eps(r)\|_{L^2(\dom)}\|\na u^\eps(r)\|_{L^2(\dom)}\dd r\bigg)^p \dd t\dd s \\
	&=: J_{11} + J_{12}.
\end{align*}
We use H\"older's inequality and fix $p=(2d+4)/d$ to obtain
$$
  J_{11}\le C\E\int_0^T\int_0^T|t-s|^{-1-\alpha p}|t-s|^{p/2}
	\bigg(\int_{s\wedge t}^{t\lor s}\|\na u^\eps(r)\|_{L^2(\dom)}^2\dd r\bigg)^{p/2}
	\dd t\dd s.
$$
In view of estimate \eqref{5.H1p} and \eqref{5.int}, the right-hand side is finite if
$1+\alpha p-p/2<2$ or, equivalently, $\alpha<(d+1)/(d+2)$, and this
holds true since $\alpha<1/2$. Applying H\"older's inequality again, we have
\begin{align*}
	J_{12} &\le C\E\int_0^T\int_0^T|t-s|^{-1-\alpha p}\bigg(
	\int_{s\wedge t}^{t\lor s}\|u^\eps(r)\|_{L^2(\dom)}^2\dd r\bigg)^{p/2}
	\bigg(\int_{s\wedge t}^{t\lor s}\|\na u^\eps(r)\|_{L^2(\dom)}^2\dd r\bigg)^{p/2}
	\dd t \dd s \\
	&\le C\E\int_0^T\int_0^T|t-s|^{-1-\alpha p}|t-s|^{p/(d+2)}\bigg(
	\int_{s\wedge t}^{t\lor s}\|u^\eps(r)\|_{L^2(\dom)}^{(2d+4)/d}\dd r
	\bigg)^{pd/(2d+4)} \\
	&\phantom{xx}{}\times\bigg(\int_{s\wedge t}^{t\lor s}\|\na u^\eps(r)\|_{L^2(\dom)}^2
	\dd r\bigg)^{p/2}\dd t\dd s \\
	&\le C\bigg\{\E\bigg(\int_0^T\int_0^T|t-s|^{-1-\alpha p+p/(d+2)}
	\bigg(\int_{s\wedge t}^{t\lor s}\|u^\eps(r)\|_{L^2(\dom)}^{(2d+4)/d}\dd r\bigg)
	\dd t\dd s\bigg)^2\bigg\}^{1/2}\\
	&\phantom{xx}{}\times\bigg\{\E\bigg(\int_0^T\|\na u^\eps(r)
	\|_{L^2(\dom)}^2\dd r\bigg)^{p}\bigg\}^{1/2}.
\end{align*}
Because of estimates \eqref{5.H1p}, \eqref{5.24d}, and \eqref{5.int}, 
the right-hand side of is finite if $1+\alpha p-p/(d+2)<2$, which is equivalent to
$\alpha<1/2$. 

To estimate $J_2$, we use the embedding $L^2(\dom)\hookrightarrow D(L)'$, the
Burkholder--Davis--Gundy inequality, the linear growth of $\sigma$ from
Assumption (A4), and the H\"older inequality:
\begin{align*}
  J_2 &\le C\int_0^T\int_0^T|t-s|^{-1-\alpha p}
	\E\bigg\|\int_{s\wedge t}^{t\lor s}
	\sum_{j=1}^n\sigma_{ij}(u^\eps(r))\dd W_j(r)\bigg\|_{L^2(\dom)}^p\dd t\dd s \\
	&\le C\int_0^T\int_0^T|t-s|^{-1-\alpha p}
	\E\bigg(\int_{s\wedge t}^{t\lor s}\sum_{k=1}^\infty\sum_{j=1}^n
	\|\sigma_{ij}(u^\eps(r))e_k\|_{L^2(\dom)}^2\dd r
	\bigg)^{p/2}\dd t \dd s \\
	&\le C\int_0^T\int_0^T|t-s|^{-1-\alpha p+(p-2)/2}\int_{s\wedge t}^{t\lor s}\E
	\sum_{j=1}^n\big(1+\|u_j^\eps(r)\|_{L^2(\dom)}^p\big)\dd r\dd t\dd s.
\end{align*}
By \eqref{5.24d} and \eqref{5.int}, the right-hand side is finite if
$1+\alpha p-(p-2)/2<2$, which is equivalent to $\alpha<(3d+2)/(2d+4)$,
and this is valid due to the condition $\alpha<1/2$.
We conclude that $(v^\eps)$ is bounded in $L^p(\Omega;W^{\alpha,p}(0,T;D(L)'))$
with $p=(2d+4)/d$. 

Next, we derive the uniform bounds for $u^\eps$. By definition of $v^\eps$
and the $W^{\alpha,p}$ seminorm,
\begin{align*}
  \E|u^\eps|_{W^{\alpha,p}(0,T;D(L)')}^p
	&= \E|v^\eps - \eps L^*LR_\eps(v^\eps)|_{W^{\alpha,p}(0,T;D(L)')}^p \\
	&\le C\E\int_0^T\int_0^T\frac{\|v^\eps(t)-v^\eps(s)\|_{D(L)'}^p}{|t-s|^{1+\alpha p}}
	\dd t\dd s \\
	&\phantom{xx}{}+ C\E\int_0^T\int_0^T	\frac{\eps^p\|L^*LR_\eps(v^\eps(t))
	-L^*LR_\eps(v^\eps(s))\|_{D(L)'}^p}{|t-s|^{1+\alpha p}}\dd t\dd s.
\end{align*}
It follows from \eqref{3.LL} 
and the Lipschitz continuity of $R_\eps$ (Lemma \ref{lem.Reps}) that
\begin{align*}
  \|L^*LR_\eps(v^\eps(t))-L^*LR_\eps(v^\eps(s))\|_{D(L)'}
	&\le \|R_\eps(v^\eps(t))-R_\eps(v^\eps(s))\|_{L^2(\dom)} \\
	&\le \eps^{-1}C\|v^\eps(t)-v^\eps(s)\|_{D(L)'}.
\end{align*}
Then we find that
$$
  \E|u^\eps|_{W^{\alpha,p}(0,T;D(L)')}^p
	\le C\E\int_0^T\int_0^T\frac{\|v^\eps(t)-v^\eps(s)\|_{D(L)'}^p}{|t-s|^{1+\alpha p}}
	\dd t\dd s = C\E|v^\eps|_{W^{\alpha,p}(0,T;D(L)')},
$$
which finishes the proof.
\end{proof}

%%%%%%%%%%%%%%%%%%%%%

\subsection{Tightness of the laws of $(u^\eps)$}\label{sec.tight}

We show that the laws of $(u^\eps)$ are tight in a certain sub-Polish space. 
For this, we introduce the following spaces:
\begin{itemize}
\item $C^0([0,T];D(L)')$ is the space of continuous functions $u:[0,T]\to D(L)'$
with the topology $\mathbb{T}_1$ induced by the norm $\|u\|_{C^0([0,T];D(L)')}
=\sup_{0<t<T}\|u(t)\|_{D(L)'}$;
\item $L_w^2(0,T;H^1(\dom))$ is the space $L^2(0,T;H^1(\dom))$ with the weak topology
$\mathbb{T}_2$.
\end{itemize}
We define the space
$$
  \widetilde{Z}_T := C^0([0,T];D(L)')\cap L_w^2(0,T;H^1(\dom)),
$$
endowed with the topology $\widetilde{\mathbb{T}}$ that is the maximum of the
topologies $\mathbb{T}_1$ and $\mathbb{T}_2$. The space $\widetilde{Z}_T$
is a sub-Polish space, since $C^0([0,T];D(L)')$ is 
separable and metrizable and 
$$
  f_m(u) = \int_0^T(u(t),v_m(t))_{H^1(\dom)}\dd t, \quad
	u\in L_w^2(0,T;H^1(\dom)),\ m\in\N,
$$
where $(v_m)_m$ is a dense subset of $L^2(0,T;H^1(\dom))$,
is a countable family $(f_m)$ of
point-separating functionals acting on $L^2(0,T;H^1(\dom))$.
In the following, we choose a number $s^*\ge 1$ such that
\begin{equation}\label{5.sstar}
  s^* < \frac{2d}{d-2}\quad\mbox{if }d\ge 3, \quad s^*<\infty\quad\mbox{if }d=2, \quad
	s^* \le \infty\quad\mbox{if }d=1.
\end{equation}
Then the embedding $H^1(\dom)\hookrightarrow L^{s^*}(\dom)$ is compact.

\begin{lemma}\label{lem.tight1}
The set of laws of $(u^\eps)$ is tight in 
$$
  Z_T = \widetilde{Z}_T\cap L^2(0,T;L^{s^*}(\dom))
$$
with the topology $\mathbb{T}$ that is the maximum of $\widetilde{\mathbb{T}}$ 
and the topology induced by the $L^2(0,T;$ $L^{s^*}(\dom))$ norm, where
$s^*$ is given by \eqref{5.sstar}. 
\end{lemma}

\begin{proof}
We apply Chebyshev's inequality for the first moment and use estimate
\eqref{5.C0} with $\beta=\alpha-1/p>0$, for any $\eta>0$ and $\delta>0$,
\begin{align*}
  \sup_{\eps>0}\,&\Prob\bigg(\sup_{\substack{s,t\in[0,T], \\ |t-s|\le\delta}}
	\|u^\eps(t)-u^\eps(s)\|_{D(L)'}>\eta\bigg)
	\le \sup_{\eps>0}\frac{1}{\eta}\E\bigg(
	\sup_{\substack{s,t\in[0,T], \\ |t-s|\le\delta}}
	\|u^\eps(t)-u^\eps(s)\|_{D(L)'}\bigg) \\
	&\le \frac{\delta^\beta}{\eta}\sup_{\eps>0}\E\bigg(\sup_{\substack{s,t\in[0,T], \\ 
	|t-s|\le\delta}}\frac{\|u^\eps(t)-u^\eps(s)\|_{D(L)'}}{|t-s|^\beta}\bigg)
	\le \frac{\delta^\beta}{\eta}\sup_{\eps>0}\E\|u^\eps\|_{C^{0,\beta}([0,T];D(L)'))}
	\le C\frac{\delta^\beta}{\eta}.
\end{align*}
This means that for all $\theta>0$ and
all $\eta>0$, there exists $\delta>0$ such that
$$
  \sup_{\eps>0}\,\Prob\bigg(\sup_{s,t\in[0,T], \, |t-s|\le\delta}
	\|u^\eps(t)-u^\eps(s)\|_{D(L)'}>\eta\bigg) \le \theta,
$$
which is equivalent to the Aldous 
condition \cite[Section 2.2]{BrMo14}. Applying \cite[Lemma 5, Theorem 3]{Sim87}
with the spaces $X=H^1(\dom)$ and $B=D(L)'$, we conclude that
$(u^\eps)$ is precompact in $C^0([0,T];D(L)')$. 
Then, proceeding as in the proof of the basic criterion for tightness 
\cite[Chapter II, Section 2.1]{Met88}, we see that 
the set of laws of $(u^\eps)$ is tight in $C^0([0,T];D(L)')$.

Next, by Chebyshev's inequality again and estimate \eqref{5.H1}, for all $K>0$,
$$
  \Prob\big(\|u^\eps\|_{L^2(0,T;H^1(\dom))} > K\big)
	\le \frac{1}{K^2}\E\|u^\eps\|_{L^2(0,T;H^1(\dom))}^2 \le \frac{C}{K^2}.
$$
This implies that for any $\delta>0$, there exists $K>0$ such that
$\Prob(\|u^\eps\|_{L^2(0,T;H^1(\dom))}\le K)\le 1-\delta$. Since closed balls with
respect to the norm of $L^2(0,T;H^1(\dom))$ are weakly compact, we infer that
the set of laws of $(u^\eps)$ is tight in $L^2_w(0,T;H^1(\dom))$. 

The tightness in $L^2(0,T;L^{s^*}(\dom))$ follows from  Lemma \ref{lem.tight} 
in Appendix \ref{sec.aux} with $p=q=2$ and $r=2+4/d$. 
\end{proof}

\begin{lemma}\label{lem.tighteps}
The set of laws of $(\sqrt{\eps}L^*LR_\eps(v^\eps))$ is tight in 
$$
  Y_T:= L_w^2(0,T;D(L)')\cap L^\infty_{w*}(0,T;D(L)')
$$ 
with the associated topology $\mathbb{T}_Y$. 
\end{lemma}

\begin{proof}
We apply the Chebyshev inequality and use the inequality
$\|L^*LR_\eps(v^\eps)\|_{D(L)'}\le C\|LR_\eps(v^\eps)\|_{L^2(\dom)}$ and
estimate \eqref{5.v}:
$$
  \Prob\big(\sqrt{\eps}\|L^*LR_\eps(v^\eps)\|_{L^2(0,T;D(L)')}>K\big)
	\le \frac{\eps}{K^2}\E\|L^*LR_\eps(v^\eps)\|_{L^2(0,T;D(L)')}^2
	\le \frac{C}{K^2}
$$
for any $K>0$. Since closed balls in $L^2(0,T;D(L)')$ are weakly compact, 
the set of laws of $(\sqrt{\eps}L^*LR_\eps(v^\eps))$ is tight in $L_w^2(0,T;D(L)')$.
The second claim follows from an analogous argument.
%$$
%  \Prob\big(\sqrt{\eps}\|L^*LR_\eps(v^\eps)\|_{L^\infty(0,T;D(L)')}>K\big)
%	\le \frac{\eps}{K^2}\E\|LR_\eps(v^\eps)\|_{L^\infty(0,T;L^2(\dom))}^2
%	\le \frac{C}{K^2}.
%$$
\end{proof}

%%%%%%%%%%%%%%%%%%%%%

\subsection{Convergence of $(u^\eps)$}\label{sec.conv1}

Let $\mbox{P}(X)$ be the space of probability measures on $X$.
We consider the space $Z_T\times Y_T\times C^0([0,T];U_0)$, equipped with the
probability measure $\mu^\eps:=\mu_u^\eps\times\mu_w^\eps\times\mu_W^\eps$, where
\begin{align*}
  \mu_u^\eps(\cdot) &= \Prob(u^\eps\in\cdot)\in \mbox{P}(Z_T), \\
	\mu_w^\eps &= \Prob(\sqrt{\eps} L^*LR_\eps(v^\eps)\in\cdot)\in
	\mbox{P}(Y_T), \\
	\mu_W^\eps(\cdot) &= \Prob(W\in\cdot) \in \mbox{P}(C^0([0,T];U_0)),
\end{align*}
recalling the choice \eqref{5.sstar} of $s^*$.
The set of measures $(\mu^\eps)$ is tight, since the set of laws of
$(u^\eps)$ and $(\sqrt{\eps}L^*LR_\eps(v^\eps))$ are tight in 
$(Z_T,\mathbb{T})$ and $(Y_T,\mathbb{T}_Y)$, respectively.
Moreover, $(\mu_W^\eps)$ consists of one element only and is consequently weakly
compact in $C^0([0,T];U_0)$. By \blue{Prokhorov's} theorem, $(\mu_W^\eps)$ is tight. Hence,
$Z_T\times Y_T\times C^0([0,T];U_0)$ satisfies the assumptions of the
Skorokhod--Jakubowski theorem \cite[Theorem C.1]{BrOn10}. We infer that there
exists a subsequence of $(u^\eps,\sqrt{\eps}L^*LR_\eps(v^\eps))$, 
which is not relabeled, a probability
space $(\widetilde{\Omega},\widetilde{\mathcal F},\widetilde{\Prob})$ and, 
on this space, $(Z_T\times Y_T\times C^0([0,T];U_0))$-valued random variables
$(\widetilde{u},\widetilde{w},\widetilde{W})$ and $(\widetilde{u}^\eps,
\widetilde{w}^\eps,\widetilde{W}^\eps)$ such that $(\widetilde{u}^\eps,
\widetilde{w}^\eps,\widetilde{W}^\eps)$ has the same law as
$(u^\eps,\sqrt{\eps}L^*LR_\eps(v^\eps),W)$ on ${\mathcal B}(Z_T\times Y_T\times
C^0([0,T];U_0))$ and, as $\eps\to 0$,
$$
  (\widetilde{u}^\eps,\widetilde{w}^\eps,\widetilde{W}^\eps)\to 
	(\widetilde{u},\widetilde{w},\widetilde{W})\quad\mbox{in }
	Z_T\times Y_T\times C^0([0,T];U_0)\quad \widetilde{\Prob}\mbox{-a.s.}
$$
By the definition of $Z_T$ and $Y_T$, this convergence means 
$\widetilde{\Prob}$-a.s., 
\begin{align*}
  \widetilde{u}^\eps\to \widetilde{u} &\quad\mbox{strongly in }C^0([0,T];D(L)'), \\
	\widetilde{u}^\eps\rightharpoonup \widetilde{u} &\quad\mbox{weakly in }
	L^2(0,T;H^1(\dom)), \\
	\widetilde{u}^\eps\to\widetilde{u} &\quad\mbox{strongly in }
	L^2(0,T;L^{s^*}(\dom)), \\
	\widetilde{w}^\eps\rightharpoonup\widetilde{w} &\quad\mbox{weakly in }
	L^2(0,T;D(L)'), \\
	\widetilde{w}^\eps\rightharpoonup\widetilde{w} &\quad\mbox{weakly* in }
	L^\infty(0,T;D(L)'), \\
	\widetilde{W}^\eps\to\widetilde{W} &\quad\mbox{strongly in }C^0([0,T];U_0).
\end{align*}

We derive some regularity properties for the limit $\widetilde{u}$. We note that
$\widetilde{u}$ is a $Z_T$-Borel random variable, since
${\mathcal B}(Z_T\times Y_T\times C^0([0,T];U_0))$ is a subset of
${\mathcal B}(Z_T)\times{\mathcal B}(Y_T)\times{\mathcal B}(C^0([0,T];U_0))$.
We deduce from estimates \eqref{5.L1} and \eqref{5.H1} and the fact that
$u^\eps$ and $\widetilde{u}^\eps$ have the same law that
$$
  \sup_{\eps>0}\widetilde{\E}\|\widetilde{u}^\eps\|_{L^2(0,T;H^1(\dom))}^p
	+ \sup_{\eps>0}\widetilde{\E}\|\widetilde{u}^\eps\|_{L^\infty(0,T;D(L)')}^p 
	< \infty.
$$
We infer the existence of a further subsequence of $(\widetilde{u}^\eps)$
(not relabeled) that is weakly converging in 
$L^p(\widetilde{\Omega};L^2(0,T;H^1(\dom)))$ and weakly* converging in
$L^p(\widetilde{\Omega};C^0([0,T];D(L)'))$ as $\eps\to 0$. 
Because $\widetilde{u}^\eps\to\widetilde{u}$ in $Z_T$ $\widetilde{\Prob}$-a.s.,
we conclude that the limit function satisfies
$$
  \widetilde{\E}\|\widetilde{u}\|_{L^2(0,T;H^1(\dom))}^p
	+ \widetilde{\E}\|\widetilde{u}\|_{L^\infty(0,T;D(L)')}^p < \infty.
$$
Let $\widetilde{\F}$ and $\widetilde{\F}^\eps$ be the filtrations generated by
$(\widetilde{u},\widetilde{w},\widetilde{W})$ and 
$(\widetilde{u}^\eps,\widetilde{w}^\eps,\widetilde{W})$, respectively.
By following the arguments of the proof of \cite[Proposition B4]{BrOn13},
we can verify that these new random variables induce actually stochastic processes.
The progressive measurability of $\widetilde{u}^\eps$ is a consequence of
\cite[Appendix B]{BHM13}. Set 
$\widetilde{W}_j^{\eps,k}(t):=(\widetilde{W}^\eps(t),e_k)_U$.
We claim that $\widetilde{W}_j^{\eps,k}(t)$ for $k\in\N$
are independent, standard $\widetilde{\mathcal F}_t$-Wiener processes.
The adaptedness is a direct consequence of the definition; the independence 
of $\widetilde{W}_j^{\eps,k}(t)$ and the independence of the increments
$\widetilde{W}^{\eps,k}(t)-\widetilde{W}^{\eps,k}(s)$ with respect to
$\widetilde{\mathcal F}_s$ are inherited from $(W(t),e_k)_U$. Passing to the
limit $\eps\to 0$ in the characteristic function, by using dominated convergence,
we find that $\widetilde{W}(t)$ are $\widetilde{\mathcal F}_t$-martingales
with the correct marginal distributions. We deduce from L\'evy's characterization
theorem that $\widetilde{W}(t)$ is indeed a cylindrical Wiener process.

By definition, $u_i^\eps=u_i(R_\eps(v^\eps))=\exp(R_\eps(v^\eps))$ is positive in 
$Q_T$ a.s. We claim that also $\widetilde{u}_i$ is nonnegative in $\dom$ a.s.

\begin{lemma}[Nonnegativity]
It holds that $\widetilde{u}_i\ge 0$ a.e.\ in $Q_T$
$\widetilde{\Prob}$-a.s. for all $i=1,\ldots,n$.
\end{lemma}

\begin{proof}
Let $i\in\{1,\ldots,n\}$. Since $u_i^\eps > 0$ in $Q_T$ a.s., we have
$\E\|(u_i^\eps)^-\|_{L^2(0,T;L^2(\dom))}=0$, where $z^-=\min\{0,z\}$.
The function $u_i^\eps$ is $Z_T$-Borel measurable and so does its negative part.
Therefore, using the equivalence of the laws of $u_i^\eps$ and $\widetilde{u}_i^\eps$
in $Z_T$ and writing $\mu_i^\eps$ and $\widetilde{\mu}_i^\eps$ for the laws of
$u_i^\eps$ and $\widetilde{u}_i^\eps$, respectively, we obtain
\begin{align*}
  \widetilde{\E}\|(\widetilde{u}_i^\eps)^-\|_{L^2(Q_T)}
	&= \int_{L^2(Q_T)}\|y^-\|_{L^2(Q_T)}\dd\widetilde{\mu}_i^\eps(y)\\
	&= \int_{L^2(Q_T)}\|y^-\|_{L^2Q_T)}\dd\mu_i^\eps(y)
	= \E\|u_i^\eps\|_{L^2(Q_T)} = 0.
\end{align*}
This shows that $\widetilde{u}_i^\eps\ge 0$ a.e.\ in $Q_T$ $\widetilde{\Prob}$-a.s.
The convergence (up to a subsequence) $\widetilde{u}^\eps\to \widetilde{u}$
a.e.\ in $Q_T$ $\widetilde{\Prob}$-a.s.\ then implies that $\widetilde{u}_i\ge 0$
in $Q_T$ $\widetilde{\Prob}$-a.s.
\end{proof}

The following lemma is needed to verify that $(\widetilde{u},\widetilde{W})$ 
is a martingale solution to \eqref{1.eq}--\eqref{1.bic}.
%In view of the previous convergence results, the proof is very similar to that 
%one of \cite[Lemma 10]{DJZ19} and therefore, we omit it.

\begin{lemma}%\label{lem.E}
It holds for all $t\in[0,T]$, $i=1,\ldots,n$, and all
$\phi_1\in L^2(\dom)$ and all $\phi_2\in D(L)$ that
\begin{align}
  \lim_{\eps\to 0}\widetilde\E\int_0^T\big(\widetilde u_{i}^\eps(t)
	-\widetilde u_{i}(t),	\phi_1\big)_{L^2(\dom)}\dd t &= 0, \label{5.E1} \\
	\lim_{\eps\to 0}\widetilde\E\big\langle\widetilde u_{i}^\eps(0)
	-\widetilde u_{i}(0),\phi_2\big\rangle_{D(L)',D(L)} &= 0, \label{5.E2} \\
	\lim_{\eps\to 0}\widetilde\E\int_0^T\big\langle\sqrt{\eps}\widetilde{w}_i^\eps(t),
	\phi_2\big\rangle_{D(L)',D(L)}\dd t &= 0, \label{5.E3} \\
	\lim_{\eps\to 0}\widetilde\E\langle\sqrt{\eps}\widetilde{w}_i^\eps(0),
	\phi_2\rangle_{D(L)',D(L)} &= 0, \label{5.E3a} \\
	\lim_{\eps\to 0}\widetilde\E\int_0^T\bigg|\sum_{j=1}^n\int_0^t\int_\dom\big(
	A_{ij}(\widetilde u^\eps(s))\na\widetilde u_j^\eps(s)
	- A_{ij}(\widetilde u(s))\na\widetilde u_j(s)\big)\cdot
	\na\phi_2	\dd x\dd s\bigg|\dd t &= 0, \label{5.E4} \\
	\lim_{\eps\to 0}\widetilde\E\int_0^T\bigg|\sum_{j=1}^n\int_0^t
	\Big(\sigma_{ij}(\widetilde u^\eps(s))\dd\widetilde W_j^\eps(s)
	-\sigma_{ij}(\widetilde u(s))\dd\widetilde W_j(s),\phi_1\Big)_{L^2(\dom)}\bigg|^2 
	\dd t &= 0. \label{5.E5}
\end{align}
\end{lemma}

\begin{proof}
The proof is a combination of the uniform bounds and Vitali's 
convergence theorem. Convergences \eqref{5.E1} and \eqref{5.E2} have been shown
in the proof of \cite[Lemma 16]{DJZ19}, and \eqref{5.E3} is a direct consequence of
\eqref{5.LLR} and
$$
  \widetilde{\E}\bigg(\int_0^T\langle\sqrt{\eps}\widetilde{w}_i^\eps(t),\phi_2
	\rangle_{D(L)',D(L)}\dd t\bigg)^p 
	\le \eps^{p/2}\widetilde{\E}\bigg(\int_0^T\|\widetilde{w}_i^\eps(t)\|_{D(L)'}
	\|\phi_2\|_{D(L)}\dd t\bigg)^p \le \eps^{p/2}C.
$$
Convergence \eqref{5.E3a} follows from $\widetilde{w}_i^\eps\rightharpoonup
\widetilde{w}_i$ weakly* in $L^\infty(0,T;D(L)')$.
We establish \eqref{5.E4}:
\begin{align*}
  \bigg|\int_0^T&\bigg|\sum_{j=1}^n\int_0^t\int_\dom\big(
	A_{ij}(\widetilde u^\eps(s))\na\widetilde u_j^\eps(s)
	- A_{ij}(\widetilde u(s))\na\widetilde u_j(s)\big)\cdot\na\phi_2\dd x\dd s\bigg| \\
	&\le \int_0^T\|A_{ij}(\widetilde u^\eps(s))-A_{ij}(\widetilde u(s))\|_{L^2(\dom)}
	\|\na \widetilde u_j^\eps(s)\|_{L^2(\dom)}\|\na\phi_2\|_{L^\infty(\dom)}\dd s \\
	&\phantom{xx}{}+ \bigg|\int_0^T\int_\dom A_{ij}(\widetilde u(s))
	\na(\widetilde u^\eps(s) - \widetilde u(s))\cdot\na\phi_2 \dd x\dd s\bigg| 
	=: I_1^\eps+I_2^\eps.
\end{align*}
By the Lipschitz continuity of $A$ and the uniform bound for $\na\widetilde u^\eps$, 
we have $I_1^\eps\to 0$ as $\eps\to 0$ $\widetilde{\Prob}$-a.s. 
At this point, we use the embedding $D(L)\hookrightarrow W^{1,\infty}(\dom)$.
Also the second
integral $I_2^\eps$ converges to zero, since $A_{ij}(\widetilde u)\na\phi_2
\in L^2(0,T;L^2(\dom))$ and $\na\widetilde u_j^\eps\rightharpoonup\na\widetilde u_j$
weakly in $L^2(0,T;L^2(\dom))$. This shows that $\widetilde \Prob$-a.s., 
$$
  \lim_{\eps\to 0}\int_0^T\int_\dom A_{ij}(\widetilde u^\eps(s))
	\na\widetilde u_j^\eps(s)\cdot\na\phi_2\dd x\dd s 
	= \int_0^T\int_\dom A_{ij}(\widetilde u(s))\na\widetilde u_j(s)
	\cdot\na\phi_2\dd x\dd s.
$$
A straightforward estimation and bound \eqref{5.H1p} lead to
\begin{align*}
  \widetilde \E&\bigg|\int_0^T\int_\dom A_{ij}(\widetilde u^\eps(s))
	\na\widetilde u_j^\eps(s)\cdot\na\phi_2\dd x\dd s\bigg|^p \\
	&\le \|\na\phi_2\|_{L^\infty(\dom)}^p\widetilde \E\bigg(\int_0^T
	\bigg\|\sum_{j=1}^n A_{ij}(\widetilde u^\eps(s))\na\widetilde u_j^\eps(s)
	\bigg\|_{L^1(\dom)}\dd s\bigg)^p \le C,
\end{align*}
Hence, Vitali's convergence theorem gives \eqref{5.E4}. 

It remains to prove \eqref{5.E5}. 
By Assumption (A4), $\widetilde \Prob$-a.s.,
$$
  \int_0^T\|\sigma_{ij}(\widetilde u^\eps(s))-\sigma_{ij}(\widetilde u(s))
	\|_{{\mathcal L}_2(U;L^2(\dom))}^2\dd s
	\le C_\sigma\|\widetilde u^\eps-\widetilde u\|_{L^2(0,T;L^2(\dom))}\to 0.
$$
This convergence and $\widetilde W^\eps\to\widetilde W$ in
$C^0([0,T];U_0)$ imply that \cite[Lemma 2.1]{DGT11}
$$
  \int_0^T\sigma_{ij}(\widetilde u^\eps)\dd\widetilde W^\eps
	\to\int_0^T\sigma_{ij}(\widetilde u)\dd\widetilde W\quad\mbox{in }
	L^2(0,T;L^2(\dom))\ \widetilde \Prob\mbox{-a.s.}
$$
By Assumption (A4) again,
\begin{align*}
  \widetilde \E&\bigg(\int_0^T\|\sigma_{ij}(\widetilde u^\eps(s))-\sigma_{ij}
	(\widetilde u(s))\|^2_{{\mathcal L}_2(U;L^2(\dom))}\dd s\bigg)^p \\
	&\le C + C\widetilde \E\bigg(\int_0^T\big(\|\widetilde u^\eps(s)\|_{L^2(\dom)}^2
	+ \|\widetilde u(s)\|_{L^2(\dom)}^2\big)\dd s\bigg)^p \le C.
\end{align*}
We infer from Vitali's convergence theorem that
$$
  \lim_{\eps\to 0}\widetilde \E\int_0^T\|\sigma_{ij}(\widetilde u^\eps(s))
	-\sigma_{ij}(\widetilde u(s))\|_{{\mathcal L}_2(U;L^2(\dom))}^2\dd s = 0.
$$
The estimate
\begin{align*}
  \widetilde \E&\bigg|\bigg(\int_0^T\sigma_{ij}(\widetilde u^\eps(s))
	\dd \widetilde W_j^\eps(s) - \int_0^T\sigma_{ij}(\widetilde u(s))
	\dd\widetilde W_j(s),\phi_1\bigg)_{L^2(\dom)}\bigg|^2 \\
	&\le C\|\phi_1\|_{L^2(\dom)}^2\widetilde \E\int_0^T
	\big(\|\sigma_{ij}(\widetilde u^\eps(s))\|_{{\mathcal L}_2(U;L^2(\dom))}^2
	+ \|\sigma_{ij}(\widetilde u(s))\|_{{\mathcal L}_2(U;L^2(\dom))}^2\big)\dd s \\
	&\le C\|\phi_1\|_{L^2(\dom)}^2\bigg\{1 + \widetilde \E\bigg(\int_0^T\big(
	\|\widetilde u^\eps(s)\|_{L^2(\dom)}^2 + \|\widetilde u(s)\|_{L^2(\dom)}^2\big)
	\dd s\bigg)\bigg\} \le C
\end{align*}
for all $\phi_1\in L^2(\dom)$ and the dominated convergence theorem yield
\eqref{5.E5}.
\end{proof}

To show that the limit is indeed a solution, we define, for $t\in[0,T]$,
$i=1,\ldots,n$, and $\phi\in D(L)$,
\begin{align*}
  \Lambda_i^\eps(\widetilde u^\eps,\widetilde w^\eps,\widetilde W^\eps,\phi)(t)
	&:= \langle\widetilde u_i(0),\phi\rangle 
	+ \sqrt{\eps}\langle\widetilde w^\eps(0),\phi\rangle \\
	&\phantom{xx}{}- \sum_{j=1}^n\int_0^t\int_\dom A_{ij}(\widetilde u^\eps(s))
	\na\widetilde u_j^\eps(s)\cdot\na\phi\dd x\dd s \\
	&\phantom{xx}{}+ \sum_{j=1}^n\bigg(\int_0^t\sigma_{ij}(\widetilde u^\eps(s))\dd
	\widetilde W^\eps_j(s),\phi\bigg)_{L^2(\dom)}, \\
	\Lambda_i(\widetilde u,\widetilde w,\widetilde W,\phi)(t) 
	&:= \langle\widetilde u_i(0),\phi\rangle
	- \sum_{j=1}^n\int_0^t\int_\dom\langle A_{ij}(\widetilde u(s))\na\widetilde u_j(s)
	\cdot\na\phi\dd x\dd s \\
	&\phantom{xx}{}+ \sum_{j=1}^n\bigg(\int_0^t\sigma_{ij}(\widetilde u(s))\dd
	\widetilde W_j(s),\phi\bigg)_{L^2(\dom)}.
\end{align*}
The following corollary is a consequence of the previous lemma.

\begin{corollary}\label{coro.skt}
It holds for any $\phi_1\in L^2(\dom)$ and $\phi_2\in D(L)$ that
\begin{align*}
  \lim_{\eps\to 0}\big\|(\widetilde u_i^\eps,\phi_1)_{L^2(\dom)}
	- (\widetilde u_i,\phi_1)_{L^2(\dom)}\big\|_{L^1(\widetilde\Omega\times(0,T))} 
	&= 0, \\
	\lim_{\eps\to 0}\|\Lambda_i^\eps(\widetilde u^\eps,\sqrt{\eps}\widetilde w^\eps,
	\widetilde W^\eps,\phi_2) - \Lambda_i(\widetilde u,0,\widetilde W,\phi_2)
	\|_{L^1(\widetilde \Omega\times(0,T))} &= 0.
\end{align*}
\end{corollary}

Since $v^\eps$ is a strong solution to \eqref{4.approx}, it satisfies
for a.e.\ $t\in[0,T]$ $\Prob$-a.s., $i=1,\ldots,n$, and $\phi\in D(L)$,
$$
  (v_i^\eps(t),\phi)_{L^2(\dom)}
	= \Lambda_i^\eps(u^\eps,\eps L^*LR_\eps(v^\eps),W,\phi)(t)
$$
and in particular,
$$
  \int_0^T\E\big|(v_i^\eps(t),\phi)_{L^2(\dom)} 
	- \Lambda_i^\eps(u^\eps,\eps L^*LR_\eps(v^\eps),W,\phi)(t)\big|\dd t = 0.
$$
We deduce from the equivalence of the laws of $(u^\eps,\eps L^*LR_\eps(v^\eps),W)$
and $(\widetilde u^\eps,\sqrt{\eps}\widetilde w^\eps,\widetilde W)$ that
$$
  \int_0^T\widetilde \E\big|\big(\widetilde u^\eps_i(t)
	+\sqrt{\eps}\widetilde w_i^\eps,\phi\big)_{L^2(\dom)} 
	- \Lambda_i^\eps\big(\widetilde u^\eps,\sqrt{\eps}\widetilde w^\eps,
	\widetilde W^\eps,\phi\big)(t)\big|\dd t = 0.
$$
By Corollary \ref{coro.skt}, we can pass to the limit $\eps\to 0$ to obtain
$$
  \int_0^T\widetilde \E\big|(\widetilde u_i(t),\phi)_{L^2(\dom)} 
	- \Lambda_i(\widetilde u,0,\widetilde W,\phi)(t)\big|\dd t = 0.
$$
This identity holds for all $i=1,\ldots,n$ and all $\phi\in D(L)$. This shows that 
$$
  \big|(\widetilde u_i(t),\phi)_{L^2(\dom)} 
	- \Lambda_i(\widetilde u,0,\widetilde W,\phi)(t)\big| = 0
	\quad\mbox{for a.e. }t\in[0,T]\ \widetilde\Prob\mbox{-a.s.},\ i=1,\ldots,n.
$$
We infer from the definition of $\Lambda_i$ that
\begin{align*}
  (\widetilde u_i(t),\phi)_{L^2(\dom)} &= (\widetilde u_i(0),\phi)_{L^2(\dom)} 
	- \sum_{j=1}^n\int_0^t \int_\dom A_{ij}(\widetilde u(s))\na\widetilde u_j(s)
	\cdot\na\phi\dd x\dd s \\
	&\phantom{xx}{}+ \sum_{j=1}^n\bigg(\int_0^t\sigma_{ij}
	(\widetilde u(s))\dd\widetilde W_j(s),\phi\bigg)_{L^2(\dom)} 
\end{align*}
for a.e.\ $t\in[0,T]$ and all $\phi\in D(L)$. 
Set $\widetilde U=(\widetilde\Omega,\widetilde{\mathcal F},\widetilde\Prob,
\widetilde\F)$. Then $(\widetilde U,\widetilde W,\widetilde u)$ is a martingale
solution to \eqref{1.eq}--\eqref{1.SKT}.

%%%%%%%%%%%%%%%%%%%%%%%%%%%%%%%%%%%%%%%%%%%%%%%%%%%%%%%%%%%%%%%%%%%%%%%%%%%%%%

\section{Proof of Theorem \ref{thm.skt2}}\label{sec.skt2}

We turn to the existence proof of the SKT model without self-diffusion.

\subsection{Uniform estimates}

Let $v^\eps$ be a global solution to \eqref{4.defsol}--\eqref{4.bc}
and set $u^\eps=u(R_\eps(v^\eps))$. We assume that $A(u)$ is given by
\eqref{1.SKT} and that $a_{i0}>0$, $a_{ii}=0$ for $i=1,\ldots,n$.
The uniform estimates of Lemmas \ref{lem.est} and \ref{lem.hom} are still valid.
Since $a_{ii}=0$, we obtain an $H^1(\dom)$ bound for $(u_i^\eps)^{1/2}$
instead of $u_i^\eps$, which yields weaker bounds than those in Lemma \ref{lem.hom2}.

\begin{lemma}%\label{lem.hom.no}
Let $p\ge 2$ and set $\rho_1:=(d+2)/(d+1)$. 
Then there exists a constant $C(p,u^0,T)>0$, which is independent of $\eps$, such that 
\begin{align}
  \E\|u_i^\eps\|_{L^2(0,T;W^{1,1}(\dom))}^p &\le C(p,u^0,T), \label{6.W11p} \\
	\E\|u_i^\eps\|_{L^{1+2/d}(Q_T)}^p &\le C(p,u^0,T), \label{6.12dp} \\
	\E\|u_i^\eps\|_{L^{4/d}(0,T;L^2(\dom))}^p &\le C(p,u^0,T), \label{6.4d2p} \\
  \E\|u_i^\eps\|_{L^{\rho_1}(0,T;W^{1,\rho_1}(\dom))}^p &\le C(p,u^0,T). 
	\label{6.rho1p}
\end{align}
\end{lemma}

\begin{proof}
The identity $\na u_i^{\eps}=2(u_i^{\eps})^{1/2}\na(u_i^{\eps})^{1/2}$ and
the H\"older inequality show that
\begin{align*}
  \E\|\na u_i^{\eps}\|_{L^2(0,T;L^1(\dom))}^{p}
	&\le C\E\bigg(\int_0^T\|(u_i^{\eps})^{1/2}\|_{L^2(\dom)}^2
	\|\na (u_i^{\eps})^{1/2}\|_{L^2(\dom)}^2\dd t\bigg)^{p/2} \\
  &\le C\E\bigg(\|u_i^{\eps}\|_{L^\infty(0,T;L^1(\dom))}\int_0^T
	\|\na (u_i^{\eps})^{1/2}\|_{L^2(\dom)}^2\dd t\bigg)^{p/2} \\
	&\le C\big(\E\|u_i^{\eps}\|_{L^\infty(0,T;L^1(\dom))}^{p}\big)^{1/2}
	\big(\E\|\na (u_i^{\eps})^{1/2}\|_{L^2(0,T;L^2(\dom))}^{2p}\big)^{1/2}.
\end{align*}
Because of \eqref{5.L1p} and \eqref{5.H1p}, the right-hand side is bounded.
Using \eqref{5.L1p} again, we infer that \eqref{6.W11p} holds.
Estimate \eqref{6.12dp} is obtained from the Gagliardo--Nirenberg inequality 
similarly as in the proof of Lemma \ref{lem.hom2}:
\begin{align*}
  \E\bigg(\int_0^T\|(u_i^{\eps})^{1/2}\|_{L^{r}(\dom)}^{s}\bigg)^{p/s}
	&\le C\big(\E\|(u_i^{\eps})^{1/2}\|_{L^\infty(0,T;L^2(\dom))}^{2(1-\theta)p}
	\big)^{1/2} \\
	&\phantom{xx}{}\times
	\big(\E\|(u_i^{\eps})^{1/2}\|_{L^2(0,T;H^1(\dom))}^{4p/s}\big)^{1/2} \le C,
\end{align*}
where $s=2/\theta\ge 2$ and $1/r=1/2-\theta/d=1/2-2/(ds)$. Choosing $r=(2d+4)/d$ gives 
$s=r$, and $r=4$ leads to $s=8/d$; this proves estimates \eqref{6.12dp}
and \eqref{6.4d2p}. Finally, \eqref{6.rho1p} follows from H\"older's inequality:
\begin{align*}
  \|u_i^\eps\|_{L^{\rho_1}(Q_T)}
	&= 2\|(u_i^\eps)^{1/2}\na(u_i^\eps)^{1/2}\|_{L^{\rho_1}(Q_T)} 
	\le 2\|(u_i^\eps)^{1/2}\|_{L^{(2d+4)/d}(Q_T)}\|\na(u_i^\eps)^{1/2}\|_{L^2(Q_T)} \\
  &\le 2\|u_i^\eps\|_{L^{1+2/d}(Q_T)}^{1/2}\|(u_i^\eps)^{1/2}\|_{L^2(0,T;H^1(\dom))}
\end{align*}
and taking the expectation and using \eqref{6.12dp} and \eqref{5.H1p} ends the proof.
\end{proof}

The following lemma is needed to derive the fractional time estimate.

\begin{lemma}
Let $p\ge 2$ and set $\rho_2:=(2d+2)/(2d+1)$. Then it holds for any 
$i,j=1,\ldots,n$ with $i\neq j$:
\begin{equation}\label{6.uuW11}
  \E\|u_i^\eps u_j^\eps\|_{L^{\rho_2}(0,T;W^{1,\rho_2}(\dom))}^p \le C(p,u^0,T).
\end{equation}
\end{lemma}

\begin{proof}
The H\"older inequality and \eqref{5.L1p} immediately yield 
\begin{equation*}%\label{4.uu12L1}
  \E\|(u_i^\eps u_j^\eps)^{1/2}\|_{L^\infty(0,T;L^1(\dom))}^p \le C,
\end{equation*}
and we conclude from the Poincar\'e--Wirtinger inequality, estimate 
\eqref{5.nablap}, and the previous estimate that
\begin{equation}\label{6.uu12H1}
  \E\|(u_i^\eps u_j^\eps)^{1/2}\|_{L^2(0,T;H^1(\dom))}^p \le C.
\end{equation}
By the Gagliardo--Nirenberg inequality, with $\theta=d/(d+1)$,
\begin{align*}
  \int_0^T\|(u_i^\eps u_j^\eps)^{1/2}\|_{L^{2(d+1)/d}(\dom)}^{2(d+1)/d}\dd t
	&\le C\int_0^T\|(u_i^\eps u_j^\eps)^{1/2}\|_{H^1(\dom)}^{2\theta(d+1)/d}
	\|(u_i^\eps u_j^\eps)^{1/2}\|_{L^1(\dom)}^{2(1-\theta)(d+1)/d}\dd t \\
	&\le C\|(u_i^\eps u_j^\eps)^{1/2}\|_{L^\infty(0,T;L^1(\dom))}^{2(1-\theta)(d+1)/d}
	\int_0^T\|(u_i^\eps u_j^\eps)^{1/2}\|_{H^1(\dom)}^2 \dd t \\
	&= C\|(u_i^\eps u_j^\eps)^{1/2}\|_{L^\infty(0,T;L^1(\dom))}^{2/d}
	\|(u_i^\eps u_j^\eps)^{1/2}\|_{L^2(0,T;H^1(\dom))}^2.
\end{align*}
Taking the expectation and applying the H\"older inequality, we infer that
\begin{equation}\label{6.uu12L}
  \E\|(u_i^\eps u_j^\eps)^{1/2}\|_{L^{2(d+1)/d}(Q_T)}^p \le C.
\end{equation}
Finally, the identity $\na(u_i^\eps u_j^\eps) = 2(u_i^\eps u_j^\eps)^{1/2}
\na(u_i^\eps u_j^\eps)^{1/2}$ and H\"older's inequality lead to
\begin{align*}
  \int_0^T&\|\na(u_i^\eps u_j^\eps)\|_{L^{\rho_2}(\dom)}^{\rho_2}\dd t
	\le C\int_0^T\|(u_i^\eps u_j^\eps)^{1/2}\|_{L^{2(d+1)/d}(\dom)}^2
	\|\na(u_i^\eps u_j^\eps)^{1/2}\|_{L^2(\dom)}^2\dd t \\
  &\le C\bigg(\int_0^T\|(u_i^\eps u_j^\eps)^{1/2}\|_{L^{2(d+1)/d}(\dom)}^{2(d+1)/d}
	\dd t\bigg)^{1-\rho_2/2}\bigg(\int_0^T
	\|\na(u_i^\eps u_j^\eps)^{1/2}\|_{L^2(\dom)}^2\dd t\bigg)^{\rho_2/2}.
\end{align*}
The bounds \eqref{6.uu12H1}--\eqref{6.uu12L} yield, 
after taking the expectation and applying
H\"older's inequality again, the conclusion \eqref{6.uuW11}.
\end{proof}

We show now that the fractional time derivative of $u^\eps$ is uniformly bounded.

\begin{lemma}[Fractional time regularity]%\label{lem.frac2}
Let $d\le 2$. Then there exist $0<\alpha<1$, $p>1$, and $\beta>0$ such that
$\alpha p>1$ and
\begin{equation}
	\E\|u^\eps\|_{W^{\alpha,p}(0,T;D(L)')}^p + \E\|u^\eps\|_{C^{0,\beta}([0,T];D(L)')}^p
	\le C. \label{6.ualpha}
\end{equation}
\end{lemma}

\begin{proof}
We proceed similarly as in the proof of Lemma \ref{lem.frac}. First, we estimate
the diffusion part, setting
$$
  g(t) = \int_0^t\bigg\|a_{i0}\na u_i^\eps + \sum_{j\neq i}a_{ij}
	\na(u_i^\eps u_j^\eps)\bigg\|_{L^1(\dom)}\dd r.
$$
Then, using $D(L)\subset W^{1,\infty}(\dom)$ (which holds due to the assumption
$m>d/2+1$),
\begin{align*}
  \E&\int_0^T\int_0^T|t-s|^{-1-\alpha p}\bigg\|\int_{s\wedge t}^{t\lor s}
	\diver\sum_{j=1}^n A_{ij}(u^\eps(r))\na u_j^\eps(r)\dd r\bigg\|_{D(L)'}^p\dd t\dd s \\
	&\le C\int_0^T\int_0^T|t-s|^{-1-\alpha p}\bigg(\int_{s\wedge t}^{t\lor s}
	\bigg\|a_{i0}\na u_i^\eps + \sum_{j\neq i}a_{ij}
	\na(u_i^\eps u_j^\eps)\bigg\|_{L^1(\dom)}\dd r\bigg)^p\dd t\dd s \\
	&\le C\E\int_0^T\int_0^T\frac{|g(t)-g(s)|^p}{|t-s|^{1+\alpha p}}\dd t\dd s
	\le C\E\|g\|_{W^{\alpha,p}(0,T;\R)}^p.
\end{align*}
The embedding $W^{1,p}(0,T;\R)\hookrightarrow W^{\alpha,p}(0,T;\R)$ 
and estimates \eqref{6.uuW11} and \eqref{6.rho1p} 
show that for $1\le p\le\rho_1 = (d+2)/(d+1)$,
\begin{align*}
  \E&\|g\|_{W^{\alpha,p}(0,T;\R)}^p
	\le C\E\|g\|_{W^{1,p}(0,T;\R)}^p
	= C\E\|\pa_t g\|_{L^p(0,T;\R)}^p + C\E\|g\|_{L^p(0,T;\R)}^p \\
	&\le C\E\int_0^T\bigg\|a_{i0}\na u_i^\eps(t) + \sum_{j\neq i}a_{ij}
	\na(u_i^\eps u_j^\eps)(t)\bigg\|_{L^1(\dom)}^p \dd t \\
	&\phantom{xx}{}+ C\E\int_0^T\int_0^t\bigg\|a_{i0}\na u_i^\eps(r) 
	+ \sum_{j\neq i}a_{ij}\na(u_i^\eps u_j^\eps)(r)\bigg\|_{L^1(\dom)}^p \dd r\dd t
	\le C.
\end{align*}
Next, we consider the stochastic part, using the Burkholder--Davis--Gundy
inequality, H\"older's inequality, and the sublinear growth condition
in the statement of the theorem:
\begin{align*}
  \E&\int_0^T\int_0^T|t-s|^{-1-\alpha p}\bigg\|\int_{s\wedge t}^{t\lor s}
	\sum_{j=1}^n\sigma_{ij}(u^\eps(r))\dd W_j(r)\bigg\|_{L^2(\dom)}^p\dd t\dd s \\
	&\le C\E\int_0^T\int_0^T|t-s|^{-1-\alpha p}\bigg(\int_{s\wedge t}^{t\lor s}
	\sum_{j=1}^n\|\sigma_{ij}(u^\eps(r))\|_{{\mathcal L}_2(U;L^2(\dom))}^2\dd r
	\bigg)^{p/2}\dd t\dd s \\
	&\le  C\int_0^T\int_0^T|t-s|^{-1-\alpha p+p/2-1}\E\int_{s\wedge t}^{t\lor s}
	\sum_{j=1}^n\|\sigma_{ij}(u^\eps(r))\|_{{\mathcal L}_2(U;L^2(\dom))}^p\dd r
	\dd t\dd s \\
	&\le C\int_0^T\int_0^T|t-s|^{-1-\alpha p+p/2-1}\E\int_{s\wedge t}^{t\lor s}
	\sum_{j=1}^n(1+\|u^\eps(r)\|_{L^2(\dom)}^{\gamma p})\dd r\dd t\dd s \le C.
\end{align*}
The last step follows from estimate \eqref{6.4d2p} (assuming that $1\le\gamma p\le 4/d$)
and Lemma \ref{lem.g}, since $1+\alpha p-p/2+1<2$ if and only if $\alpha<1/2$. 
We conclude that the second term of the right-hand side of
$$
  v^\eps(t) = v^\eps(0) + \int_0^t\diver(A(u^\eps(s))\na u^\eps(s))\dd s
	+ \int_0^t\sigma(u^\eps(s))\dd W(s)
$$
is uniformly bounded in $\E|\cdot|_{W^{\alpha,p}(0,T;D(L)')}$ 
for $\alpha<1$ and $p\le(d+2)/(d+1)$, while the third term is uniformly
bounded in that norm for $\alpha<1/2$ and $p\le 4/(\gamma d)$. In both cases, we
can choose $\alpha$ such that $\alpha p>1$. At this point, we need the
condition $\gamma<1$ if $d=2$. (The result holds for any space dimension
if $\gamma<2/d$.)
Taking into account \eqref{5.vp}, $(v^\eps)$ is bounded in $W^{\alpha,p}(0,T;D(L)')$.
The embedding $W^{\alpha,p}(0,T;D(L)')\hookrightarrow C^{0,\beta}([0,T];D(L)')$ 
for $\beta=\alpha-1/p>0$ implies that $(v^\eps)$ is bounded in the latter space.

We turn to the estimate of $u^\eps$ in the $W^{\alpha,p}(0,T;D(L)')$ norm:
$$
  \E\|u^\eps\|_{W^{\alpha,p}(0,T;D(L)')}^p
	\le C\big(\E\|v^\eps\|_{W^{\alpha,p}(0,T;D(L)')}^p
	+ \eps\E\|L^*LR_\eps(v^\eps)\|_{W^{\alpha,p}(0,T;D(L)')}^p\big).
$$
It remains to consider the last term. In view of estimate \eqref{3.LL} and
the Lipschitz continuity of $R_\eps$ with Lipschitz constant $C/\eps$, we obtain
\begin{align*}
  \E&|\eps L^*LR_\eps(v^\eps)|_{W^{\alpha,p}(0,T;D(L)')}^p \\
	&= \eps^p\E\int_0^T\int_0^T|t-s|^{-1-\alpha p}\|L^*LR_\eps(v^\eps(t))
	- L^*LR_\eps(v^\eps(s))\|_{D(L)'}^p\dd t\dd s \\
	&\le \eps^p C\E\int_0^T\int_0^T|t-s|^{-1-\alpha p}\|R_\eps(v^\eps(t))
	- R_\eps(v^\eps(s))\|_{D(L)}^p \dd t\dd s \\
	&\le \eps^p C\E\int_0^T\int_0^T|t-s|^{-1-\alpha p}\frac{C}{\eps^p}
	\|v^\eps(t)-v^\eps(s)\|_{D(L)'}^p\dd t\dd s \\
	&= C\E\|v^\eps\|_{W^{\alpha,p}(0,T;D(L)')}^p \le C.
\end{align*}
Moreover, by \eqref{3.LL} and the Lipschitz continuity of $R_\eps$ again,
$$
  \|\eps L^*LR_\eps(v^\eps)\|_{L^p(0,T;D(L)')}^p
	\le \eps^p C\|R_\eps(v^\eps)\|_{L^p(0,T;D(L))}^p
	\le \eps^p C\|v^\eps\|_{L^p(0,T;D(L)')}^p \le C,
$$
where we used estimate \eqref{5.vp}. This finishes the proof.
\end{proof}

%%%%%%%%%%%%%%%%%%%%%

\subsection{Tightness of the laws of $(u^\eps)$}

The tightness is shown in a different sub-Polish space
than in Section \ref{sec.tight}:
$$
  \widetilde{Z}_T := C^0([0,T];D(L)')\cap 
	L_w^{\rho_1}(0,T;W^{1,\rho_1}(\dom)),
$$
endowed with the topology $\widetilde{\mathbb{T}}$ that is the maximum of the
topology of $C^0([0,T];D(L)')$ and the weak topology of 
$L_w^{\rho_1}(0,T;W^{1,\rho_1}(\dom))$, recalling that $\rho_1=(d+2)/(d+1)>1$.

\begin{lemma}%\label{lem.tight2}
The family of laws of $(u^\eps)$ is tight in
$$
  Z_T := \widetilde{Z}_T \cap L^2(0,T;L^{2}(\dom))
$$
with the topology that is the maximum of $\widetilde{\mathbb{T}}$ and the 
topology induced by the $L^2(0,T;L^{2}(\dom))$ norm.
\end{lemma}

\begin{proof}
The tightness in $L^2(0,T;L^q(\dom))$ for $q<d/(d-1)=2$ is a consequence of
the compact embedding $W^{1,1}(\dom)\hookrightarrow L^q(\dom)$ as well as
estimates \eqref{6.W11p} and \eqref{6.ualpha}. In fact, we can extend this result
up to $q=2$ because of the uniform bound of $u_i^\eps\log u_i^\eps$
in $L^\infty(0,T;L^1(\dom))$, which originates from the entropy estimate.
Indeed, we just apply \cite[Prop.~1]{BCJ20}, using additionally \eqref{5.H1}
with $a_{i0}>0$. Then the tightness in $L^2(0,T;L^2(\dom))$
follows from Lemma \ref{lem.tight}.
Finally, the tightness in $\widetilde{Z}_T$ is shown as
in the proof of Lemma \ref{lem.tight1} in Appendix \ref{sec.aux}.
\end{proof}

In three space dimensions, we do not obtain tightness in $L^2(0,T;L^2(\dom))$
but in the larger space $L^{4/3}(0,T;L^2(\dom))$.
This follows similarly as in the proof
of Lemma \ref{lem.tight1} taking into account the compact embedding
$W^{1,\rho_1}(\dom)\hookrightarrow L^2(\dom)$, which holds as long as $d\le 3$,
as well as estimates \eqref{6.rho1p} and \eqref{6.ualpha}. 
Unfortunately, this result seems to be not sufficient to identify the limit
of the product $\widetilde{u}_i^\eps\widetilde{u}_j^\eps$.
Therefore, we restrict ourselves to the two-dimensional case.

The following result is shown exactly as in Lemma \ref{lem.tighteps}.

\begin{lemma}
The family of laws of $(\sqrt{\eps}L^*LR_\eps(v^\eps))$ is tight in 
$Y_T=L_w^2(0,T;D(L)')\cap L^\infty_{w*}(0,T;D(L)')$.
\end{lemma}

Arguing as in Section \ref{sec.conv1}, the Skorokhod--Jakubowski theorem implies
the existence of a subsequence, a probability space
$(\widetilde\Omega,\widetilde{\mathcal{F}},\widetilde{\Prob})$, and, on this space,
$(Z_T\times Y_T\times C^0([0,T];U_0))$-valued random variables
$(\widetilde{u}^\eps,\widetilde{w}^\eps,\widetilde{W}^\eps)$
and $(\widetilde{u},\widetilde{w},\widetilde{W})$ such that
$(\widetilde{u}^\eps,\widetilde{w}^\eps,\widetilde{W}^\eps)$
has the same law as $(u^\eps,\sqrt{\eps}L^*LR_\eps(v^\eps),W)$ 
on $\mathcal{B}(Z_T\times Y_T\times C^0([0,T];U_0))$ and, as 
$\eps\to 0$ and $\widetilde{\Prob}$-a.s.,
\begin{equation*}
  (\widetilde{u}^\eps,\widetilde{w}^\eps,\widetilde{W}^\eps)
	\to (\widetilde{u},\widetilde{w},\widetilde{W})
	\quad\mbox{in }Z_T\times Y_T\times C^0([0,T];U_0).
\end{equation*}
This convergence means that $\widetilde{\Prob}$-a.s.,
\begin{align*}
  \widetilde{u}^\eps\to\widetilde{u} &\quad\mbox{strongly in }C^0([0,T];D(L)'), \\
  \na\widetilde{u}^\eps\rightharpoonup\na\widetilde{u} &\quad\mbox{weakly in }
	L^{\rho_1}(Q_T), \\
	\widetilde{u}^\eps\to\widetilde{u} &\quad\mbox{strongly in }L^{2}(Q_T), \\
	\widetilde{w}^\eps\rightharpoonup\widetilde{w} &\quad\mbox{weakly in }
	L^2(0,T;D(L)'), \\
	\widetilde{w}^\eps\rightharpoonup\widetilde{w} &\quad\mbox{weakly* in }
	L^\infty(0,T;D(L)'), \\
	\widetilde{W}^\eps\to \widetilde{W} &\quad\mbox{strongly in }C^0([0,T];U_0).
\end{align*}

The remainder of the proof is very similar to that one of Section \ref{sec.conv1},
using slightly weaker convergence results.
The most difficult part is the convergence of the nonlinear term 
$\na(\widetilde{u}_i^\eps\widetilde{u}_j^\eps)$,
since the previous convergences do not allow us to perform the limit
$\widetilde{u}_i^\eps\na\widetilde{u}_j^\eps$ because of $\rho_1<2$.
The idea is to consider the ``very weak'' formulation by performing the limit in
$\widetilde{u}_i^\eps\widetilde{u}_j^\eps\Delta\phi$ instead of
$\na(\widetilde{u}_i^\eps\widetilde{u}_j^\eps)\cdot\na\phi$ for suitable test
functions $\phi$. Indeed, let $\phi\in L^\infty(0,T;C_0^\infty(\dom))$.
Since $\widetilde{u}_i^\eps\to\widetilde{u}$ strongly in
$L^2(0,T;L^2(\dom))$ $\widetilde{\Prob}$-a.s., we have
\begin{align*}
  \int_0^T\int_\dom \na(\widetilde{u}_i^j\widetilde{u}_j^\eps)\cdot\na\phi \dd x\dd t
  = -\int_0^T\int_\dom\widetilde{u}_i^\eps\widetilde{u}_j^\eps\Delta\phi\dd x\dd t
	\to -\int_0^T\int_\dom\widetilde{u}_i\widetilde{u}_j\Delta\phi\dd x\dd t.
\end{align*}
It follows from the equivalence of the laws that
$$
  \widetilde{\E}\bigg(\int_0^T\int_\dom\widetilde{u}_i^\eps
	\widetilde{u}_j^\eps\Delta\phi\dd x\dd t\bigg)^2 \le C,
$$
and we conclude from Vitali's theorem that
$$
  \widetilde{E}\bigg|\int_0^T\int_\dom\big(\widetilde{u}_i^\eps
	\widetilde{u}_j^\eps - \widetilde{u}_i\widetilde{u}_j\big)(t)\Delta\phi
	\dd x\dd t\bigg|\to 0\quad\mbox{as }\eps\to 0.
$$
By density, this convergence holds for all test functions $\phi\in
L^\infty(0,T;W^{2,\infty}(\dom))$ such that $\na\phi\cdot\nu=0$ on $\pa\dom$.
This ends the proof of Theorem \ref{thm.skt2}.

\begin{remark}[Three space dimensions]\rm
The three-dimensional case is delicate since $u_i^\eps$ lies in a space larger than
$L^2(Q_T)$. We may exploit the regularity \eqref{6.uuW11} for
$\na(u_i^\eps u_j^\eps)$, but this leads only
to the existence of random variables $\widetilde{\eta}_{ij}^\eps$ and
$\widetilde{\eta}_{ij}$ with $i,j=1,\ldots,n$ and $i\neq j$
on the space $X_T=L_w^{\rho_2}(0,T;L^{\rho_2}(\dom))$
such that $\widetilde{\eta}_{ij}^\eps$ and $u_i^\eps u_j^\eps$ have the same law
on $\mathcal{B}(X_T)$ and, as $\eps\to 0$,
$$
  \widetilde{\eta}_{ij}^\eps\rightharpoonup \widetilde{\eta}_{ij}
	\quad\mbox{weakly in }X_T.
$$
Similar arguments as before lead to the limit
$$
  \widetilde{\E}\bigg|\int_0^T\int_\dom\na(\widetilde{\eta}_{ij}^\eps
	- \widetilde{\eta}_{ij})(t)\cdot\na\phi(t)\dd x\dd t\bigg|\to 0,
$$
but we cannot easily identify $\widetilde{\eta}_{ij}$ with
$\widetilde{u}_i\widetilde{u}_j$. 
\qed
\end{remark}

%%%%%%%%%%%%%%%%%%%%%%%%%%%%%%%%%%%%%%%%%%%%%%%%%%%%%%%%%%%%%%%%%%%%%%%%%%%%%%

\section{Discussion of the noise terms}\label{sec.noise}

We present some examples of admissible terms $\sigma(u)$. Recall that
$(e_k)_{k\in\N}$ is an orthonormal basis of $U$.

\begin{lemma}\label{lem.sigma}
The stochastic diffusion
$$
  \sigma_{ij}(u) = \delta_{ij}s(u_i)\sum_{\ell=1}^\infty
	a_\ell(e_\ell,\cdot)_{U}, \quad 
	s(u_i) = \frac{u_i}{1+u_i^{1/2+\eta}}
$$
satisfies Assumption (A5) for $\eta>0$ and $(a_\ell)\in\ell^2(\R)$.
\end{lemma}

\begin{proof}
With the entropy density $h$ given by \eqref{1.h}, 
we compute $(\pa h/\pa u_i)(u)=\pi_i\log u_i$ and 
$(\pa^2 h/\pa u_i\pa u_j)(u)=(\pi_i/u_i)\delta_{ij}$. Therefore,
by Jensen's inequality and the elementary inequalities
$|u_i\log u_i|\le C(1+u_i^{1+\eta})$ for any $\eta>0$ and $|u|\le C(1+h(u))$,
\begin{align*}
  J_1 &:= \bigg\{\int_0^T\sum_{k=1}^\infty\sum_{i,j=1}^n
	\bigg(\int_\dom \frac{\pa h}{\pa u_i}(u)\sigma_{ij}(u)e_k\dd x\bigg)^2\dd s
	\bigg\}^{1/2} \\
	&= \bigg\{\sum_{k=1}^\infty a_k^2\int_0^T\sum_{i=1}^n\bigg(\int_\dom
	\pi_i\frac{u_i\log u_i}{1+u_i^{1/2+\eta}}\dd x\bigg)^2\dd s
	\bigg\}^{1/2} \\
	&\le C\bigg\{\sum_{i=1}^n\int_0^T\bigg(\int_\dom
	\frac{1+u_i^{1+\eta}}{1+u_i^{1/2+\eta}}\dd x\bigg)^2\dd s\bigg\}^{1/2} \\
	&\le C\bigg\{\sum_{i=1}^n\int_0^T\int_\dom
	\bigg(\frac{1+u_i^{1+\eta}}{1+u_i^{1/2+\eta}}\bigg)^2\dd x\dd s\bigg\}^{1/2} \\
	&\le C\bigg\{\sum_{i=1}^n\int_0^T\int_\dom(1+u_i)\dd x\dd s\bigg\}^{1/2}
	\le C\bigg(1+\int_0^T\int_\dom h(u)\dd x\dd s\bigg).
\end{align*}
The second condition in Assumption (A5) becomes
\begin{align*}
	J_2 &:= \int_0^T\sum_{k=1}^\infty\int_\dom\operatorname{tr}\big[(\sigma(u)e_k)^T
	h''(u)\sigma(u)e_k\big]\dd x\dd s \\
	&= \sum_{k=1}^\infty a_k^2\sum_{i=1}^n\int_0^T\int_\dom
	\frac{\pi_iu_i}{(1+u_i^{1/2+\eta})^2}\dd x\dd s \le C(\dom,T).
\end{align*}
Thus, Assumption (A5) is satisfied.
\end{proof}

The proof shows that $J_1$ can be estimated if $s(u_i)^2\log(u_i)^2$ is bounded
from above by $C(1+h(u))$. This is the case if $s(u_i)$ behaves like 
$u_i^\alpha$ with $\alpha<1/2$. Furthermore, $J_2$ can be estimated if
$s(u_i)^2/u_i$ is bounded, which is possible if $s(u_i)=u_i^\alpha$ with
$\alpha\ge 1/2$. Thus, to both satisfy the growth restriction and avoid the singularity
at $u_i=0$, we have chosen $\sigma_{ij}$ as in Lemma \ref{lem.sigma}.
This example is rather artificial. To include more general
choices, we generalize our approach. In fact, it is sufficient
to estimate the integrals in inequality \eqref{4.Hest} in such a way that the entropy 
inequality of Proposition \ref{prop.ent} holds. The idea is to exploit the
gradient bound for $u_i$ for the estimatation of $J_1$ and $J_2$. 

Consider a trace-class, positive, and symmetric operator $Q$ on $L^2(\dom)$
and the space $U=Q^{1/2}(L^2(\dom))$, equipped with the norm 
$\|Q^{1/2}(\cdot)\|_{L^2(\dom)}$. We will work in the following with an
$U$-cylindrical Wiener process $W^Q$.
This setting is equivalent to a spatially colored noise
on $L^2(\dom)$ in the form of a $Q$-Wiener process (with $Q\neq\mathrm{Id}$).
The latter viewpoint provides, in our opinion, a more intuitive insight.
In particular, the operator $Q$ is constructed from the eigenfunctions and 
eigenvalues described below. 

Let $(\eta_k)_{k\in\N}$ be a basis of
$L^2(\dom)$, consisting of the normalized eigenfunctions of the Laplacian subject to
Neumann boundary conditions with eigenvalues $\lambda_k\ge 0$, and set
$a_k=(1+\lambda_k)^{-\rho}$ for some $\rho>0$ such that
$\sum_{k=1}^\infty a_k^2\|\eta_k\|_{L^\infty(\dom)}^2<\infty$.
Since $\lambda_k\le Ck^{2/d}$ \cite[Corollary 2]{Kro92} and 
$\|\eta_k\|_{L^\infty(\dom)}\le Ck^{(d-1)/2}$ \cite[Theorem 1]{Gri02},
we may choose $\rho>(d/2)^2$. 
%The operator $Q$ is defined through its
%eigenfunctions $\eta_k$ and associated eigenvalues $a_k^{-1/\rho}$. 
Considering a sequence of independent Brownian motions $(W_1^k,\ldots,W_n^k)_{k\in\N}$,
we assume the noise to be of the form $W^Q=(W_1^Q,\ldots,W_n^Q)$, where
$$
  W_j^Q(t) = \sum_{k=1}^\infty a_k e_k W_j^k(t), \quad j=1,\ldots,n,\ t>0,
$$
and $(e_k)_{k\in\N}=(a_k\eta_k)_{k\in\N}$ is a basis of $U=Q^{1/2}(L^2(\dom))$.

\begin{lemma}\label{lem.sigma2}
For the SKT model with self-diffusion, let
$\sigma_{ij}(u)=\delta_{ij}u_i^\alpha$ for $1/2\le\alpha\le 1$,
$i,j=1,\ldots,n$, interpreted as a map from $L^2(\dom)$ to
$\L_2(H^\beta(\dom);L^2(\dom))$, where $\beta>\rho$. Then the entropy inequality
\eqref{5.ei} holds, i.e., $\sigma_{ij}$ is admissible for Theorem \ref{thm.skt}.
\end{lemma}

\begin{proof}
We can write inequality \eqref{4.Hest} for $0<T<T_R$ as
\begin{align}\label{7.ei}
  \E&\sup_{0<t<T}\int_\dom h(u^\eps(t))\dd x 
	+ \frac{\eps}{2}\E\sup_{0<t<T}\|Lw^\eps(t)\|_{L^2(\dom)}^2 \\
	&\phantom{xx}{}
	+ \E\sup_{0<t<T}\int_0^t\int_\dom \na w^\eps(s):B(w^\eps)\na w^\eps(s)
  \dd x\dd s - \E\int_\dom h(u^0)\dd x \nonumber \\
	&\le\E\sup_{0<t<T}\bigg\{\int_0^t\sum_{k=1}^\infty\sum_{i,j=1}^n\bigg(
	\int_\dom\pi_i\log u_i^\eps(s)\sigma_{ij}(u^\eps(s))e_k\dd x\bigg)^2
	\dd s\bigg\}^{1/2} \nonumber \\
	&\phantom{xx}{}+ \frac12\E\sup_{0<t<T}\sum_{k=1}^\infty\sum_{i=1}^n
	\int_0^t\int_\dom(\sigma_{ii}(u^\eps)e_k\frac{\pi_i}{u_i^\eps}
	\sigma_{ii}(u^\eps)e_k\dd x\dd s \nonumber \\
	&=: J_3 + J_4, \nonumber 
\end{align}
recalling that $w^\eps=R_\eps(v^\eps)$ and $u^\eps=u(w^\eps)$. 
We simplify $J_3$ and $J_4$, using the definition $e_k=a_k\eta_k$:
\begin{align*}
  J_3 &= \E\sup_{0<t<T}\bigg\{\sum_{k=1}^\infty a_k^2\int_0^t\sum_{i=1}^n\pi_i^2\bigg(
	\int_\dom u_i^\eps(s)^\alpha\log u_i^\eps(s)\eta_k\dd x\bigg)^2\dd s\bigg\}^{1/2} \\
	&\le C\E\sup_{0<t<T}\bigg\{\sum_{k=1}^\infty a_k^2\int_\dom \eta_k^2\dd x
	\int_0^t\sum_{i=1}^n\int_\dom(u_i^\eps(s)^\alpha\log u_i^\eps(s))^2
	\dd x\dd s\bigg\}^{1/2} \\
	&\le C\sum_{i=1}^n\E\|(u_i^\eps)^\alpha\log u_i^\eps\|_{L^2(0,T;L^2(\dom))}, \\
	J_4 &= \sum_{k=1}^\infty a_k^2\E\sup_{0<t<T}\sum_{i=1}^n\pi_i
	\int_0^{t}\int_\dom (u_i^\eps)^{2\alpha}(u_i^\eps)^{-1}\eta_k^2\dd x\dd s \\
	&\le C\sum_{k=1}^\infty a_k^2\|\eta_k\|_{L^\infty(\dom)}^2\sum_{i=1}^n
	\E\|(u_i^\eps)^{2\alpha-1}\|_{L^1(0,T;L^1(\dom))} \\
	&\le C\sum_{i=1}^n\E\|(u_i^\eps)^{2\alpha-1}\|_{L^1(0,T;L^1(\dom))}.
\end{align*}
The last inequality follows from our assumption on $(a_k)$. 
By \eqref{5.RBR}, we can estimate the integrand of the third integral on the 
left-hand side of \eqref{7.ei} according to
$$
  \na w^\eps:B(w^\eps)\na w^\eps \\
 	\ge 2\sum_{i=1}^n\pi_ia_{ii}|\na u^\eps|^2.
$$
Hence, because of $|u|\le C(1+h(u))$, we can formulate \eqref{7.ei} as
\begin{align*}%\label{7.ei2}
	\E&\sup_{0<t<T}\|h(u^\eps(t))\|_{L^1(\dom)}
	+ \frac{\eps}{2}\E\sup_{0<t<T}\|Lw^\eps(t)\|_{L^2(\dom)}^2
	+ C\E\|\na u^\eps(s)\|_{L^2(0,T;L^2(\dom))}^2 \\
	&\le C + C\sum_{i=1}^n\E\|(u_i^\eps)^{\alpha}\log u_i^\eps\|_{L^2(0,T;L^2(\dom))}
	+ C\sum_{i=1}^n\E\|(u_i^\eps)^{2\alpha-1}\|_{L^1(0,T;L^1(\dom))}. %\nonumber
\end{align*}
%The Poincar\'e--Wirtinger inequality 
%\begin{align*}
%  \|u_i^\eps\|_{L^2(0,T;H^1(\dom))}
%	&\le C\|\na u_i^\eps\|_{L^2(0,T;L^2(\dom))} + C\|u_i^\eps\|_{L^2(0,T;L^1(\dom))} \\
%	&\le C + C\|\na u_i^\eps\|_{L^2(0,T;L^2(\dom))}^2 
%	+ C\|u_i^\eps\|_{L^\infty(0,T;L^1(\dom))}
%\end{align*}
%allows us to reformulate the left-hand side of \eqref{7.ei2}: There exists
%$C>0$ such that
It is sufficient to continue with the case $\alpha=1$, since the proof for
$\alpha<1$ follows from the case $\alpha=1$. Then, using $|u^\eps|\le C(1+h(u^\eps))$,
\begin{align}\label{7.ei3}
  \E&\|h(u^\eps)\|_{L^\infty(0,T;L^1(\dom))} + \E\|u^\eps\|_{L^\infty(0,T;L^1(\dom))}\\
	&\phantom{xx}{}+ \eps\E\|Lw^\eps\|_{L^\infty(0,T;L^2(\dom))}^2 
	+ \E\|\na u^\eps\|_{L^2(0,T;L^2(\dom))}^2 \nonumber \\
	&\le C + C\sum_{i=1}^n\E\|u_i^\eps\log u_i^\eps\|_{L^2(0,T;L^2(\dom))}
	+ C\E\|u^\eps\|_{L^1(0,T;L^1(\dom))}. \nonumber
\end{align}

Now, we use the following lemma which is proved in Appendix \ref{sec.proofs}.

\begin{lemma}\label{lem.vlogv}
Let $d\ge 2$ and let $v\in L^2(0,T;H^1(\dom))$ satisfy 
$v\log v\in L^\infty(0,T;L^1(\dom))$.
Then for any $\delta>0$, there exists $C(\delta)>0$ such that
\begin{align*}
  \|v\log v\|_{L^2(0,T;L^2(\dom))}
	&\le \delta\big(\|v\log v\|_{L^1(0,T;L^1(\dom))}
	+ \|v\|_{L^\infty(0,T;L^1(\dom))} + \|\na v\|_{L^2(0,T;L^2(\dom))}^2\big) \\
	&\phantom{xx}{}+ C(\delta)\|v\|_{L^1(0,T;L^1(\dom))}.
\end{align*}
\end{lemma}

It follows from \eqref{7.ei3} that, for any $\delta>0$,
\begin{align*}
  \E&\|h(u^\eps)\|_{L^\infty(0,T;L^1(\dom))} + \E\|u^\eps\|_{L^\infty(0,T;L^1(\dom))}\\
	&\phantom{xx}{}+ \eps\E\|Lw^\eps\|_{L^\infty(0,T;L^2(\dom))}^2 
	+ \E\|\na u^\eps\|_{L^2(0,T;L^2(\dom))}^2 \\
	&\le C + C(\delta)\E\|u^\eps\|_{L^1(0,T;L^1(\dom))}
	+ \delta C\sum_{i=1}^n\E\|u_i^\eps\log u_i^\eps\|_{L^1(0,T;L^1(\dom))} \\
	&\phantom{xx}{}+ \delta C\big(\E\|u^\eps\|_{L^\infty(0,T;L^1(\dom))} 
	+ \E\|\na u^\eps\|_{L^2(0,T;L^2(\dom))}^2\big).
\end{align*}
For sufficiently small $\delta>0$, the last terms on the right-hand side
can be absorbed by the corresponding terms on the left-hand side, leading to
\begin{align*}
   \E&\|h(u^\eps)\|_{L^\infty(0,T;L^1(\dom))} + \E\|u^\eps\|_{L^\infty(0,T;L^1(\dom))}\\
	&\phantom{xx}{}+ \eps\E\|Lw^\eps\|_{L^\infty(0,T;L^2(\dom))}^2 
	+ \E\|\na u^\eps\|_{L^2(0,T;L^2(\dom))}^2 \\
	&\le C + C\int_0^T\|u^\eps\|_{L^\infty(0,t;L^1(\dom))}\dd t
	\quad\mbox{for all }T>0.
\end{align*}
Gronwall's lemma ends the proof.
\end{proof}

In the case without self-diffusion, we have an $H^1(\dom)$ estimate for
$(u_i^\eps)^{1/2}$ only, and it can be seen that
stochastic diffusion terms of the type 
$\delta_{ij}u_i^\alpha$ for $\alpha>1/2$ are not admissible.
However, we may choose $\sigma_{ij}(u)e_k 
= \delta_{ij}u_i^\alpha(1+(u_i^\eps)^\beta)^{-1}a_k\eta_k$ 
for $1/2\le\alpha<1$ and $\beta\ge\alpha/2$.

%%%%%%%%%%%%%%%%%%%%%%%%%%%%%%%%%%%%%%%%%%%%%%%%%%%%%%%%%%%%%%%%%%%%%%%%%%%%%%

\begin{appendix}

\section{Proofs of some lemmas}\label{sec.proofs}

\subsection{Proof of Lemma \ref{lem.ab}}

The operator equation $\DD R_\eps[w](a)=b$ can be written as
$a = \DD Q_\eps[w](b) = u'(w)b + \eps L^*Lb$. Hence,
\begin{equation}\label{a.aux}
  \int_\dom a:b\dd x = \int_\dom u'(w)b:b \dd x
	+ \eps\int_\dom Lb:Lb\dd x.
\end{equation}
The matrix $u'(w)=(h'')^{-1}(u(w))$ is symmetric and positive semidefinite
(since $h$ is convex).
Thus, the square root operator $\sqrt{u'(w)}$ exists and is symmetric.
This shows that
\begin{align*}
  u'(w)b:b &= \sqrt{u'(w)}\sqrt{u'(w)}b:b 
	= \operatorname{tr}\big[\big(\sqrt{u'(w)}\sqrt{u'(w)}b\big)^Tb\big] \\
	&= \operatorname{tr}\big[\big(\sqrt{u'(w)}b\big)^T\big(\sqrt{u'(w)}b\big)\big] 
	= \big\|\sqrt{u'(w)}b\big\|_F^2.
\end{align*}
Furthermore, by the Cauchy--Schwarz inequality $\operatorname{tr}[A^TB]
\le \operatorname{tr}[A^TA]\operatorname{tr}[B^TB]$ and the property
$\operatorname{tr}[AB]=\operatorname{tr}[BA]$ for matrices $A$ and $B$,
\begin{align*}
  a:b &= \operatorname{tr}\big[a^T\sqrt{u'(w)}{\,}^{-1}	\sqrt{u'(w)}b\big] \\
	&\le \operatorname{tr}\big[\big(\sqrt{u'(w)}b\big)^T\sqrt{u'(w)}b\big]^{1/2}
	\operatorname{tr}\big[\big(a^T\sqrt{u'(w)}{\,}^{-1}\big)^Ta^T\sqrt{u'(w)}{\,}^{-1}
	\big]^{1/2} \\
	&\le \frac12\operatorname{tr}\big[\big(\sqrt{u'(w)}b\big)^T\sqrt{u'(w)}b\big]
	+ \frac12\operatorname{tr}\big[\big(\sqrt{u'(w)}{\,}^{-1}a\big)
	\big(a^T\sqrt{u'(w)}{\,}^{-1}\big)\big] \\
	&= \frac12\big\|\sqrt{u'(w)}b\big\|_F^2
	+ \frac12\operatorname{tr}\big[\big(a^T\sqrt{u'(w)}{\,}^{-1}\big)
	\big(\sqrt{u'(w)}{\,}^{-1}a\big)\big] \\
	&= \frac12\big\|\sqrt{u'(w)}b\big\|_F^2
	+ \frac12\operatorname{tr}[a^Tu'(w)^{-1}a].
\end{align*}
Inserting these relations into \eqref{a.aux} leads to
\begin{align}\label{a.aux2}
  \int_\dom&\|\sqrt{u'(w)}b\|_F^2\dd x + \eps\int_\dom Lb:Lb\dd x
  = \int_\dom a:b\dd x \\
	&\le \frac12\int_\dom\|\sqrt{u'(w)}b\|_F^2\dd x
	+ \frac12\int_\dom\operatorname{tr}[a^Tu'(w)^{-1}a]\dd x \nonumber
\end{align}
and consequently,
$$
  \int_\dom\|\sqrt{u'(w)}b\|_F^2\dd x
	\le \int_\dom\operatorname{tr}[a^Tu'(w)^{-1}a]\dd x.
$$
Together with \eqref{a.aux2} we obtain the statement.

%%%%%%%%%%%%%%%%

\subsection{Proof of Lemma \ref{lem.v0}}
 
It follows from the convexity of $h$ that
$$
  h(v^0) \ge h(u(R_\eps(v^0))) + \big(v^0-u(R_\eps(v^0))\big)\cdot h'(u(R_\eps(v^0))).
$$
Since $R_\eps$ and $Q_\eps$ are inverse to each other, 
we can replace $v^0$ by $Q_\eps(R_\eps(v^0)) = u(R_\eps(v^0)) + \eps L^*LR_\eps(v^0)$:
\begin{align*}
  h(v^0) &\ge h(u(R_\eps(v^0))) + \big\langle 
	u(R_\eps(v^0)) + \eps L^*LR_\eps(v^0) - u(R_\eps(v^0)),h'(u(R_\eps(v^0)))
	\big\rangle_{D(L)',D(L)} \\
	&= h(u(R_\eps(v^0))) + \eps\langle L^*LR_\eps(v^0),R_\eps(v^0)\rangle_{D(L)',D(L)}.
\end{align*}
We find after an integration that
$$
  \int_\dom h(v^0)\dd x
	\ge \int_\dom h(u(R_\eps(v^0)))\dd x 
	+ \eps\int_\dom LR_\eps(v^0)\cdot LR_\eps(v^0)\dd x,
$$
which yields the statement.

%%%%%%%%%%%%%%%%%%%

\subsection{Proof of Lemma \ref{lem.g}}

We show that
$$
  I:= \int_0^T\int_0^T|t-s|^{-\delta}\int_{s\wedge t}^{t\lor s}g(r)\dd r\dd t\dd s 
	< \infty.
$$
A change of the integration domain and an integration by parts lead to
\begin{align}
  I &= 2\int_0^T\int_s^T(t-s)^{-\delta}\bigg(\int_s^t g(r)dr\bigg)\dd t\dd s 
	\label{3.I} \\
  &= -\frac{2}{1-\delta}\int_0^T\int_s^T(t-s)^{1-\delta}g(t)\dd t\dd s
	+ \frac{2}{1-\delta}\int_0^T(T-s)^{1-\delta}\int_s^t g(r)\dd r\dd s, \nonumber
\end{align}
observing that $\lim_{t\to s}(t-s)^{1-\delta}\int_s^t g(r)\dd r=0$ 
for $1-\delta>-1$, since the integrability of $g$ implies that 
$\lim_{t\to s}(t-s)^{-1}\int_s^t g(r)\dd r=g(s)$ for a.e.\ $s$. The result follows
as the integrals on the right-hand side of \eqref{3.I} are finite.

%%%%%%%%%%%%%%%%%%%%%

\subsection{Proof of Lemma \ref{lem.vlogv}}

We use the interpolation inequality with $1/2=\theta_1 + (1-\theta_1)/(2p)$ and 
some $1<p<d/(d-2)$ (and $p>1$ if $d=2$)
as well as the Young inequality with $\delta>0$:
\begin{align}\label{aux.eq}
  \|v&\log v\|_{L^2(0,T;L^2(\dom))}
	\le \bigg(\int_0^{T}\|v\log v\|_{L^1(\dom)}^{2\theta_1}
	\|v\log v\|_{L^{2p}(\dom)}^{2(1-\theta_1)}\dd t\bigg)^{1/2} \\
	&\le \bigg(C(\delta)^2\int_0^{T}\|v\log v\|_{L^1(\dom)}^2\dd t
	+ \delta^2\int_0^T\|v\log v\|_{L^{2p}(\dom)}^2\dd t\bigg)^{1/2} 
	\nonumber \\
	&\le C(\delta)\|v\log v\|_{L^2(0,T;L^1(\dom))}
	+ \delta\|v\log v\|_{L^2(0,T;L^{2p}(\dom))}. \nonumber
\end{align}
The first term on the right-hand side is estimated in a similar way as before, 
where $\eta>0$:
\begin{align}\label{aux.vlogv1}
  \|v&\log v\|_{L^2(0,T;L^1(\dom))}
	\le \|v\log v\|_{L^\infty(0,T;L^1(\dom))}^{1/2}
	\|v\log v\|_{L^1(0,T;L^1(\dom))}^{1/2} \\
	&\le \eta\|v\log v\|_{L^\infty(0,T;L^1(\dom))}
	+ C(\eta)\|v\log v\|_{L^1(0,T;L^1(\dom))}. \nonumber
\end{align}
For the second term on the right-hand side of \eqref{aux.eq}, we introduce
the function $g(v)=\max\{2,v\log v\}$ for $v\ge 0$. Then $g\in C^1([0,\infty))$.
(The function $v\mapsto v\log v$ is not $C^1$ at $v=0$, therefore we need to
truncate.)
We use the Sobolev inequality:
\begin{align*}
  \|v\log v\|_{L^2(0,T;L^{2p}(\dom))}
	&\le \|g(v)\|_{L^2(0,T;L^{2p}(\dom))}
	\le C\|g(v)\|_{L^2(0,T;W^{1,q}(\dom))} \\
	&\le C\big(\|g(v)\|_{L^2(0,T;L^q(\dom))} + \|\na g(v)\|_{L^2(0,T;L^q(\dom))}\big),
\end{align*}
where  $q=2dp/(d+2p)$.
The condition $p<d/(d-2)$ guarantees that $q<2$, while $d\ge 2$ yields $q>1$;
thus $q\in(1,2)$.
Applying the Gagliardo--Nirenberg inequality, combined with the Poincar\'e--Wirtinger
inequality, with $\theta_2=d(q-1)/(d(q-1)+q)\le 1$, and then the Young inequality,
we find that
\begin{align*}
  \|g(v)\|_{L^q(\dom)} &\le C\|\na g(v)\|_{L^q(\dom)}^{\theta_2}
	\|g(v)\|_{L^1(\dom)}^{1-\theta_2} + \|g(v)\|_{L^1(\dom)} \\
	&\le \|\na g(v)\|_{L^q(\dom)} + C\|g(v)\|_{L^1(\dom)}
	\le \|\na g(v)\|_{L^q(\dom)} + C\big(1+\|v\log v\|_{L^1(\dom)}\big).
\end{align*}
This yields
\begin{equation}\label{aux.vlogv2}
  \|v\log v\|_{L^2(0,T;L^{2p}(\dom))}
	\le C + C\|\na g(v)\|_{L^2(0,T;L^q(\dom))} + C\|v\log v\|_{L^2(0,T;L^1(\dom))}.
\end{equation}
The last term is estimated as in \eqref{aux.vlogv1}.
We consider the norm of $\na g(v)=\textrm{1}_{\{v\log v>2\}}(1+\log v)\na v$.
For this, we observe that $\textrm{1}_{\{v\log v>2\}}\log v\le C(1+v^\gamma)$ for
some $0<\gamma<(2-q)/(2q)$ and use the H\"older inequality:
\begin{align*}
  \|\na g(v)\|_{L^q(\dom)} &\le \|(1+v^\gamma)\na v\|_{L^q(\dom)}
	\le \big(1 + \|v^\gamma\|_{L^{2q/(2-q)}(\dom)}\big)\|\na v\|_{L^2(\dom)} \\
	&\le C\big(1 + \|v\|_{L^1(\dom)}^{(2-q)/(2q)}\big)\|\na v\|_{L^2(\dom)},
\end{align*}
since the property $2\gamma q/(2-q)<1$ gives 
$v^{2\gamma q/(2-q)}\le C(1+v)$ for $v\ge 0$.
Consequently, by Young's inequality,
$$
  \|\na g(v)\|_{L^q(\dom)} \le C\big(1 + \|v\|_{L^1(\dom)}^{(2-q)/q}
	+ \|\na v\|_{L^2(\dom)}^2\big),
$$  
and an integration over time gives
\begin{align*}
  \|\na g(v)\|_{L^2(0,T;L^q(\dom))}
	&\le C\big(1 + \|v\|_{L^\infty(0,T;L^1(\dom))}^{(2-q)/q}
	+ \|\na v\|_{L^2(0,T;L^2(\dom))}^2\big) \\
	&\le C\big(1 + \|v\|_{L^\infty(0,T;L^1(\dom))}
	+ \|\na v\|_{L^2(0,T;L^2(\dom))}^2\big),
\end{align*}
where we used $(2-q)/q<1$. Thus, \eqref{aux.vlogv2} becomes
\begin{align*}
  \|v\log v\|_{L^2(0,T;L^{2p}(\dom))} 
	&\le C + C\|\na v\|_{L^2(0,T;L^2(\dom))}^2 + C\|v\|_{L^\infty(0,T;L^1(\dom))}	\\
	&\phantom{xx}{}+ C\|v\log v\|_{L^2(0,T;L^1(\dom))}.
\end{align*}

It remains to insert \eqref{aux.vlogv1} and the previous estimate into \eqref{aux.eq}
to conclude that
\begin{align*}
  \|v\log v\|_{L^2(0,T;L^2(\dom))}
  &\le \eta C(\delta)\|v\log v\|_{L^\infty(0,T;L^1(\dom))}
	+ C(\delta,\eta)\|v\log v\|_{L^1(0,T;L^1(\dom))} \\
	&\phantom{xx}{}
	+ \delta C\big(\|\na v\|_{L^2(0,T;L^2(\dom))}^2 + \|v\|_{L^\infty(0,T;L^1(\dom))}
	\big).
\end{align*}
Choosing first $\delta>0$ and then $\eta>0$ sufficiently small finishes the proof.

%%%%%%%%%%%%%%%%%%%%%%%%%%%%%%%%%%%%%%%%%%%%%%%%%%%%%%%%%%%%%%%%%%%%%%%%%%%%

\section{Tightness criterion}\label{sec.aux}

\begin{lemma}[Tightness criterion]\label{lem.tight}
Let $\dom\subset\R^d$ ($d\ge 1$) be a bounded domain with Lipschitz 
boundary and let $T>0$, $p,q,r\ge 1$, $\alpha\in(0,1)$ if $r\ge p$ and
$\alpha\in(1/r-1/p,1)$ if $r<p$. Let $s\ge 1$ be such that the embedding
$W^{1,q}(\dom)\hookrightarrow L^s(\dom)$ is compact, and let $Y$ be
a Banach space such that the embedding $L^s(\dom)\hookrightarrow Y$ is continuous.
Furthermore, let $(u_n)_{n\in\N}$ be a sequence of functions such that there
exists $C>0$ such that for all $n\in\N$,
$$
  \E\|u_n\|_{L^p(0,T;W^{1,q}(\dom))} + \E\|u_n\|_{W^{\alpha,r}(0,T;Y))} \le C.
$$
Then the laws of $(u_n)$ are tight in $L^p(0,T;L^s(\dom))$ if $q\le d$ and
in $L^p(0,T;C^0(\overline{\dom}))$ if $q>d$. If $p=\infty$, the space
$L^p(0,T;\cdot)$ is replaced by $C^0([0,T];\cdot)$.
\end{lemma}

\begin{proof}
By Theorem 3 and Lemma 5 of \cite{Sim87}, the set
\begin{align*}
  B_R &= \big\{u_n\in L^p(0,T;W^{1,q}(\dom))\cap W^{\alpha,r}(0,T;Y): \\
	&\phantom{xxi}\|u_n\|_{L^p(0,T;W^{1,q}(\dom))}\le R\mbox{ and }
	\|u_n\|_{W^{\alpha,r}(0,T;Y)}\le R\big\}
\end{align*}
is relatively compact in $L^p(0,T;L^s(\dom))$. We deduce from Chebyshev's inequality
that
\begin{align*}
  \Prob(B_R^c) &\le \Prob(\|u_n\|_{L^p(0,T;W^{1,q}(\dom))}>R)
	+ \Prob(\|u_n\|_{W^{\alpha,r}(0,T;Y)}>R) \\
	&\le \frac{1}{R}\big(\E\|u_n\|_{L^p(0,T;W^{1,q}(\dom))}
	+ \E\|u_n\|_{W^{\alpha,r}(0,T;Y)}\big) \le \frac{C}{R}.
\end{align*}
The definition of tightness finishes the proof.
\end{proof}

%The following theorem is proved in \cite[Theorem 6.1.21]{DrMi13}.
%
%\begin{theorem}[Browder]\label{thm.browder}
%Let $X$ be a reflexive Banach space, $S:X\to X'$ be bounded, demicontinuous,
%coercive, monotone, $f\in X'$. Then there exists $u\in X$ such that $S(u)=f$.
%If $S$ is strictly monotone, this solution is unique. 
%\end{theorem}
%
%For the solution of the approximative stochastic system, we use the following result
%proved in \cite[Theorem 4.2.4, Prop.~4.1.4]{LiRo15}.
%
%\begin{theorem}[Existence for stochastic equations]\label{thm.sde}
%Let $H$ be a separable Hilbert space and 
%$M:H\to H$ be Fr\'echet differentiable such that for some $c>0$ and all $v\in H$, 
%the operator $\DD M(v)-cI$ is negative definite and 
%$\|M(v)\|_H\le c(1+\|v\|_H)$. Furthermore, let $u^0\in L^2(\Omega;H)$ be
%$\F_0$-measurable, $\sigma:\mathcal{L}_2(U,H)\to H$ is Lipschitz continuous, 
%and $W(t)$ be a cylindrical $Q$-Wiener process
%with $Q=I$ taking values in another Hilbert space and being defined
%on a complete probability space $(\Omega,\F,\Prob)$ with normal
%filtration $(\F_t)_{t\in[0,T]}$. Then there exists a unique progressively
%measurable solution $u$ to
%$$
%  \dd u(t) = M(u(t))\dd t + \sigma(u(t))\dd W(t), \quad t>0, \quad u(0)=u^0,
%$$
%in the sense that $u\in L^2(\Omega\times[0,T];H)$ and for $t\in[0,T]$,
%$$
%  u(t) = u(0) + \int_0^t M(u(s))\dd s + \int_0^t\sigma(u(s))\dd W(s) \quad
%	\Prob-a.s.
%$$
%\end{theorem}

\end{appendix}

%%%%%%%%%%%%%%%%%%%%%%%%%%%%%%%%%%%%%%%%%%%%%%%%%%%%%%%%%%%%%%%%%%%%%%%%%%%%%%


\begin{thebibliography}{11}
\bibitem{Ama89} H.~Amann. Dynamic theory of quasilinear parabolic systems. III. Global 
existence. {\em Math. Z.} 202 (1989), 219--250.

\bibitem{BMM21} V.~Bansaye, A.~Moussa, and F.~Mu\~{n}oz-Hern\'andez. Stability of a 
cross-diffusion system and approximation by repulsive random walks: a duality approach.
Submitted for publication, 2021. arXiv:2109.07146.

\bibitem{BCJ20} M.~Braukhoff, X.~Chen, and A.~J\"ungel. Corrigendum: 
Cross diffusion preventing blow up in the two-dimensional Keller-Segel model. 
{\em SIAM J. Math. Anal.} 52 (2020), 2198--2200. 

\bibitem{BHM13} Z.~Brze\'zniak, E.~Hausenblas, and E.~Motyl. Uniqueness in law of 
the stochastic convolution process driven by L\'evy noise. 
{\em Electron. J. Probab.} 18 (2013), 1--15.

\bibitem{BrMo14} Z.~Brze\'zniak and E.~Motyl. The existence of martingale solutions 
to the stochastic Boussinesq equations. {\em Global Stoch. Anal.} 1 (2014), 175--216.

\bibitem{BrOn10} Z.~Brze\'zniak and M.~Ondrej\'at. Stochastic wave equations with 
values in Riemanninan manifolds. {\em Stochastic Partial Differential Equations and 
Applications, Quad. Mat.} 25 (2010), 65--97.

\bibitem{BrOn13} Z.~Brze\'zniak and M.~Ondrej\'at. Stochastic geometric wave equations
with values in compact Riemannian homogeneous spaces. {\em Ann. Prob.} 41 (2013),
1938--1977.

\bibitem{CDHJ21} L.~Chen, E.~Daus, A.~Holzinger, and A.~J\"ungel. Rigorous derivation 
of population cross-diffusion systems from moderately interacting particle systems. 
{\em J. Nonlin. Sci.} 31 (2021), no. 94, 38 pages.

\bibitem{ChJu04} L.~Chen and A.~J\"ungel. Analysis of a multi-dimensional parabolic 
population model with strong cross-diffusion. 
{\em SIAM J. Math. Anal}. 36 (2004), 301--322.

\bibitem{ChJu06} L.~Chen and A.~J\"ungel. Analysis of a parabolic cross-diffusion 
population model without self-diffusion. {\em J. Diff. Eqs.} 224 (2006), 39--59.

\bibitem{CDJ18} X.~Chen, E.~Daus, and A.~J\"ungel. Global existence analysis of 
cross-diffusion population systems for multiple species. 
{\em Arch. Ration. Mech. Anal.} 227 (2018), 715--747.

\bibitem{DaZa14} G.~Da Prato and J.~Zabczyk. {\em Stochastic Equations in Infinite
Dimensions}. Second edition. Cambridge University Press, Cambridge, 2014.

\bibitem{DDD19} E.~Daus, L.~Desvillettes, and H.~Dietert. About the entropic structure 
of detailed balanced multi-species cross-diffusion equations.  
{\em J. Diff. Eqs.} 266 (2019), 3861--3882.

\bibitem{DGT11} A.~Debussche, N.~Glatt-Holtz, and R.~Temam. Local martingale and 
pathwise solutions for an abstract fluids model. 
{\em Physica D} 240 (2011), 1123--1144.

\bibitem{DLM14} L.~Desvillettes, T.~Lepoutre, and A.~Moussa. Entropy, duality, and 
cross diffusion. {\em SIAM J. Math. Anal.} 46 (2014), 820--853.

\bibitem{DLMT15} L.~Desvillettes, T.~Lepoutre, A.~Moussa, and A.~Trescases.
On the entropic structure of reaction-cross diffusion systems.
{\em Commun. Partial Diff. Eqs.} 40 (2015), 1705--1747.

\bibitem{Deu87} P.~Deuring. An initial-boundary value problem for a certain 
density-dependent diffusion system. {\em Math. Z.} 194 (1987), 375--396.

\bibitem{DJZ19} G.~Dhariwal, A.~J\"ungel, and N.~Zamponi. Global martingale solutions 
for a stochastic population cross-diffusion system. 
{\em Stoch. Process. Appl.} 129 (2019), 3792--3820.

\bibitem{DiMo21} H.~Dietert and A.~Moussa. Persisting entropy structure for nonlocal 
cross-diffusion systems. Submitted for publication, 2021. arXiv:2101.02893.

\bibitem{DrMi13} P.~Dr\'abek and J.~Milota. {\em Methods of Nonlinear Analysis: 
Applications to Differential Equations}. 2dn edition. Springer, Basel, 2013.

\bibitem{Dre08} M.~Dreher. Analysis of a population model with strong cross-diffusion 
in unbounded domains. {\em Proc. Roy. Soc. Edinb. Sec. A} 138 (2008), 769--786.

\bibitem{GGJ03} G.~Galiano, M.~Gar{\'z}on, and A.~J\"ungel. Semi-discretization in time 
and numerical convergence of solutions of a nonlinear cross-diffusion population model. 
{\em Numer. Math.} 93 (2003), 655--673.

\bibitem{Gri02} D.~Grieser. Uniform bounds for eigenfunctions of the Laplacian on
manifolds with boundary. {\em Commun. Partial Diff. Eqs.} 27 (2002), 1283--1299.

\bibitem{Jue15} A.~J\"ungel. The boundedness-by-entropy method for cross-diffusion
systems. {\em Nonlinearity} 28 (2015), 1963--2001.

\bibitem{Jue16} A.~J\"ungel. {\em Entropy Methods for Diffusive Partial Differential
Equations}. Springer Briefs Math., Springer, 2016.

\bibitem{Kim84} J.~Kim. Smooth solutions to a quasi-linear system of diffusion 
equations for a certain population model. {\em Nonlin. Anal.} 8 (1984), 1121--1144.

\bibitem{KPS82} S.~Krein, J.~I.~Petunin, and E.~Semenov. {\em Interpolation of 
Linear Operators}. Amer. Math. Soc., Providence, 1982.

\bibitem{Kro92} P.~Kr\"oger. Upper bounds for the Neumann eigenvalues on a bounded
domain in Euclidean spaces. {\em J. Funct. Anal.} 106 (1992), 353--357.

\bibitem{Kry13} N.~Krylov. A relatively short proof of It\^o's formula for SPDEs 
and its applications. {\em Stoch. Partial Diff. Eqs.: Anal. Comput.}
1 (2013), 152--174.

\bibitem{KuNe20} C.~Kuehn and A.~Neam\c tu. Pathwise mild solutions for quasilinear 
stochastic partial differential equations. {\em J. Diff. Eqs.} 269 (2020), 2185--2227.

\bibitem{LeMo17} T.~Lepoutre and A.~Moussa. Entropic structure and duality for multiple 
species cross-diffusion systems. {\em Nonlin. Anal.} 159 (2017), 298--315.

\bibitem{LiRo15} W.~Liu and M.~R\"ockner. {\em Stochastic Partial Differential 
Equations: an Introduction}. Springer, Cham, 2015.

\bibitem{LNW98} Y.~Lou, W.-M.~Ni, and Y.~Wu. On the global existence of a 
cross-diffusion system. {\em Discrete Contin. Dyn. Syst.} 4 (1998), 193--203.

\bibitem{Met88} M.~M\'etivier. {\em Stochastic Partial Differential Equations in 
Infinite Dimensional Spaces}. Scuola Normale Superiore, Pisa, 1988.

\bibitem{Mou20} A.~Moussa. From nonlocal to classical Shigesada--Kawasaki--Teramoto 
systems: triangular case with bounded coefficients. 
{\em SIAM J. Math. Anal.} 52 (2020), 42--64.

\bibitem{PrRo07} C.~Pr\'ev\^ot and M.~R\"ockner. {\em A Concise Course on
Stochastic Partial Differential Equations}. Lecture Notes Math. 1905.
Springer, Berlin, 2007.

\bibitem{SKT79} N.~Shigesada, K.~Kawasaki, and E.~Teramoto. Spatial segregation 
of interacting species. {\em J. Theor. Biol.} 79 (1979), 83--99.

\bibitem{Sim87} J.~Simon. Compact sets in the space $L^p(0,T;B)$. 
{\em Ann. Math. Pura. Appl.} 146 (1987), 65--96.
\end{thebibliography}
\end{document}